%% file: Hierarchy_Paper.tex
\documentclass[microtype]{gtpart}
\usepackage{hyperref}
\usepackage{transparent}
\usepackage{amsmath}
\usepackage{amscd}
\usepackage{xcolor}
\usepackage{mathabx}
\usepackage{graphicx}
\usepackage{thm-restate}

\author[E. Einstein]{Eduard Einstein}
\givenname{Eduard}
\surname{Einstein}
\address{Department of Mathematics and Statistics,
Swarthmore College,
500 College Ave,
Swarthmore, PA 19081}
\email{eeinste1@swarthmore.edu}
\urladdr{https://einstein.domains.swarthmore.edu/}

\keyword{cube complexes}
\keyword{hierarchies}
\keyword{virtually special groups}
\keyword{relatively hyperbolic \CAT$(0)$ groups}

\subject{primary}{msc2010}{20F65}
\subject{secondary}{msc2010}{20F67}

\arxivreference{1903.12284}

\definecolor{red}{rgb}{1,0,0} 

\renewcommand{\tilde}{\widetilde}
\newcommand{\mathsym}[1]{{}}

\newcommand{\reals}{\mathbb{R}}

\newcommand{\naturals}{\mathbb{N}}

\newcommand{\integers}{\mathbb{Z}}

\newcommand{\diam}{\operatorname{diam}}

\newcommand{\zmodnz}[1]{\integers / #1 \integers}
\newcommand{\inv}{^{-1}}

\newcommand{\cyc}[1]{\langle #1 \rangle}

\renewcommand{\bar}{\overline}

\newcommand{\mc}[1]{\mathcal{#1}}
\newcommand{\CAT} {\ensuremath{\operatorname{CAT}}}

\newcounter{probnum}
\newcounter{ownThm}
\newtheorem{theorem}{Theorem}[section]
\newtheorem{definition}[theorem]{Definition}

\newtheorem{proposition}[theorem]{Proposition}
\newtheorem{cor}[theorem]{Corollary}
\newtheorem{lemma}[theorem]{Lemma}
\newtheorem{example}[theorem]{Example}

\newtheorem{observation}[theorem]{Observation}

\newtheorem{hypotheses}[theorem]{Hypotheses}
\newtheorem{notation}[theorem]{Notation}

\begin{document}

\title{Hierarchies for Relatively Hyperbolic Virtually Special Groups}

\begin{abstract}
Wise's Quasiconvex Hierarchy Theorem classifying hyperbolic virtually compact special groups in terms of quasiconvex hierarchies played an essential role in Agol's proof of the Virtual Haken Conjecture. Answering a question of Wise, we construct a new virtual quasiconvex hierarchy for relatively hyperbolic virtually compact special groups. We use this hierarchy to prove a generalization of Wise's Malnormal Special Quotient Theorem for relatively hyperbolic virtually compact special groups with arbitrary peripheral subgroups. 
\end{abstract}

\maketitle

\section{Introduction}

\subsection{Background, History and Motivation}
One of the main goals of cube complex theory is to use the geometry and combinatorial structure of cube complexes to better understand groups. 
The study of cubical groups has played an important role in recent developments in the theory of hyperbolic $3$--manifold groups, particularly in Agol's proof of the Virtual Haken Conjecture \cite{AgolVirtualHaken}. 

\textbf{Virtually special cube complexes}, developed by Wise and his collaborators, are central to the theory of cubical groups. A group is called \textbf{compact virtually special} if it is the fundamental group of a compact virtually special cube complex whose hyperplanes satisfy certain combinatorial conditions. Virtually special cube complexes have desirable separability properties that allow certain immersions to be promoted to embeddings using Scott's Criterion \cite{Scott78}.

A construction in \cite{Sageev95} due to Sageev provides a method for constructing a group action on a \CAT$(0)$ cube complex using ``\textbf{codimension--$1$ subgroups};'' however, in general, this action may not be proper, cocompact, or have a virtually special quotient. 
For hyperbolic groups, the situation is much clearer: Bergeron and Wise \cite{BergeronWise} proved that hyperbolic groups with an ample supply of quasiconvex codimension--$1$ subgroups have a proper and cocompact action on a \CAT$(0)$ cube complex. 
The key to Agol's proof of the Virtual Haken Conjecture is  that any geometric action of a hyperbolic group on a \CAT$(0)$ cube complex has virtually special quotient {\cite[Theorem 1.1]{AgolVirtualHaken}}. In the case of closed $3$--manifolds, the ample supply of codimension--$1$ subgroups comes from immersed surfaces constructed by Kahn and Markovic in \cite{KahnMarkovic}.

Two key ingredients in Agol's theorem are Wise's Quasiconvex Hierarchy Theorem and Malnormal Special Quotient theorem (MSQT).
Wise's Quasiconvex Hierarchy Theorem~{\cite[Theorem 13.3]{WiseManuscript}} characterizes the virtually special hyperbolic groups in terms of virtual quasiconvex hierarchies. 
\begin{definition}[{\cite[Definition 11.5]{WiseManuscript}}]
Let $\mathcal{QVH}$ be the smallest class of hyperbolic groups closed under the following operations.
\begin{enumerate}
\item $\{1\} \in \mathcal{QVH}$.
\item If $G = A *_C B$ and $A,B\in \mathcal{QVH}$ and $C$ is finitely generated and quasi-isometrically embedded in $G$ then $G\in\mathcal{QVH}$.
\item If $G = A_{*C}$, $A\in \mathcal{QVH}$ and $C$ is finitely generated and quasi-isometrically embedded in $G$, then $G\in \mathcal{QVH}$. 
\item If $H\le G$ with $|G:H|<\infty$ and $H\in\mathcal{QVH}$, then $G\in\mathcal{QVH}$. 
\end{enumerate}
\end{definition}
In other words, groups in $\mathcal{QVH}$ are hyperbolic groups that can be built from the trivial group by taking finite index subgroups or taking amalgamations and HNN extensions over quasiconvex subgroups. 
\begin{theorem}[{\cite[Theorem 13.3]{WiseManuscript}}, Wise's Quasiconvex Hierarchy Theorem]\label{Thm: Wise main thm}
Let $G$ be a hyperbolic group. Then $G\in \mathcal{QVH}$ if and only if $G$ is virtually compact special. 
\end{theorem}
As Wise notes in {\cite[Section 12]{WiseManuscript}}, the MSQT is an essential ingredient in the proof of the Quasiconvex Hierarchy Theorem.
\begin{theorem}[Wise's Malnormal Special Quotient Theorem {\cite[Theorem 12.2]{WiseManuscript}}]\label{MSQT}
Let $G$ be a hyperbolic and virtually special group with $G$ hyperbolic relative to a collection of subgroups $\{P_1,\ldots,P_m\}$. Then there exist finite index subgroups $\dot{P_i}\le P_i$ such that if $\bar{G} = G(N_1,\ldots,N_m)$ is any peripherally finite Dehn filling with $N_i\le \dot{P_i}$, then $\bar{G}$ is hyperbolic and virtually special. 
\end{theorem}
The MSQT together with virtually special amalgamation criteria from \cite{HW2012} and \cite{HW2015} are used to prove Theorem~\ref{Thm: Wise main thm}. 

For relatively hyperbolic groups, much less is known. Wise's methods from \cite{WiseManuscript} extend to more general situations than hyperbolic groups. In particular, many of the methods for hyperbolic groups extend to finite volume hyperbolic $3$--manifolds.  Hsu and Wise \cite{HW2015} also proved a special combination result for relatively hyperbolic groups albeit with much more restrictive hypotheses.

The main goal of this paper is to prove relatively hyperbolic analogs of important ingredients in the proof of Theorem~\ref{Thm: Wise main thm}. The first result answers a question posed by Wise:
\begin{restatable}{ownTheorem}{mainthm}\label{Thm: main thm}
Let $(G,\mathcal{P})$ be a relatively hyperbolic group pair and let $G$ be a virtually compact special group. 
Then there exists a finite index subgroup $G_0\le G$ and an induced relatively hyperbolic group pair $(G_0,\mathcal{P}_0)$ so that $G_0$ has a quasiconvex, malnormal and fully $\mathcal{P}_0$--elliptic hierarchy terminating in groups isomorphic to elements of $\mathcal{P}_0$. 
\end{restatable}
Proving that the hierarchy is not only quasiconvex and \textbf{malnormal} but also \textbf{fully $\mathcal{P}_0$--elliptic} is a way of ensuring that the hierarchy is compatible with the relatively hyperbolic structure on $G$ and allows for the use of relatively hyperbolic Dehn filling arguments. See Sections~\ref{Sec: hierarchy section} and \ref{P ell H} for definitions of quasiconvex, malnormal and fully $\mc{P}_0$--elliptic hierarchies.

Theorem~\ref{Thm: main thm} will be used to prove a relatively hyperbolic generalization of the MSQT using relatively hyperbolic Dehn filling techniques similar to those used in \cite{AGM}:
\begin{restatable}{ownTheorem}{relhypmsqt}\label{MSQT rel hyp}
Let $(G,\mathcal{P})$ be a relatively hyperbolic group pair with $\mathcal{P} = \{P_1,\ldots,P_m\}$. If $G$ is virtually compact special, then there exist subgroups $\{\dot{P}_i\lhd P_i\}$ where $\dot{P}_i$ is finite index in $P_i$ such that if $\bar{G} = G(N_1,\ldots,N_m)$ is any peripherally finite filling with $N_i\lhd \dot{P}_i$, then $\bar{G}$ is hyperbolic and virtually special.  
\end{restatable}
Peripherally finite fillings are defined formally in Definition~\ref{Def: gp Dehn filling}. 
While Wise proved a generalized relatively hyperbolic version of the MSQT in {\cite[Theorem 15.6]{WiseManuscript}} for relatively hyperbolic groups with virtually abelian peripherals, Theorem~\ref{MSQT rel hyp} holds for arbitrary peripheral subgroups. 

\subsection{Outline}

Section~\ref{NPC complexes section} contains a brief overview of the geometry of relatively hyperbolic groups. 
Section~\ref{Sec: Graphs and Hierarchies} covers preliminaries about graphs of groups and quasiconvex hierarchies. 

Section~\ref{RelFelTrv} is devoted to proving a relative fellow traveling result for a \CAT$(0)$ spaces with a geometric action by a relatively hyperbolic group, a generalized version of quasigeodesic stability in hyperbolic spaces. The main result is Theorem~\ref{Prop: rel fellow traveling}. Similar results were proved by Hruska \cite{Hruska04NPC2} and Hruska-Kleiner in \cite{HruskaKleiner} for \CAT$(0)$ spaces with isolated flats, and this result was previously known to experts in the field. However, it was difficult to find an exact formulation of Theorem~\ref{Prop: rel fellow traveling} in the literature, so a proof is produced here.

Section~\ref{Combo Lemma Section} contains a combination lemma for certain subspaces of \CAT$(0)$ spaces with a geometric action by a relatively hyperbolic group. The main result, Theorem~\ref{quasiconvexity thm} shows that subspaces of such a \CAT$(0)$ space that are unions of convex cores for peripheral coset orbits and convex subspaces that obey a separation property are quasiconvex. The proof technique is inspired partly by the proof of the combination lemma in \cite{HW2015}.

Section~\ref{Sec: Special} reviews the properties of special cube complexes. In particular, Section~\ref{Sec: Elevations} will introduce separability and explain how to pass to a finite cover so that each hyperplane's elevations to the universal cover obey a separation property.  Section~\ref{Sec: Convex Cores} recalls a result of Sageev and Wise \cite{SW2015} used to represent peripheral subgroups of a relatively hyperbolic compact special group $G$ as immersed complexes in a NPC cube complex $X$ with $\pi_1X=G$. 

Section~\ref{Hierarchy section} follows the outline of  \cite[Section~5]{AGM} and uses Wise's double dot hierarchy construction to prove Theorem~\ref{Thm: main thm}. While the general strategy is the same, the hyperbolic geometry used in \cite{AGM} to prove the edge groups of the hierarchy are $\pi_1$-injective and quasi-isometrically embedded needs to be replaced by relatively hyperbolic geometric results from the preceding sections. 

Section~\ref{MSQT section} uses Theorem~\ref{Thm: main thm} along with a relatively hyperbolic Dehn filling argument similar to the one used in a new proof of Wise's MSQT from \cite{AGM} to prove Theorem~\ref{MSQT rel hyp}, a relatively hyperbolic analog of Wise's MSQT.

\subsection{Acknowledgements}

The author would like to thank Jason Manning and Daniel Groves for their invaluable guidance and suggestions. Specifically, the author would like to thank Groves for explaining the proof of Proposition~\ref{P: frqc almost malnormal}.  The author also thanks Lucien Clavier, Yen Duong, Chris Hruska, Michael Hull and Daniel Wise for useful conversations that helped shape this work. Finally, the author thanks the anonymous referee for careful reading and suggestions that greatly helped improve the structure of this paper. 



\section{Relatively Hyperbolic Geometry}\label{NPC complexes section}

\subsection{The geometry of \CAT$(0)$ spaces being acted on by relatively hyperbolic groups}
\label{relhypgeom section}

In the situation where a relatively hyperbolic group acts properly and cocompactly on a \CAT$(0)$ space, it is reasonable to hope to partially recover the geometric features of a hyperbolic space. There are many equivalent definitions of a relatively hyperbolic group, see \cite{Hruska2010} for several examples; one definition, originally due to Farb \cite{FarbRelHypGroups}, is produced here:
\begin{definition}[\cite{Hruska2010} Definition 3.6]\label{Def: relhyp}
Let $G$ be finitely generated relative to $\mathcal{P}$ with each $P\in \mathcal{P}$ finitely generated. The pair $(G,\mathcal{P})$ is a \textbf{relatively hyperbolic group pair} if for some finite relative generating set $S$, the coned-off Cayley graph $\hat\Gamma(G,\mathcal{P},S)$ is hyperbolic and $(G,\mathcal{P},S)$ has Farb's bounded coset penetration property (see \cite[Section 3.3]{FarbRelHypGroups}). 

The elements of $\mathcal{P}$ and their conjugates are called \textbf{peripheral subgroups} and the cosets $\{gP:~g\in~G,~P\in~\mathcal{P}\}$ are called peripheral cosets.
\end{definition}

Definition~\ref{Def: relhyp} establishes useful notation to refer to a relatively hyperbolic group pair, but the technical details will be less useful. 
Instead, most of the arguments involving relatively hyperbolic groups will be made using two key properties: that coarse intersections of peripheral cosets are uniformly bounded and that triangles are \textbf{relatively thin} in a sense defined in Section~\ref{relthintrisec}.  
  
The following fact is well known:
\begin{proposition}\label{Prop: coset separation}
Let $(G,\mathcal{P})$ be a relatively hyperbolic group pair. 
Let $S$ be a finite generating set for $G$. 
For all $R\ge 0$, there exists $M_R\ge 0$ such if $gP,\,g'P'$ is a pair of distinct peripheral cosets, then $\diam \mathcal{N}_R(gP)\cap \mathcal{N}_R(g'P') \le M_R$ in the word metric on $\Gamma(G,S)$.  
\end{proposition}

The uniform bounds on coarse intersections of peripheral cosets transfers nicely to the case where a relatively hyperbolic group acts properly and cocompactly on a geodesic space by isometries:
\begin{cor}\label{peripheral separation}
Let $G$ be a finitely generated group acting properly and cocompactly by isometries on a geodesic metric space $X$, and let $x\in X$ be a base point. If $(G,\mc{P})$ is a relatively hyperbolic group pair, then for all $R\ge 0$, there exists $M_{R,X,x}\ge 0$ such that if $P,P'\in \mathcal{P}$, $g,g'\in G$ with $gP\ne g'P'$, then $\diam (\mathcal{N}_R(gPx) \cap \mathcal{N}_R(g'P'x))\le M_{R,X,x}$.   
\end{cor}


\subsection{Relatively Thin Triangles}\label{relthintrisec}

Comparison tripods help compare geodesic triangles in $X$ with tripods: 
\begin{definition}
Let $a,b,c\in X$ and let $\triangle abc$ be a geodesic triangle. There exists a map $h:\triangle abc \to T(a,b,c)$ where $T(a,b,c)$ is a unique tripod (up to isometry) with center point $x$ such that $h$ is isometric on each side of the triangle and the three legs of the tripod are $[h(a),x],\,[h(b),x]$ and $[h(c),x]$. The tripod $T(a,b,c)$ is called a \textbf{comparison tripod} for $\triangle abc$. The map $h$ is the \textbf{comparison map}. 
\end{definition}

A geodesic metric space $X$ is \textbf{hyperbolic} if there exists a $\delta>0$ so that for every geodesic triangle in $X$, the preimage of every point in the comparison map has diameter less than $\delta$. 

\begin{definition}\label{Def: rel thin triangle}
Let $X$ be a geodesic metric space, and let $F\subseteq X$ be a subset of $X$.

Let $\triangle abc$ be a geodesic triangle in $X$ and let $\delta>0$. Let $T(a,b,c)$ be the comparison tripod, and let $h:\triangle abc \to T(a,b,c)$ be the comparison map. If for all $p\in T(a,b,c)$
\begin{enumerate}
\item $\diam (h\inv(p))<\delta$ or
\item $h\inv(p)\subseteq \mathcal{N}_\delta(F)$, 
\end{enumerate}
then $\triangle abc$ is \textbf{$\delta$--thin relative to $F$}. 
\end{definition}

\begin{definition}
Let $X$ be a geodesic metric space, $\delta\ge 0$ and let $\mathcal{B}$ be a collection of subspaces. The space $X$ has the $\delta$--\textbf{relatively thin triangle property relative to $\mathcal{B}$} if each geodesic triangle $\Delta$ is $\delta$--thin relative to some $F\in\mathcal{B}$. 
\end{definition}

\begin{figure}
\input{relThinTriangle2.pdf_tex}
\caption{An example of a triangle which is $\delta$--thin relative to some $F$ with its comparison tripod. Points in the blue part of the tripod have preimages in the triangle which lie in the blue shaded region. All other points have preimages in the triangle with diameter $\delta$ like the point $p$ whose preimages $x,y$ have $d(x,y)<\delta$. The fat part (see Definition~\ref{Def: corner segments}) of each side is the subsegment that intersects the blue shaded region.} 
\end{figure}
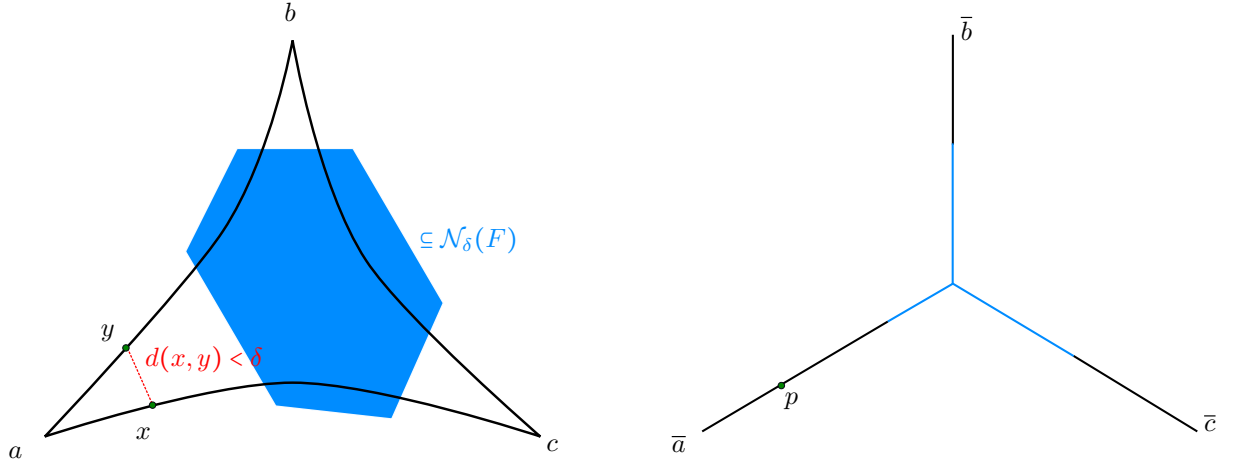

The space $X$ may contain triangles that are $\delta$--thin. By definition, these triangles are $\delta$--thin relative to every element of $\mathcal{B}$. 
In the applications, $X$ will usually be a \CAT$(0)$ space with a geometric action by a relatively hyperbolic group $G$ where the elements of $\mathcal{B}$ are convex subspaces of $X$ that lie in uniformly bounded neighborhoods of peripheral coset orbits.
If $(G,\mathcal{P})$ is a relatively hyperbolic group pair, a \CAT$(0)$ space with a geometric action by $G$ has the relatively thin triangle property relative to $\mathcal{B} = \{gPx\,|\,g\in G,\, P\in\mathcal{P}\}$:
\begin{proposition}[\cite{SW2015} Theorem 4.1, Proposition 4.2, see also Section 8.1.3 of \cite{DrutuSapir}] \label{Prop: thin triangles}
Let $(G,\mathcal{P})$ be a relatively hyperbolic group pair and let $G$ act properly and cocompactly on a \CAT$(0)$ space $X$ by isometries. Let $x\in X$ be a base point and set \[\mathcal{B} = \{gPx\,|\,g\in G,\, P\in\mathcal{P}\}.\]

Then for some $\delta>0$, $X$ has the $\delta$--relatively thin triangle property relative to $\cal{B}$.   
\end{proposition}
Note that when $X$ has the relatively thin triangle  property relative to $\mathcal{B}$, $R\ge 0$ and $\mathcal{B}' =\{\mathcal{N}_R(F):\, F\in \mathcal{B}\}$, then $X$ still has the relatively thin triangle property relative to $\mathcal{B}'$. 

The notion of fellow-traveling will be useful for describing behavior of geodesics that issue from the same point. Definitions of fellow-traveling may vary, so the one that will be used is recorded here:
\begin{definition}
Let $\alpha:[a_1,a_2]\to X$ and $\beta:[b_1,b_2]\to X$ be geodesics, and let $k\ge 0$. The geodesics $\alpha$ and $\beta$ \textbf{$k$--fellow travel for distance $D$} if $d(\alpha(a_1+t),\beta(b_1+t))\le k$ for all $0\le t\le D$.
If $x:= \alpha(a_1) = \beta(b_1)$ and $\alpha$ and $\beta$ $k$--fellow travel for distance $D$, then \textbf{$\alpha$ and $\beta$ $k$--fellow travel distance $D$ from $x$.} 
\end{definition}

We also introduce \text{tails} of a geodesic to help us make geometric arguments:
\begin{definition}\label{D: tail}
Let $\gamma$ be a geodesic in $\tilde{X}$, let $p$ be an endpoint of $\gamma$, and let $k\ge 0$. The \textbf{$k$--tail of $\gamma$ at $p$} is the geodesic subsegment of $T$ consisting of all $x\in\gamma$ so that $d(x,p)\le k$. 
\end{definition}

\begin{definition}\label{Def: corner segments}
Let $X$ be a \CAT$(0)$ geodesic metric space with triangles that are $\delta$--thin relative to $\mathcal{B}$. Let $\triangle\subseteq X$ with vertices $a,b,c$ with comparison map $h:\triangle abc\to T(a,b,c)$.
Let $L_a$ be the closure of the leg of the tripod $T(a,b,c)$ that contains $h(a)$.
Let $\textbf{Thin}_a := \{x\in h\inv(L_a):\,\diam(h\inv(h(x))<\delta\}$. 
The \textbf{corner segments} of $\triangle$ at $a$ are the two closures of the parts of $\textbf{Thin}_a$ in each side and the \textbf{corner length} is the length of a corner segment at $a$. 

The \textbf{fat part} of the side $ab\subseteq \triangle$ in $\triangle$ is 
$ab\setminus (\textbf{Thin}_a \cup \textbf{Thin}_b)$.
\end{definition}

The corner segments at $a$ are subsegments of the sides issuing from $a$ that $\delta$--fellow travel. Each of these segments have the same length, which is defined to be the corner length. 
If $\triangle$ is $\delta$--thin relative to $B_\triangle\in\mathcal{B}$, the fat part of each side of $\triangle$ is the maximal subsegment that does not lie in any of the corner segments and hence lies in $\mathcal{N}_\delta(B_\triangle)$. 
Note that the fat part of a side may be empty.
Since $X$ is \CAT$(0)$, each corner segment or fat part of a side is connected.

A $(\lambda,\epsilon)$--\textbf{quasigeodesic} in $X$ is a $(\lambda,\epsilon)$--quasi-isometric embedding of a (possibly unbounded) interval in the real line in $X$, see \cite[Definition I.8.22]{BridsonHaefliger} for details.

Quasigeodesic triangles in the Cayley graph of a relatively hyperbolic group also satisfy a thinness condition which is used to obtain Proposition~\ref{Prop: thin triangles}:
\begin{theorem}[\cite{SW2015} Theorem 4.1, originally due to \cite{DrutuSapir}]\label{qgd triangle}
Let $(G,\mathcal{P})$ be a relatively hyperbolic group pair with Cayley graph $\Gamma$. For all $\lambda\ge 1,\,\epsilon>0$ there exists a $\delta>0$ such that if $\triangle$ is a $(\lambda,\epsilon)$--quasigeodesic triangle in $\Gamma$ with sides $c_0,c_1,c_2$, either:
\begin{enumerate}
\item there exists a point $p$ that lies within $\frac\delta2$ of each side or
\item there is a peripheral coset $gP$ so that each side $c_i$ of $\triangle$ has a subpath $c_i'$ where $c_i'\subseteq \mathcal{N}_{\delta} (gP)$ and the terminal endpoint of $c_i'$ and the initial point of $c_{i+1}'$ (indices mod $3$) are within distance $\delta$ of each other. 
\end{enumerate}
\end{theorem}

Lemma~\ref{Transferrence lemma} is simple but is instrumental for working with relatively thin triangles.
\begin{lemma}\label{Transferrence lemma}
Let $\tilde{X}$ be a \CAT$(0)$ space. Let $\Delta abc$ be a geodesic triangle in $\tilde{X}$ that is $\delta$--thin relative to $F$.
 Let $ab,bc,ac$ denote the sides of $\Delta abc$. 
If the length of the fat part of $ac$ in $\Delta abc$ is bounded above by $k_{fat}\ge 0$, then the length of the fat part of $bc$ and the length of the fat part of $ab$ differ by at most $k_{fat}+3\delta$. 
\end{lemma}

The proof involves four applications of the triangle inequality. See Figure~\ref{transferrence picture} for a schematic. 
With Lemma~\ref{Transferrence lemma}, a bound on the fat part of one side of a relatively thin triangle helps control the lengths of the fat parts of the other two sides. This technique will be used repeatedly, particularly in Section~\ref{Combo Lemma Section}. 

\begin{figure}
\includegraphics[scale=.6]{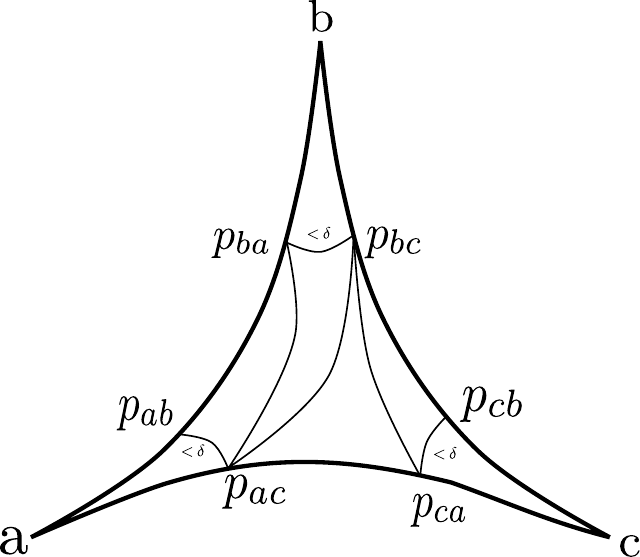}\caption{Applying the triangle inequality four times gives a bound on the difference between the length of $[p_{ab},p_{ba}]$ and the length of $[p_{bc},p_{cb}]$ in terms of $|[p_{ac},p_{ca}]|,\delta$.}\label{transferrence picture}
\end{figure}

Relatively hyperbolic groups interact nicely with passing to finite index subgroups:
\begin{proposition}[{\cite[Notation 2.9]{AGM}}]\label{Prop: induced peripheral structure}
Let $G$ be a group and let $\mathcal{P}$ be a finite collection of subgroups of $G$.
Let $H\lhd G$ be a finite index normal subgroup. For each $P\in \mathcal{P}$, let $\mathcal{E}_0(P) = \{gPg\inv\cap H\,|\, g\in G\}$ and let $\mathcal{E}(P)$ be a set of representatives of $H$-conjugacy classes in $\mathcal{E}_0(P)$.
Let $\mathcal{P}' = \bigsqcup_{P\in \mc{P}} \mathcal{E}(P)$.  

The pair $(G,\mathcal{P})$ is relatively hyperbolic if and only if $(H,\mathcal{P}')$ is relatively hyperbolic. 
\end{proposition}

There is also a generalized version of quasiconvexity for relatively hyperbolic groups. 
\begin{definition}[{\cite[Definition 6.10]{Hruska2010}}]
Let $(G,\mathcal{P})$ be a relatively hyperbolic group pair. 
Let $H\le G$. 
Let $S$ be any finite set such that $S\cup\mathcal{P}$ generates $G$. Suppose there exists $\kappa(S,d_S)$ such that for any $\hat{\Gamma}(G,\mathcal{P},S)$-geodesic $\gamma$ with endpoints in $H$, $\gamma\cap G$ lies in $\mathcal{N}_\kappa(H)$ with respect to $d_S$. Then $H$ is \textbf{relatively quasiconvex in $(G,\mathcal{P})$}. 
\end{definition}
Note that there are other equivalent definitions which are discussed in \cite{Hruska2010}. The definition is also independent of the choice of finite relative generating set (see \cite[Theorem 7.10]{Hruska2010}). Relative quasiconvexity will only be needed for the peripheral subgroups:

\begin{proposition}\label{peripheral qc}
Let $(G,\mathcal{P})$ be a relatively hyperbolic group pair. Then every element of $\mathcal{P}$ is relatively quasiconvex in $G$. 
\end{proposition}

\begin{proof}
In $\hat\Gamma(G,\mathcal{P},S)$ every $P\in\mathcal{P}$ has diameter $1$.
\end{proof}



\section{Graphs of Groups and Hierarchies} \label{Sec: Graphs and Hierarchies}
\subsection{Graphs of Groups}
A graph of groups (together with an isomorphism from the fundamental group) is a way of decomposing a group along a finite number of splittings and HNN extensions. Further decomposing the vertex groups as graphs of groups, decomposing the resulting vertex groups as a graph of groups again and continuing this process a finite number of times yields a kind of ``multilevel graph of groups'' called a \textbf{hierarchy} which will be defined in Definition~\ref{Def: hierarchy}. 

\begin{definition}\label{Def: graph of groups}
A \textbf{graph of groups} $(\Gamma,\chi)$ consists of the following data:
\begin{enumerate}
\item a connected finite graph $\Gamma = \Gamma(V,E)$ where $V$ is the vertex set of $\Gamma$ and $E$ is the oriented edge set of $\Gamma$ with an involution $e\mapsto\bar{e}$ that switches the orientation of each edge,
\item an \textbf{assignment map} $\chi:V\sqcup E\to \textbf{Grp}$ that assigns a group to each vertex and edge,
\item for all $e\in E$, $\chi(e) = \chi(\bar{e})$,
\item \textbf{attachment homomorphisms} $\psi_e:\chi(e)\to \chi(t(e))$ where $t(e)$ is the terminal vertex of the edge $e$.
\end{enumerate}

$\Gamma$ is a \textbf{faithful} graph of groups if the attachment homomorphisms $\psi_e$ are injective.
\end{definition}

A \textbf{graph of spaces} is constructed like a graph of groups, except that the assignment map $\chi$ assigns a (path connected) topological space instead of a group to each edge and vertex. The attachment homomorphisms are replaced by continuous \textbf{attachment maps}, and a \textbf{faithful graph of spaces} has $\pi_1$-injective attachment maps. 
A \textbf{graph of spaces realization of a space $X$} for a graph of spaces $(\Gamma,\chi)$ is a triple $(\Gamma,\chi,q)$ where $q$ is a homotopy equivalence from $X$ to the mapping cylinders of the attachment maps glued along vertex spaces.  

Some authors, for example Wise and Serre, take faithfulness to be a part of the definition of a graph of groups. 
Not requiring faithfulness makes it easier to define graphs of groups in terms of graphs of spaces. For the applications in Section~\ref{Hierarchy section}, graphs of groups will be constructed first without showing that they are faithful, but these graphs of groups will turn out to be faithful.

If $(\Gamma,\chi)$ is a graph of groups, and $T$ is a maximal tree in $\Gamma$, then $\pi_1(\Gamma,T)$ will denote the \textbf{fundamental group of the graph of groups $\Gamma$ with respect to the tree $T$.} See \cite{SerreTrees} for further details about graphs of groups.

A \textbf{graph of groups structure} is the group theoretic analogue of a graph of spaces realization:
\begin{definition}
Let $G$ be a group, let $(\Gamma,\chi)$ be a graph of groups where $T$ is a maximal tree and let $\phi:G\to \pi_1(\Gamma,T)$ be an isomorphism. The triple $(\Gamma,\phi,T)$ is a \textbf{graph of groups structure on $G$}.

The structure $(\Gamma,\phi,T)$ is \textbf{degenerate} if $\Gamma$ is a single vertex labeled with $G$ and $\phi$ is the identity.
\end{definition}
While a graph of groups structure determines a splitting of $G$, the choice of isomorphism and maximal tree affects the precise splitting. 
In many cases, it suffices to give a splitting of $G$ up to conjugacy which will be the case in the examples below.
When the splitting is given up to conjugacy, the choice of maximal tree also becomes unnecessary. 

\begin{example}
Figure~\ref{graph of spaces example} shows a graph of spaces decomposition of a genus $2$ surface and a graph of groups splitting of the fundamental group induced by the graph of spaces decomposition.
\end{example}

\begin{example}
If $\Sigma_g$ is a closed surface of genus $g$, then a pants decomposition of $\Sigma_g$ induces a splitting of $\pi_1\Sigma_g$ as a graph of groups where the vertex groups are isomorphic to a free group of rank $2$ and the edge groups are infinite cyclic groups. 
\end{example}

\begin{figure}
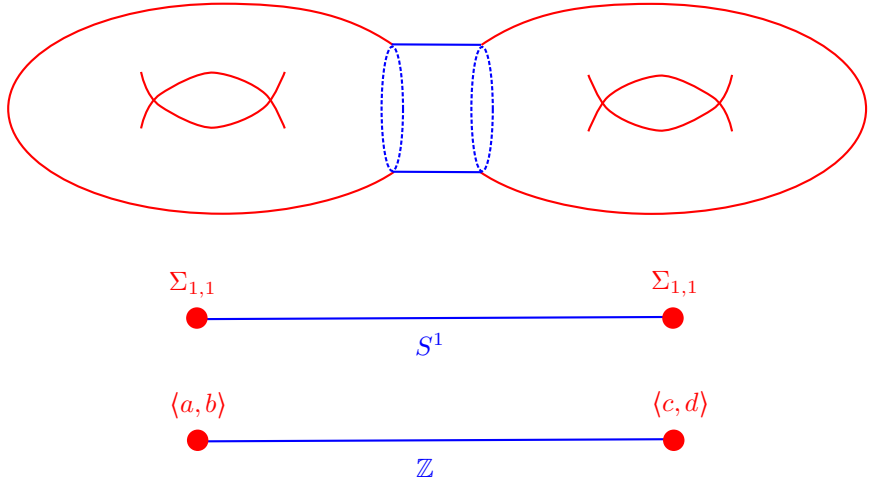\caption{A graph of spaces realization of a genus 2 surface where $\Sigma_{1,1}$ is a punctured torus, together with the corresponding graph of groups obtained by applying the $\pi_1$ functor.}\label{graph of spaces example}
\end{figure}

 Graph of groups structures interact naturally with finite index normal subgroups.
The following is \cite[Proposition 3.18]{AGM} but is originally due to Bass \cite{Bas93}.

\begin{proposition}\label{Prop: induced graph of groups}
Suppose $G$ has a graph of groups structure $(\Gamma,\phi,T)$, $H\lhd G$ and $H$ is finite index in $G$. Then $H$ has an induced graph of groups structure $(\tilde{\Gamma},\tilde\phi,T')$ so that:
\begin{enumerate}
\item Every vertex group of $(\tilde{\Gamma},T')$ has the form $(K^g\cap H)\lhd K^g$ and is finite index in $K^g$ for some vertex group $K$ of $(\Gamma,T)$ and some $g\in G$.
\item Every edge group of $(\tilde{\Gamma},T')$ has the form $(K^g\cap H)\lhd K^g$ and is finite index in $K^g$ for some edge group $K$ of $(\Gamma,T)$ and some $g\in G$.
\end{enumerate}
\end{proposition}


\subsection{Hierarchies}\label{Sec: hierarchy section}
Hierarchies of groups are  inductively defined multilevel graphs of groups: 
\begin{definition}\label{Def: hierarchy}
A \textbf{hierarchy of groups of length $0$} is a single vertex labeled by a group. 

A \textbf{hierarchy of groups of length $n$} is a graph of groups $(\Gamma_n,\chi_n)$ together with hierarchies of length $n-1$ on each vertex of $\Gamma_n$.

If $\mathcal{H}$ is a length $n$ hierarchy of groups, the \textbf{$n$th level} of $\mathcal{H}$ is the graph of groups $\Gamma_n$. For $1\le k\le n$, the $(n-k)$th level of $\mathcal{H}$ is the disjoint union of the $(n-k)$th levels of the hierarchies on the vertices of $\Gamma_n$.  

The \textbf{terminal groups} are the groups labeling the vertices at level $0$. 
\end{definition}

It will be useful to think of graphs of groups as length $1$ hierarchies. 
Realizing a group as a hierarchy is similar to finding a graph of groups structure for that group:

\begin{definition}\label{Def: Hierarchy structure}
Let $G$ be a group, $\mathcal{H}$ be a hierarchy of length $n$. Let $(\Gamma_n,\chi_n)$ be the level $n$ graph of groups. 
When $n=0$, a \textbf{hierarchy for $G$} is a single vertex labeled by $G$. If $n\ge 1$, a \textbf{hierarchy for $G$} is $\mathcal{H}$ together with a graph of groups structure $(\Gamma_n,\phi,T)$ for $G$ so that for every vertex $v$ of $\Gamma_n$, the hierarchy on length $n-1$ on $v$ is a hierarchy for the vertex group $\chi_n(v)$.  
Let $\mathcal{P}$ be a collection of subgroups of $G$. The hierarchy structure \textbf{terminates in $\mathcal{P}$} if every terminal group of $\mathcal{H}$ is conjugate to $\phi(\mathcal{P})$ for some $P\in \mathcal{P}$.
\end{definition}
It will often be convenient to forget the choice of maximal tree and only give a hierarchy structure for a group up to conjugacy. In general, hierarchies will be allowed to contain degenerate splittings, but in order to obtain non-trivial results, it will be necessary to ensure that at least one of the splittings in the hierarchy is non-degenerate.

Wise's hierarchies in \cite{WiseManuscript} permit only one-edge splittings rather than allowing a graph of groups splitting for each vertex group in the hierarchy. The hierarchies in Definition~\ref{Def: Hierarchy structure} can be converted to hierarchies with one-edge splittings for each vertex group at the expense of increasing the length of the hierarchy. 
Wise's hierarchies also terminate in the trivial group while Definition~\ref{Def: Hierarchy structure} allows arbitrary terminal groups. In practice, the goal in Section~\ref{Hierarchy section} will be to (virtually) find a hierarchy for a relatively hyperbolic group $(G,\mathcal{P})$ that terminates in groups isomorphic to those in the induced peripheral structure. Section~\ref{MSQT section} will explore what happens to the hierarchy after quotienting out finite index subgroups of the peripheral subgroups.

A \textbf{hierarchy of spaces} and a \textbf{hierarchy realization for a space $X$} can be defined analogously by replacing groups in Definition~\ref{Def: hierarchy} with topological spaces and replacing graph of groups structures by realizations in Definition~\ref{Def: Hierarchy structure}.  

Malnormality is an important group property which will play a role in Section~\ref{MSQT section} and is useful for amalgamating virtually special groups to make new virtually special groups (see \cite{HW2015}). 
\begin{definition}
Let $G$ be a group and let $H\le G$. The subgroup $H$ is \textbf{malnormal in $G$} if for all $g\in G\setminus H$, $g\inv Hg\cap H = \{1\}$. 
Similarly, $H$ is \textbf{almost malnormal in $G$} if for all $g\in G\setminus H$, $|g\inv Hg \cap H|<\infty$.

Malnormality also extends to collections of subgroups. Let $\mathcal{P}$ be a collection of subgroups of $G$. The collection $\mathcal{P}$ is (almost) malnormal in $G$ if for all $g\in G$ and $P,P'\in \mathcal{P}$ either $g\inv P g\cap P'$ is trivial (finite) or $P=P'$ and $g\in P$. 
\end{definition}
For example, if $(G,\mathcal{P})$ is a relatively hyperbolic group pair and $G$ is finitely generated, then the collection $\mathcal{P}$ is almost malnormal in $G$ by Proposition~\ref{Prop: coset separation}.

Definition \ref{Def: graph of groups} (graphs of groups) and Definition \ref{Def: hierarchy} (hierarchies) are very flexible, but in practice, some further restrictions will be needed to ensure that graphs of groups and hierarchies produce useful splittings:
\begin{definition}\label{Def: enhanced hierarchy}
Let $(\Gamma,\chi)$ be a faithful graph of groups and let $(\Gamma,\phi)$ be a graph of groups structure (up to conjugacy) for a group $G$.
\begin{enumerate}
\item $\Gamma$ is \textbf{quasiconvex} if every edge attachment map is a quasi-isometric embedding into $\pi_1(\Gamma)$.
\item $\Gamma$ is \textbf{(almost) malnormal} if for every $e\in E$, the image of the attachment homomorphism $\psi_e$ in $\pi_1(\Gamma)$ is (almost) malnormal in $\pi_1(\Gamma)$. 
\end{enumerate}
Let $\mathcal{H}$ be a hierarchy for $G$. 
\begin{enumerate}
\item $\mathcal{H}$ is \textbf{faithful} if every graph of groups at every level of $\mathcal{H}$ is faithful.
\item $\mathcal{H}$ is \textbf{quasiconvex} if every edge group of every graph of groups at every level of $\mathcal{H}$ quasi-isometrically embeds in $G$.
\item $\mathcal{H}$ is \textbf{(almost) malnormal} if every edge group of every graph of groups at every level of $\mathcal{H}$ is (almost) malnormal in $G$. 
\end{enumerate}
\end{definition}
It may be possible to give a reasonable weaker definition of quasiconvex (or malnormal) hierarchy by only requiring an edge group $G_e$ of a graph of groups $H$ in $\mathcal{H}$ to be quasi-isometrically embedded (malnormal) in each adjacent vertex group, but the stronger definition given here will be needed in Section~\ref{MSQT section}.

Here are some examples to help illustrate the definition of a hierarchy:


\begin{example}\label{genus 4 surface hierarchy}
A splitting of the fundamental group of a hyperbolic surface group can be realized along quasiconvex infinite cyclic subgroups by using a pants decomposition. The splitting can be achieved either as a sequence of $1$--edge splittings to create a hierarchy or can be achieved a single multi-edge graph of groups splitting.
\end{example}

There are iterated hierarchy splittings that cannot be realized by a single graph of groups splitting:

\begin{figure}[ht]
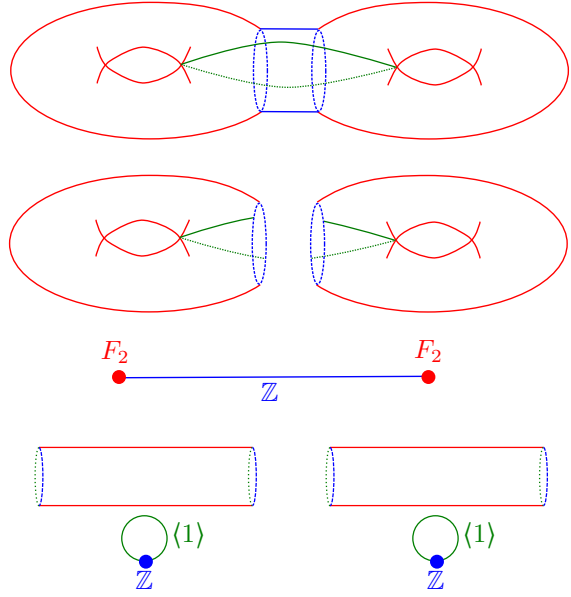
\caption{A hierarchy for $\pi_1(\Sigma_2)$, the fundamental group of a genus $2$ surface $\Sigma_2$, where the iterated splitting of $\pi_1(\Sigma_2)$ cannot be realized by a graph of groups. The first splitting is over the infinite cyclic subgroup of $\pi_1(\Sigma_2)$ corresponding to one of the blue copies of $S^1$. The resulting vertex spaces are punctured tori whose fundamental groups are rank $2$ free groups. Cutting along the green arc in each punctured torus makes an annulus. Then the fundamental group of a punctured torus splits as an HNN extension of the fundamental group of an annulus ($\Z$) over the trivial group (corresponding to the green arcs in each annulus which are glued together to make a punctured torus). }
\label{Fig: nontrivialhierarchy}
\end{figure}

\begin{example}
Figure~\ref{Fig: nontrivialhierarchy} shows a length $2$ hierarchy for the fundamental group of a genus $2$ surface, $\Sigma_2$. Cuts are made along the both the blue and green simple closed curves which intersect, so the iterated splitting of the fundamental group cannot be accomplished by a graph of groups (length $1$ hierarchy). 
\end{example}

Other notable examples of hierarchies are the Haken Hierarchy for Haken $3$--manifolds, see \cite[Section 9.4]{Martelli}, and the Magnus-Moldvanskii hierachy for one-relator groups, see \cite[Chapter 19]{WiseManuscript}. 

Proposition~\ref{Prop: induced graph of groups} extends to hierarchies by induction on the length of the hierarchy.
\begin{cor}
Suppose $G$ has a hierarchy $\mathcal{H}$ and $H$ is a finite index normal subgroup of $G$, then $\mathcal{H}$ has an induced hierarchy $\mathcal{H}'$ such that the length of $\mathcal{H}$ is the length of $\mathcal{H}'$ and:
\begin{enumerate}
\item every vertex group at level $i$ of the hierarchy $\mathcal{H}'$ is of the form $K^g\cap H$ which is finite index and normal in $K^g$ for some vertex group $K$ of $\mathcal{H}$ at level $i$ and some $g\in G$,
\item  every edge group at level $i$ of the hierarchy $\mathcal{H}'$ is of the form $K^g\cap H$ which is finite index and normal in $K^g$ for some edge group $K$ of $\mathcal{H}$ at level $i$ and some $g\in G$.
\end{enumerate}\label{findex forms}
\end{cor}

Lemma~\ref{findex quasiconvex} follows from Corollary~\ref{findex forms}:
\begin{lemma}\label{findex quasiconvex}
If $\mathcal{H}$ is a quasiconvex hierarchy for $G$ and $G_0$ is a finite index normal subgroup of $G$, then the induced hierarchy on $\mathcal{H}_0$ on $G_0$ is quasiconvex. 
\end{lemma}

The definition of a quasiconvex hierarchy for a group $G$ only requires that the edge groups are quasi-isometrically embedded in $G$; when a graph of groups $(\Gamma,\phi,T)$ structure for $G$ is quasiconvex, the vertex groups are quasi-isometrically embedded as well.
\begin{lemma}\label{vertex are qi}
Let $(\Gamma,T)$ be a graph of groups structure for $G$. If the edge groups of $\Gamma$ are quasi-isometrically embedded in $G$, then the vertex groups of $\Gamma$ are quasi-isometrically embedded in $G$. 
\end{lemma}

Here is a rough sketch of the proof of Lemma~\ref{vertex are qi}.
A Cayley graph $\Lambda(G,S)$ of $G$ coarsely looks like a ``tree of spaces'' whose underlying (infinite) graph is the covering tree of $(\Gamma,T)$ where the edge spaces are Cayley graphs of edge groups and the vertex spaces are Cayley graphs of vertex groups. 
If $\Lambda_v:= \Lambda(G_v,S_v)$ is one of the vertex spaces, the coarse tree structure ensures that if a $\Lambda(G,S)$--geodesic shortcut $\gamma$ between two points in $\Lambda_v$ exits $\Lambda_v$ through an edge space $\Lambda_e$, it must return through $\Lambda_e$.
If $\gamma$ enters and exits $\Lambda_v$ at points $p_{e_1},p_{e_1}',\ldots,p_{e_m},p_{e_m}'$, let $\gamma_i$ be the image (in $\Lambda(G,S)$) of a $\Lambda_e$--geodesic between $p_{e_i}$ and $p_{e_i}'$. There exist $\lambda\ge 1$ and $\epsilon>0$ so that every $\gamma_i$ is $(\lambda,\epsilon)$--quasigeodesic in $\Lambda(G,S)$. 
We can build a new path $\rho$ from $\gamma$ by replacing the subsegment of $\gamma$ from $p_{e_i}$ to $p_{e_i}'$ with $\gamma_i$. Then $\rho$ lies entirely in the image of $\Lambda_v$ and hence $\rho$ is at least as long as the $\Lambda_v$--distance between its endpoints.  
Now the length of $\rho$ is at most $\lambda|\gamma|+\epsilon$, or equivalently, $|\gamma|\ge \frac1\lambda |\rho| -\epsilon$. Thus $\gamma$ cannot be much shorter than the shortest path in $\Lambda_v$ between the endpoints of $\gamma$. 

\subsection{Fully $\mathcal{P}$-Elliptic Hierarchies}\label{P ell H}
Given a relatively hyperbolic group pair $(G,\mathcal{P})$ and a hierarchy $\mathcal{H}$ for $G$, the goal in Section~\ref{MSQT section} will be to strategically find a quotient of $G$ that has a hierarchy induced by $\mathcal{H}$ and inherits a relatively hyperbolic structure from $(G,\mathcal{P})$ that is also compatible with the induced hierarchy structure. Theorem~\ref{Thm: Wise main thm} can then be used to show the resulting quotient is virtually special. To ensure that this happens, some additional restrictions must be imposed on the interactions between the edge and vertex groups of the hierarchy and the peripheral subgroups of $G$.

\begin{definition}\label{Def: P elliptic}
Let $\mathcal{H}$ be a hierarchy for a group $G$ and let $\mathcal{P}$ be a collection of subgroups of $G$. Let $\mathcal{V}$ be the vertex groups of $\mathcal{H}$. For each $H\in \mathcal{V}$, let $\pi_1(\Gamma_H,\phi_H,T_H)$ be the graph of groups structure for $H$ induced by the hierarchy $\mathcal{H}$. 
The hierarchy $\mathcal{H}$ is \textbf{$\mathcal{P}$-elliptic} if the following holds: whenever there exists a $g\in G$ such that $P^g := gPg\inv \subseteq H\in \mathcal{V}$, then there exists an $h\in H$ such that $hP^gh\inv$ is contained in some vertex group of $\Gamma_H$.

A $\mathcal{P}$-elliptic hierarchy is \textbf{fully $\mathcal{P}$ elliptic} if the following holds: whenever $E$ is an edge group in $\mathcal{H}$, then for all $g\in G$, either $P^g \cap E$ is finite or $P^g\le E$. 
\end{definition}

When $\mathcal{H}$ is a fully $\mathcal{P}$-elliptic hierarchy for $G$ and $G_0$ is a finite index normal subgroup of $G$, the induced hierarchy from Corollary~\ref{findex forms} for $H$ is also fully $\mathcal{P}$-elliptic in the induced peripheral structure provided by Proposition~\ref{Prop: induced peripheral structure}:
\begin{proposition}\label{findex Pelliptic}
Suppose that $G_0$ is finite index normal in $G$ and let $(G_0,\mathcal{P}_0)$ be the peripheral structure induced on $G_0$ by Proposition~\ref{Prop: induced peripheral structure}. If $G$ has a fully $\mathcal{P}$-elliptic hierarchy, then the induced hierarchy $\mathcal{H}_0$ of $G_0$ is fully $\mathcal{P}_0$-elliptic. 
\end{proposition}

Proposition~\ref{findex Pelliptic} follows immediately from the explicit characterizations of the edge and vertex groups of the induced hierarchies in Corollary~\ref{findex forms} and from the explicit description of the induced peripheral structure.

\section{The Relative Fellow Traveling Property}\label{RelFelTrv}


\input{ch4new}

\section{A Relatively Hyperbolic Combination Lemma}
\label{Combo Lemma Section}

The construction of hierarchies in Section~\ref{Hierarchy section} is quite similar to the hierarchy constructed in \cite{AGM}. The goal of this section is to prove a combination theorem for the relatively hyperbolic setting that will be used to show the edge groups of the hierarchy are undistorted. 

\subsection{The Attractive Property in \CAT$(0)$ Relatively hyperbolic pairs}\label{S: attraction}

The first goal is to improve a \CAT$(0)$ relatively hyperbolic pair so that geodesics that stay near a peripheral space intersect the peripheral space. 

\begin{definition}\label{Def: attractive}
Let $\tilde{X}$ be a geodesic metric space, let $Z$ be a subspace of $\tilde{X}$ and let $K_{att}:\reals^{\ge 0}\to\reals^{\ge 0}$ be a function. The subspace $Z$ is \textbf{$K_{att}$--attractive} if for all $R\ge \delta$  whenever $\gamma$ is a geodesic with endpoints in $\mathcal{N}_R(Z)$ and $|\gamma| \ge K_{att}(R)$, then $\gamma\cap Z \ne \emptyset.$
\end{definition}

We now fix hypotheses for the remainder of the Section~\ref{S: attraction}. 
\begin{hypotheses}\label{H: attraction}
Suppose that $(\tilde{X},\mc{B}')$ is a $(\delta,f')$--\CAT$(0)$ relatively hyperbolic pair where every $F'\in \mc{B}'$ is convex. 
Let $\mc{B} = \{\mc{N}_{2\delta}(F'):\,F'\in\mc{B}'\}$ so that for some $f:\reals^{\ge 0}\to\reals^{\ge 0}$, $(\tilde{X},\mc{B})$ is a $(\delta,f)$--\CAT$(0)$ relatively hyperbolic pair by Proposition~\ref{P: fattening peripherals}. Fix $M=f(6\delta)$. 
\end{hypotheses}

\begin{proposition}\label{New rel hyp pair}
Under Hypotheses~\ref{H: attraction}, every $B\in\mathcal{B}$ is $(3M+6R+21\delta)$--attractive. 
\end{proposition}




\begin{figure}
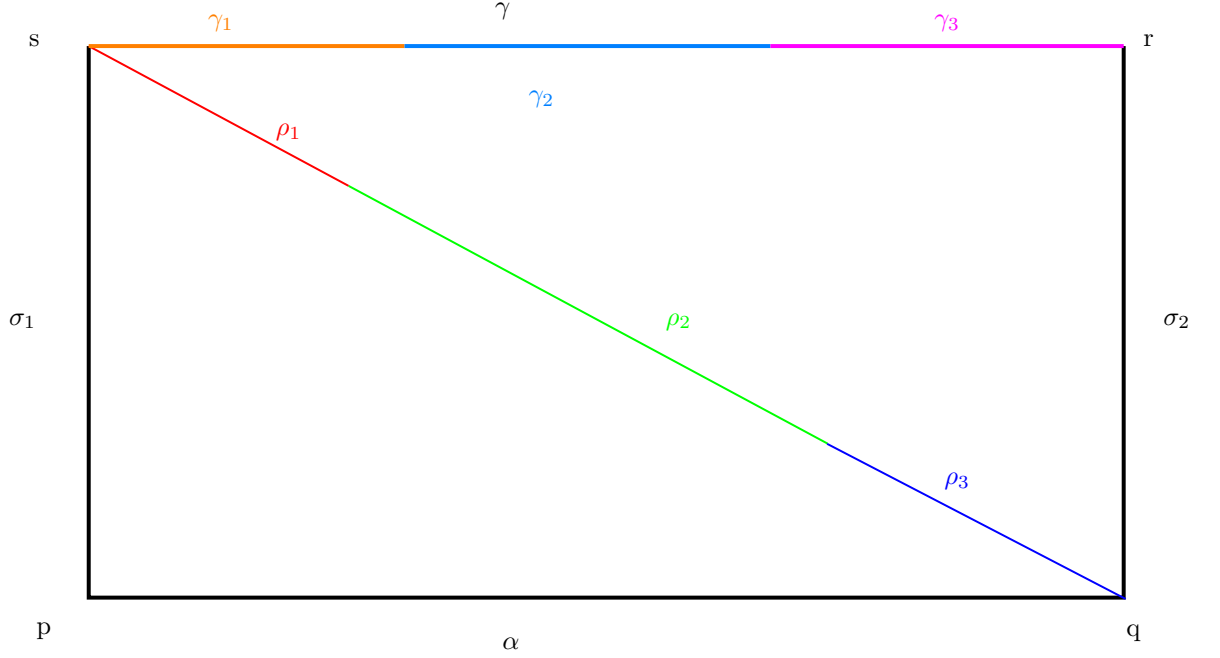
\caption{The quadrilateral constructed in the proof of Proposition~\ref{New rel hyp pair}.}
\label{Fig: quadrilateral for avoidance lemma}
\end{figure}

The following result will be used to prove Proposition~\ref{New rel hyp pair}:

\begin{proposition}\label{avoidance lemma}
Assume Hypotheses~\ref{H: attraction}, let $\gamma$ be a geodesic and let $F\in\mc{B}'$. If $\gamma$ has endpoints in $\mathcal{N}_{R}(F)$, then $\diam(\gamma\cap \mathcal{N}_{2\delta}(F)) > |\gamma| - (3M +6R+9\delta)$. 
\end{proposition}

\begin{proof}
There is a quadrilateral whose sides are $\gamma$, two geodesics $\sigma_1,\sigma_2$ of length at most $R$ connecting the endpoints of $\gamma$ to points in $F$ and a geodesic $\alpha$ connecting the endpoints of $\sigma_1,\sigma_2$ that are in $F$. By convexity, $\alpha\subseteq F$. 
Let $\rho$ be a diagonal so that there are two triangles, $\triangle_1,\triangle_2$ so that $\triangle_1$ has sides $\alpha,\rho,\sigma_1$ as a side and $\triangle_2$ has sides $\gamma,\rho,\sigma_2$. 
Designate vertices $p,q,r,s$ so that $\alpha = [p,q],\,\sigma_2 = [q,r],\,\gamma = [r,s],\,\sigma_1 = [p,s]$,\,and $\rho=[q,s]$ as shown in Figure~\ref{Fig: quadrilateral for avoidance lemma}. 

\textbf{Case 1: $\triangle_1$ is $\delta$--thin relative to some $F'\ne F$}. 

Since $F'\ne F$ and $\alpha\subseteq F$, the length of the fat part of $\alpha$ in $\triangle_1$ is at most $M$. 


Let $\rho_1$ be the corner segment of $\rho$ in $\triangle_1$ at $s$. Then $|\rho_1|\le R$.  Let $\rho_2$ be the fat part of $\rho$ in $\triangle_1$. The fat part of $\sigma_1$ in $\triangle_1$ has length at most $R$, so by Lemma~\ref{Transferrence lemma}, $|\rho_2| \le M+R+3\delta$. Let $\rho_3$ be the corner segment of $\rho$ in $\triangle_1$ at $q$. By construction, $\rho_3 \subseteq \mathcal{N}_\delta(F)$. 

Let $\gamma_1$ be the corner segment of $\gamma$ at $s$ in $\triangle_2$, let $\gamma_2$ be the fat part of $\gamma$ in $\triangle_2$ and let $\gamma_3$ be the corner segment of $\gamma$ in $\triangle_2$ at $r$. 
Observe that $(\gamma_1\cap\mathcal{N}_\delta(\rho_3))\subseteq \mathcal{N}_{2\delta}(F)$ and
\[\diam(\gamma_1\cap \mathcal{N}_\delta(\rho_3)) \ge |\gamma_1|-|\rho_1|-|\rho_2| \ge |\gamma_1|-(M+2R+3\delta).\]

If $\triangle_2$ is $\delta$--thin relative to $F$, then $\gamma_2\subseteq\mathcal{N}_\delta(F)$. 
If $\triangle_2$ is $\delta$--thin relative to some other element of $\mathcal{B}'$, the fat part of $\rho$ in $\triangle_2$ has length at most $|\rho_1|+|\rho_2|+M \le 2M+2R+3\delta$ because $\rho_3 \subseteq \mathcal{N}_\delta(F)$. 
By Lemma~\ref{Transferrence lemma}:
\[|\gamma_2| \le 2M+2R+3\delta + R+3\delta\]
because $|\sigma_2|\le R$.
Finally, $|\gamma_3|\le R$. 

In summary, at most $M+2R+3\delta$ of $\gamma_1$, lies outside of $\mathcal{N}_{2\delta}(F)$, at most $2M+3R+6\delta$ of $\gamma_2$ lies outside of $\mathcal{N}_{2\delta}(F)$, and at most $R$ of $\gamma_3$ lies outside of $\mathcal{N}_{2\delta}(F)$, so:
\[\diam(\gamma\cap{N}_{2\delta}(F)) \ge |\gamma| - (3M +6R+9\delta)\]
as desired.

\textbf{Case 2: $\triangle_1$ is $\delta$--thin relative to $F$}.

Let $\rho_1,\rho_2,\rho_3$ and $\gamma_1,\gamma_2,\gamma_3$ be as in the previous case.
Here, $|\rho_1|\le R$, $\rho_2\subseteq \mathcal{N}_\delta(F)$ since $\triangle_1$ is $\delta$--thin relative to $F$ and $\rho_3\subseteq \mathcal{N}_\delta(\alpha) \subseteq \mathcal{N}_\delta(F)$. 
Since $\gamma_1$ $\delta$--fellow travels a subsegment of $\rho$ at $s$, $\diam(\gamma_1\cap \mathcal{N}_\delta(\rho_2\cup \rho_3)) \ge |\gamma_1|-R$ because $|\rho_1|\le R$. 
Since $\rho_2\cup\rho_3\subseteq \mathcal{N}_\delta(F)$, $\diam(\gamma_1\cap \mathcal{N}_{2\delta}(F))\ge |\gamma_1|-R$. 
If $\triangle_2$ is $\delta$--thin relative to some $F''\ne F$, the fat part of $\rho$ in $\triangle_2$ has length at most $R+M$ because its intersection with $\rho_2\cup\rho_3\subseteq \mathcal{N}_\delta(F)$ has length at most $M$ and $|\rho_1|\le R$. 
Therefore by Lemma~\ref{Transferrence lemma}, $|\gamma_2| < M+2R+3\delta$. On the other hand, if $\triangle_2$ is $\delta$--thin relative to $F$, then $\gamma_2\subseteq \mathcal{N}_\delta(F)$ so in both cases, all but a less than $M+2R+3\delta$ subsegment of $\gamma_2$ lies in $\mathcal{N}_{2\delta}(F)$. 

In summary, $\diam(\gamma_1\cap \mathcal{N}_{2\delta}(F))\ge |\gamma_1|-R$, $\diam(\gamma_2 \cap\mathcal{N}_{2\delta}(F))\ge |\gamma_2|-(M+2R+3\delta)$ and $|\gamma_3|\le R$. 
Therefore, by the convexity of $\mathcal{N}_{2\delta}(F)$:
\[|\gamma\cap \mathcal{N}_{2\delta}(F)|\ge |\gamma| - (M+4R+3\delta)\]
as desired.
\end{proof}

\begin{proof}[Proof of Proposition~\ref{New rel hyp pair}]
Let $\gamma$ be a geodesic with endpoints in $\mathcal{N}_{R}(F)$. Then by convexity, $\gamma\subseteq \mathcal{N}_{R+2\delta}(F')$ for some $F'\in\mathcal{B}'$ where $F = \mathcal{N}_{2\delta}(F')$. By Proposition~\ref{avoidance lemma}, if $|\gamma| > 3M+6(R+2\delta)+9\delta$, then $\gamma\cap \mathcal{N}_{2\delta}(F')\ne \emptyset.$ Noting that $F = \mathcal{N}_{2\delta}(F')$ completes the proof.
\end{proof}

\subsection{A combination lemma for \CAT$(0)$ relatively hyperbolic pairs}\label{combo lemma sec}

Maintain the following baseline hypotheses for Section~\ref{combo lemma sec}:
\begin{hypotheses}\label{baseline combo lemma}
Let $(\tilde{X},\mathcal{B})$ be a $(\delta,f)$--\CAT$(0)$ relatively hyperbolic pair and let $M=f(6\delta)$ as before. Suppose that every $B\in\mathcal{B}$ is closed, convex and $(3M+6R+2f(R)+21\delta)$--attractive.
\end{hypotheses}

In Section~\ref{Hierarchy section}, we will use Proposition~\ref{New rel hyp pair} to obtain attractiveness for a $(\delta,f)$--\CAT$(0)$ relatively hyperbolic pair, and then thicken the peripheral spaces to make a new $(\delta,f)$--\CAT$(0)$ relatively hyperbolic pair. We will then prove that the new peripheral spaces are $(3M+6R+2f(R)+21\delta)$--attractive. For this reason, Hypotheses~\ref{baseline combo lemma} are slightly weaker than what would follow from Hypotheses~\ref{H: attraction} and the conclusions of Proposition~\ref{New rel hyp pair}.

\begin{theorem}\label{quasiconvexity thm}
Assume Hypotheses~\ref{baseline combo lemma}.
Let $\gamma = b_1a_2b_2a_3b_3\ldots a_nb_n$ be a broken geodesic. Let $\gamma_i$ be the geodesic connecting the endpoints of the subpath $b_1a_2b_2a_3b_3\ldots a_ib_i$ of $\gamma$.
Suppose that:
\begin{enumerate}
\item For each $1\le i \le n$, there exists some $F_i\in\mathcal{B}$ so that $b_i\subseteq F_i$.
\item If $F_i = F_j$, then $i=j$. 
\item For $1\le i\le n-1$, $|b_i|\ge 37M+250 \delta$. 
\item  For all $2\le i\le n$,\,$\diam(a_i\cap\mathcal{N}_{3\delta}(F_{i}))\le 5M+39\delta$ and $\diam(a_i\cap \mc{N}_{3\delta}(F_{i-1}))\le 5M+39\delta$. \label{H: ai avoiding}
\item For all $2\le i\le n$, $\diam(a_i\cap \mathcal{N}_{6\delta}(F_{i}))\le 5M + 57\delta$ and $\diam(a_i\cap \mc{N}_{6\delta}(F_{i-1}))\le 5M+57\delta$.
\end{enumerate}
Then $\gamma_n$ has a length at least $|b_n| - (24M+165\delta)$--tail at the endpoint it shares with $b_n$ (recall Definition~\ref{D: tail}) that lies in $\mathcal{N}_{2\delta}(F_n)$ and for all $2\le i\le n$, $|\gamma_i|\ge |\gamma_{i-1}|+|a_n|+|b_n| - 68M -628\delta$.
\end{theorem}

\begin{figure}
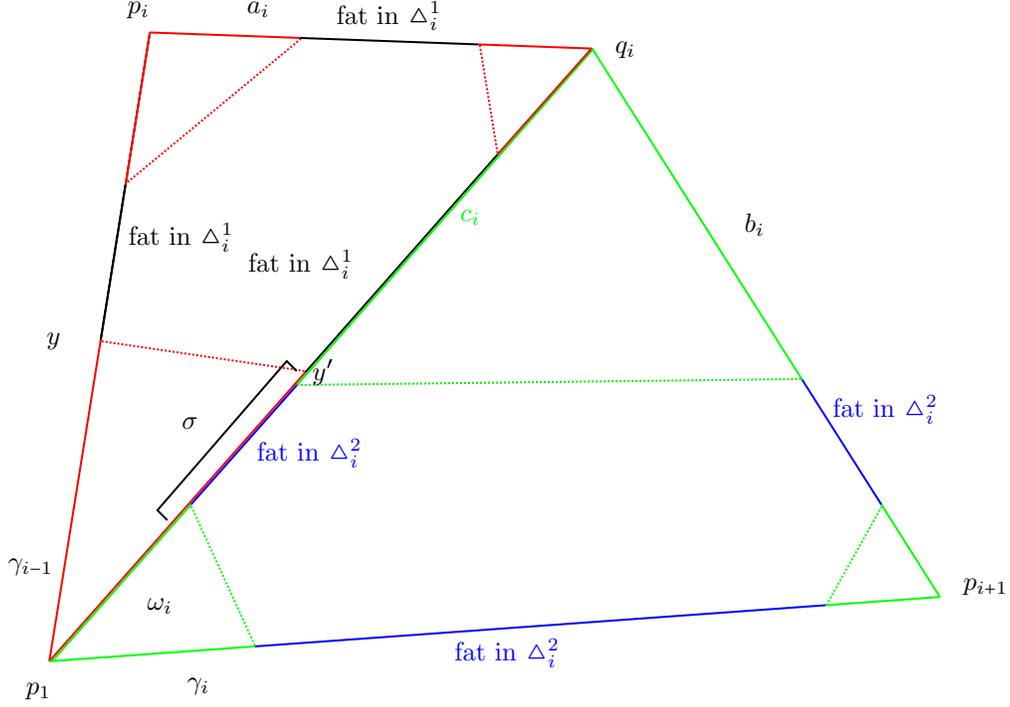
\caption{One possible configuration of $\triangle_i^1$ and $\triangle_i^2$ in the proof of Theorem~\ref{quasiconvexity thm}. Corner segments of triangles at the same point are connected by dotted lines. }\label{quasiconvexity triangle fig}
\end{figure}

\begin{proof}
In the case $n=1$, the proof is straightforward. The proof of Theorem~\ref{quasiconvexity thm} is by induction on $n$. 

\begin{notation}\label{N: combo}
We now establish notation that will be used throughout the proof of Theorem~\ref{quasiconvexity thm}.
\begin{enumerate}
\item For each $2\le i\le n$, let $\omega_i$ be the geodesic connecting the endpoints of the broken geodesic $b_1a_2b_2\ldots b_{i-1}a_i$, and
\item For each $2\le i\le n$, let $\triangle_i^1$ be the triangle with sides $\gamma_{i-1},\, \omega_i$ and $a_i$. 
\item For each $2\le i\le n$, let $\triangle_i^2$ be the triangle with sides $\omega_i$, $b_i$ and $\gamma_i$. 
\item Label vertices so that $a_i= [p_i,q_i]$ and $b_i=[q_i,p_{i+1}]$.
\item Let $c_i$ be the corner segment of $\omega_i$ in $\triangle_i^2$ at $q_i$. \label{I: ci corner}
\end{enumerate}
\end{notation}
See Figure~\ref{quasiconvexity triangle fig} for a visual representation.

We make the additional inductive assumption that for $1\le i < n$, $\gamma_i$ has a $|b_i| - (24M+165\delta)$--tail at $p_{i+1}$ in $\mathcal{N}_{2\delta}(F_i)$.

\begin{proposition}\label{Prop: main qgd utility}
If $i\ge 2$ and we assume the inductive hypotheses for the proof of Theorem~\ref{quasiconvexity thm}, then there is a point $x_i\in \gamma_i$ so that $d(x_i,F_{i-1}) \le 4\delta$. 
Further, $|c_i|\le 12M+81\delta$ (recall Notation~\ref{N: combo}\eqref{I: ci corner}). 
When $\triangle_i^2$ is $\delta$--thin relative to $F_i$, then the length of the fat part of $\omega_i$ in $\triangle_i^2$ is at most $12M+81\delta$.  
\end{proposition}

\begin{proof}
Since $1\le i-1<n$, $\gamma_{i-1}$ has a length at least $13M+85\delta$--tail at $p_i$ in $\mathcal{N}_{2\delta}(F_{i-1})$ by our inductive assumption.

\textbf{Case: $\triangle_i^1$ is thin relative to $F\ne F_{i-1}$}.
The corner segments of $\triangle_i^1$ at $p_i$ have length at most $5M+39\delta$ to avoid violating Theorem~\ref{quasiconvexity thm} Hypothesis \eqref{H: ai avoiding} because a more than $5M+39\delta$--tail of $\gamma_{i-1}$ at $p_i$ lies in $\mathcal{N}_{2\delta}(F_{i-1})$. 
Since $\triangle_i^1$ is thin relative to $F\ne F_{i-1}$, the length of the fat part of $\gamma_{i-1}$ in $\triangle_i^1$ is at most $M$. 
Therefore, there is a point $y\in\gamma_{i-1}$ and a point $y'\in\omega_i$ so that $d(y,p_i) \le 6M+39\delta$ and $d(y,y')\le \delta$ so that $y,y'$ are endpoints of the corner segments of $\triangle_i^1$ at $q_1$ and further, there exists a subsegment $\sigma$ (see Figure~\ref{quasiconvexity triangle fig}) of the corner segment $[q_1,y']\subseteq\omega_i$ with endpoint $y'$ so that $|\sigma| >  2M$ and $\sigma\subseteq \mathcal{N}_{3\delta}(F_{i-1})$. 

The intersection of $\sigma$ with the corner segment of $\omega_i$ in $\triangle_i^2$ at $q_i$ lies in $\mc{N}_{3\delta}(F_{i-1})\cap \mc{N}_{3\delta}(F_i)$ and therefore has length at most $M$. 
The fat part of $\omega_i$ in $\triangle_i^2$ is either contained in $\mathcal{N}_\delta(F_{i-1})$ or intersects $\sigma$ in a segment of length at most $M$. 
Therefore, either there is a point in $\gamma_i$ that is at most $\delta$ from the fat part of $\omega_i$ in $\triangle_i^2$ and the fat part of $\omega_i$ in $\triangle_i^2$ is contained in $\mathcal{N}_\delta(F_{i-1})$ or $\sigma$ intersects the corner segment of $\triangle_i^2$ at $q_1$.
In the first case, there is a point $x_i\in \gamma_i$ that lies in $\mathcal{N}_{2\delta}(F_{i-1})$ and in the second case, there is a point $x_i\in\gamma_i$ so that $d(x_i,\sigma)<\delta$, so $x_i\in \mathcal{N}_{4\delta}(F_{i-1})$. 

The next tasks are to bound $|c_i|$ from above and to prove that when $\triangle_i^2$ is $\delta$--thin relative to $F_i$, the fat part of $\omega_i$ in $\triangle_i^2$ has length at most $M$. 
Note that $c_i\subseteq \mathcal{N}_{2\delta}(F_i)$. 
The intersection of $c_i$ with the corner segment of $\omega_i$ in $\triangle_i^1$ at $q_i$ has length at most $5M+39\delta$ because $\diam(a_i\cap \mathcal{N}_{3\delta}(F_i))\le 5M+39\delta$. 
If $F\ne F_i$, the intersection of $c_i$ with the fat part of $\omega_i$ in $\triangle_i^1$ is a segment of length at most $M$. 
Since $|c_i\cap \sigma|\le M$ and $|\sigma|> 2M$, $|c_i|\le 7M+39\delta$. Further, if $\triangle_i^2$ is $\delta$--thin relative to $F_i$, then the fat part of $\omega_i$ in $\triangle_i^2$ intersects $\sigma$ in a segment of length at most $M$, intersects the fat part of $\omega_i$ in $\triangle_i^1$ in a length at most $M$ segment and intersects the corner segment of $\omega_i$ in $\triangle_i^1$ at $q_i$ in a segment of length at most $5M+39\delta$.
Hence the fat part of $\omega_i$ in  $\triangle_i^2$ has length at most $7M +39\delta$ when $\triangle_i^2$ is thin relative to $F_i$.

If $F = F_i$, then the fat parts of $a_i$ and $\gamma_{i-1}$ in $\triangle_i^1$, which are contained in $\mathcal{N}_\delta(F_i)$, have length at most $5M+39\delta$ and $M$ respectively. 
Therefore, the length of the fat part of $\omega_i$ in $\triangle_i^1$ is at most $5M+39\delta+M+3\delta$. 
Then $|c_i| \le 12M+81 \delta$ by a computation similar to the one in the previous case.

When $F=F_i$, the fat part of $\omega_i$ in  $\triangle_i^2$ intersects $\sigma$ in a segment of length at most $M$, intersects the fat part of $\omega_i$ in $\triangle_i^1$ in a segment of length at most $6M+42\delta$ and intersects the corner segment of $\omega_i$ in $\triangle_i^1$ at $q_i$ in a segment of length at most $5M+39\delta$. 
Therefore, if $\triangle_i^2$ is thin relative to $F_i$, then the length of the fat part of $\omega_i$ in $\triangle_i^2$ is at most $12M+81\delta$.  

\textbf{Case: $\triangle_i^1$ is thin relative to $F_{i-1}$}.

Recall $c_i$ is the corner segment of $\omega_i$ in $\triangle_i^2$ at $q_i$. 
The intersection of $c_i$ with the corner segment of $\omega_i$ in $\triangle_i^1$ at $q_i$ again lies in $\mathcal{N}_{2\delta}(F_i)\cap\mathcal{N}_{\delta}(a_i)$ and hence has length at most $5M+39\delta$.
The fat part of $\omega_i$ in $\triangle_i^1$ lies in $\mathcal{N}_\delta(F_{i-1})$. Hence, if the length of the fat part of $\omega_i$ in $\triangle_i^1$ exceeds $M$, then its intersection with $c_i$ has length at most $M$ so $|c_i|\le 6M+39\delta$. 
Hence for the purposes of bounding $|c_i|$ from above, assume the fat part of $\omega_i$ in $\triangle_i^1$ has length at most $M$.  
The length of the fat part of $a_i$ in $\triangle_i^1$ is at most $5M+39\delta$. If the length of the fat part of $\omega_i$ in $\triangle_i^1$ is at most $M$, then by Lemma~\ref{Transferrence lemma}, the length of the fat part of $\gamma_{i-1}$ in $\triangle_i^1$ is at most $6M+42\delta$. 
Now, if $y\in \gamma_{i-1},y'\in \gamma_{i-1}$ are the endpoints of the corner segments of $\triangle_i^1$ at $q_1$, then $d(y,p_i)\le 5M+39\delta + 6M+42\delta = 11M + 81\delta$.
Therefore there is a tail at $y'$ of the corner segment of $\omega_i$ in $\triangle_i^1$ at $q_1$ called $\sigma$ so that $|\sigma|>2M$ and $\sigma\subseteq \mathcal{N}_{3\delta}(F_{i-1})$ because $\gamma_{i-1}$ has a more than $13M+84\delta$--tail in $\mathcal{N}_{2\delta}(F_{i-1})$. 
Therefore, $c_i$ intersects $[y',q_1]$ in a segment of length at most $M$ because $c_i\subseteq \mathcal{N}_{2\delta}(F_i)$. 
Hence $|c_i|\le 7M + 39\delta$ because the union of the two corner segments of $\omega_i$ in $\triangle_i^1$ and the fat part of $\omega_i$ in $\triangle_i^1$ is $\omega_i$. 

In all cases, $|c_i|\le 12M+ 81\delta$.

If $\triangle_i^2$ is $\delta$--thin relative to $F_i$, the fat part of $\omega_i$ in $\triangle_i^2$ has length at most $6M+39\delta$ because the corner segment of $\omega_i$ in $\triangle_i^1$ at $q_i$ lies in $ \mathcal{N}_\delta(a_i)$, and both $\sigma$ and the fat part of $\triangle_i^1$ lie in $\mathcal{N}_{\delta}(F_{i-1})$. In particular, the fat part of $\omega_i$ in $\triangle_i^2$ may only intersect $[q_1,y']$ in $\sigma$ because otherwise its intersection with $\sigma$ has length more than $M$ and lies in $\mathcal{N}_{3\delta}(F_{i-1})\cap \mathcal{N}_{\delta}(F_i)$. 
 
The only remaining thing to prove is that there is a point $x_i\in \gamma_i$ so that $d(x_i,F_{i-1})\le 4\delta$. 
If $\triangle_i^2$ is $\delta$--thin relative to $F_{i-1}$ and is not $\delta$--thin relative to any other $F\in\mathcal{B}$, 
 then there is a point on $\gamma_i$ in $\mathcal{N}_{2\delta}(F_{i-1})$. Hence assume $\triangle_i^2$ is thin relative to some $G\in\mathcal{B}$ with $G\ne F_{i-1}$. 

Let $\omega^1\subseteq \mc{N}_\delta(F_{i-1})$ be the fat part of $\omega_i$ in $\triangle_i^1$ and let $\omega^2$ be the corner segment of $\omega_i$ in $\triangle_i^2$ at $q_1$. If there exists $r\in \omega^1\cap \omega^2$, then $d(r,\gamma_i) <\delta$, so there exists an $x_i\in \gamma_i$ such that $\gamma_i\in \mathcal{N}_{2\delta}(F_{i-1})$. 

Otherwise, $\omega^1$ intersects $c_i$ in a segment of length at most $M$ because $c_i$ lies in $\mathcal{N}_{2\delta}(F_i)$ and intersects the fat part of $\omega_i$ in $\triangle_i^2$ in a segment of length at most $M$ (the fat part of $\omega_i$ in $\triangle_i^2$ lies in $\mathcal{N}_\delta(G)$).  Hence $|\omega^1|\le 2M$. Let $\omega^3$ be the corner segment of $\omega_i$ in $\triangle_i^1$ at $q_1$.
Let $z\in\omega_i$ be the point where $\omega^1$ intersects $\omega^3$.
By Lemma~\ref{Transferrence lemma}, the fat part of $\gamma_{i-1}$ in $\triangle_i^1$ has length at most $2M+5M+39\delta+3\delta = 7M +42\delta$ because $\diam(a_i\cap \mathcal{N}_{2\delta}(F_{i-1}))\le 5M+39\delta$.
The corner length of $\triangle_i^1$ at $p_i$ is at most $5M+39\delta$ because any subsegment of $a_i$ in $\mathcal{N}_{3\delta}(F_i)$ has length at most $5M+39\delta$.   
Then at least a $13M+84\delta -(5M+39\delta +7M +42\delta)>M$--tail of $\omega^3$ at $z$, which will be called $\omega'$, lies in $\mathcal{N}_{3\delta}(F_{i-1})$ because it $\delta$--fellow travels a subsegment of the tail of $\gamma_{i-1}$ at $p_{i}$ contained in $\mathcal{N}_{2\delta}(F_{i-1})$. 
The union of $c_i$ and the fat part of $\triangle_i^2$ lie in $\mc{N}_{2\delta}(F_i)$, so they collectively cannot extend past $\omega'$ in the direction of $q_1$ because otherwise $\omega'$contains a length more than $M$ subsegment in $\mc{N}_{3\delta}(F_i)\cap \mc{N}_{3\delta}(F_{i-1})$. Therefore, $\omega^2$, the corner segment of $\triangle_i^2$ at $q_1$, must intersect $\omega'$. 
Since $\omega'$ lies in $\mathcal{N}_{3\delta}(F_{i-1})$ and $\omega^2$ is a corner segment of $\triangle_i^2$ at $q_1$, there is a point $x_i\in \gamma_i$ so that $x\in \mathcal{N}_{4\delta}(F_{i-1})$.
\end{proof}

\begin{proposition}
If $b_i\subseteq \mathcal{N}_{\delta}(F_i)$, then the geodesic $\gamma_i$ has a $|b_i|-(24M+165\delta)$--tail at $p_{i+1}$ that is contained in $\mathcal{N}_{2\delta}(F_i)$. \label{inductive tail}
\end{proposition}

\begin{proof}
There are two cases:

\textbf{Case 1: $\triangle_i^2$ is $\delta$--thin relative to some $F\ne F_i$.}

The corner length of $\triangle_i^2$ at $q_i$ is at most $12M+81\delta$ by Proposition~\ref{Prop: main qgd utility}. The length of the fat part of $b_i$ in $\triangle_i^2$ is at most $M$ because $b_i\subseteq \mathcal{N}_\delta(F)$. Therefore, the corner length of $\triangle_i^2$ at $p_{i+1}$ is at least $|b_i|-(13M+81\delta)$. 
Thus the corner segment of $\gamma_i$ at $p_{i+1}$ has length at least $|b_i|-(13M+81\delta)$ and lies in $\mathcal{N}_{\delta}(b_i)\subseteq \mathcal{N}_{2\delta}(F_i)$. 

\textbf{Case 2: $\triangle_i^2$ is $\delta$--thin relative to $F_i$.}

The corner length of $\triangle_i^2$ at $q_i$ is at most $12M+81\delta$. Let $s$ be the length of the fat part of $b_i$ in $\triangle_i^2$. 
Then the corner length of $\triangle_i^2$ at $p_{i+1}$ is at least $|b_i|-s-(12M+81\delta)$. 
By Proposition~\ref{Prop: main qgd utility}, the length of the fat part of $\omega_i$ in $\triangle_i^2$ is at most $12M+81\delta$. 
By Lemma~\ref{Transferrence lemma}, the fat part of $\gamma_i$ in $\triangle_i^2$ has length at least $s-(12M+81\delta+3\delta)$. 
The corner segment of $\gamma_i$ at $p_{i+1}$ in $\triangle_i^2$ and the fat part of $\gamma_i$ in $\triangle_i^2$ both lie in $\mathcal{N}_{2\delta}(F_i)$ and their combined length is at least $s-(12M+84\delta) + |b_i| - s - (12M+81\delta) = |b_i| -(24M+165\delta)$. 
\end{proof}

\begin{lemma}\label{bound F i-1}
Let $\eta: = [p_i,p_{i+1}]$. Then $\diam(\eta\cap \mathcal{N}_{5\delta}(F_{i-1}) )\le 12M+117\delta$. 

Further, $d(q_i,\eta)\le 10M+79\delta$. 
\end{lemma}

\begin{proof}Let $\triangle$ be the geodesic triangle with sides $a_i,b_i,\eta$. 
If the corner segment of $\eta$ in $\triangle$ at $p_i$ lies in $\mathcal{N}_{5\delta}(F_{i-1})$, then the corner length of $\triangle$ at $p_i$ is at most $5M+57\delta$ because $a_i\cap \mathcal{N}_{6\delta}(F_{i-1})$ has diameter at most $5M+57\delta$. 

Suppose $\triangle$ is $\delta$--thin relative to $F_{i-1}$. The fat part of $b_i$ in $\triangle$ then lies in $F_i$ and therefore has length at most $M$. The fat part of $a_i$ in $\triangle$ has length at most $5M+57\delta$ because $a_i\cap \mathcal{N}_{6\delta}(F_{i-1})$ has diameter at most $5M+57\delta$.
Hence by Lemma~\ref{Transferrence lemma}, the length of the fat part of $\eta$ in $\triangle$ is at most $6M+60\delta$. 
On the other hand, if $\eta$ is $\delta$--thin relative to some $F\ne F_{i-1}$, then the intersection of the fat part of $\eta$ with $\mathcal{N}_{5\delta}(F_{i-1})$ has length at most $M$. 
In all cases, the fat part of $\eta$ in $\triangle$ intersects $\mathcal{N}_{5\delta}(F_{i-1})$ in a segment of length at most $6M+60\delta$.

Finally, the corner segment of $\eta$ in $\triangle$ at $p_{i+1}$ lies in $\mathcal{N}_{2\delta}(F_i)$ and can hence intersect $\mathcal{N}_{5\delta}(F_{i-1})$ in a segment of length at most $M$. 

Since $\eta$ is the union of its two corner segments and its fat part in $\triangle$, its intersection with $\mathcal{N}_{5\delta}(F_{i-1})$ has diameter at most $12M + 117\delta$. 

The corner length of $\triangle$ at $q_i$ is at most $5M+39\delta$, because the corner segment of $a_i$ in $\triangle$ at $q_i$ lies in $a_i\cap \mathcal{N}_{2\delta}(F_i)$. 
If $\triangle$ is $\delta$--thin relative to $F_i$, then the length of the fat part of $a_i$ in $\triangle$ is at most $5M+39\delta$. 
Otherwise, if $\triangle$ is $\delta$-thin relative to $F\ne F_i$, then the length of the fat part $b_i$ in $\triangle$ is at most $M$. 
Since $\triangle$ is relatively $\delta$--thin, in both cases, there exists a point on $\eta$ that is at most $5M+39\delta + 5M+39\delta +\delta = 10M +79\delta$ from $q_i$.
\end{proof}

\begin{lemma}\label{fat part of tr prime}
Let $x_i$ be a point on $\gamma_i$ so that $x_i\in \mathcal{N}_{4\delta}(F_{i-1})$ and $x_i$ is the point closest to $p_{i+1}$ with this property.
Let $\eta' = [p_i,x_i]$ and let $\eta'' = [x_i,p_{i+1}]\subseteq \gamma_i$. 
Let $\triangle'$ be the triangle with sides $\eta,\eta',\eta''$. Then at least one of the following holds:
\begin{enumerate}
\item the length of the fat part of $\eta$ in $\triangle'$ is at most $12M+117\delta$ OR
\item the length of the fat part of $\eta'$ in $\triangle'$ is at most $M\le 12M+117\delta$.
\end{enumerate}
\end{lemma}

\begin{proof}
Suppose $\triangle'$ is $\delta$--thin relative to $F_{i-1}$. Then by Lemma~\ref{bound F i-1}, the fat part of $\eta$ has length at most $12M+117\delta$. On the other hand if $\triangle'$ is $\delta$--thin relative to some $F\ne F_{i-1}$, then the fat part of $\eta'$ in $\triangle'$ lies in $\mathcal{N}_{4\delta}(F_{i-1})$ by convexity, so the length of the fat part of $\eta'$ in $\triangle'$ is at most $M$. 
\end{proof}

\begin{lemma}\label{defect estimate 2}
There exists $y_i\in\gamma_i$ so that $d(p_i,y_i)\le 24M + 235\delta$. 
\end{lemma}

\begin{proof}
The corner segment of $\eta$ in $\triangle'$ at $p_i$ lies in $\mathcal{N}_{5\delta}(F_{i-1})\cap \eta$, so by Lemma~\ref{bound F i-1}, the corner length of $\triangle'$ at $p_i$ is at most $12M+117\delta$. 
By Lemma~\ref{fat part of tr prime}, the length of fat part of $\eta$ in $\triangle'$ or the length of the fat part of $\eta'$ in $\triangle'$ is at most $12M+117\delta$, so there is a point $y_i$ in $\eta'' \subseteq \gamma_i$ so that $d(p_i,y_i) \le 24M+235\delta$ because $\triangle'$ is relatively $\delta$--thin. 
\end{proof}

The next lemma follows immediately from the triangle inequality, but is convenient to have recorded:
\begin{lemma}\label{Lem: defect counter}
Let $\triangle_0$ be a geodesic triangle in $\tilde{X}$ with sides $abc$ and suppose that $a$ and $b$ meet at the vertex $p$ and $d(p,c)\le J$. Then $|c|\ge |a|+|b|-2J$. 
\end{lemma}

\begin{proposition}\label{final defect estimate}
The length $|\gamma_n|\ge |\gamma_{n-1}| + |a_n|+|b_n| - 2(24M+235\delta) -2(10M+79\delta) = |\gamma_{n-1}| +|a_n|+|b_n| - 68M - 628\delta$. 
\end{proposition}

\begin{proof}
By Lemma~\ref{defect estimate 2} and Lemma~\ref{Lem: defect counter}:
\[|\gamma_n| \ge |\gamma_{n-1}| + |\eta| - 2 (24M+235\delta)\]
Then by Lemma~\ref{bound F i-1} and Lemma~\ref{Lem: defect counter}:
\[|\eta| \ge |a_n|+|b_n| - 2(10M+79\delta)\]
Putting the two preceding inequalities together yields the desired inequality. 
\end{proof}

Proposition~\ref{inductive tail} and Proposition~\ref{final defect estimate} complete the inductive proof of Theorem~\ref{quasiconvexity thm}.
\end{proof} 


\begin{definition}
Let $\mathcal{A}$ be a collection of subsets of a geodesic metric space and let $K\ge 0$.  Suppose that for all $A_1,A_2\in\mathcal{A}$ with $A_1\ne A_2$, $d(A_1,A_2)\ge K$, then the collection $\mathcal{A}$ is \textbf{$K$--separated}.\label{Def: separated collection}
\end{definition}

The paths in Theorem~\ref{quasiconvexity thm} are of a special type to facilitate the inductive proof. Proposition~\ref{Lem: path classes} generalizes Theorem~\ref{quasiconvexity thm} to apply to all geodesic paths coming from certain subspaces of $\tilde{X}$ with some additional assumptions:

\begin{hypotheses}\label{Hyp: strong qc}
Assume Hypotheses~\ref{baseline combo lemma} and assume the following:
\begin{enumerate}
\item Let $\Lambda:= 500M + 10000\delta$.
\item Let $\mathcal{A}$ be a $\Lambda$--separated collection of convex subspaces of $\tilde{X}$. 
\item Let $\mathcal{B}_0\subseteq \mathcal{B}$.
\item Let $T = \left(\bigsqcup_{A\in\mathcal{A}} A \right)\sqcup \left(\bigsqcup_{B\in \mathcal{B}_0} B\right)$.  Define an equivalence relation $\sim$ on $T$ by $x\sim y$ if and only if $x=y$ or for some $A\in \mathcal{A}$ and $B\in\mathcal{B}_0$, the images of $x$ and $y$ in $\tilde{X}$ agree and lie in the images of both $A$ and $B$. 
\end{enumerate}
\end{hypotheses}

\begin{proposition}\label{Lem: path classes}
Under Hypotheses~\ref{Hyp: strong qc}, if $T/\sim$ is path connected, then the natural inclusion of $T/\sim\hookrightarrow \tilde{X}$ is a $(2,114M+1592\delta)$--quasi-isometric embedding (where the metric on $T/\sim$ is the induced path metric). 
\end{proposition}

\begin{proof}
Let $\gamma$ be the image in $\tilde{X}$ of a geodesic in $T/\sim$ and let $\gamma'$ be the $\tilde{X}$-geodesic between its endpoints. 

Up to reversing the direction of $\gamma$, $\gamma$ can be written as a piecewise geodesic of one of the following piecewise geodesic forms:
\begin{enumerate}
\item $b_1a_2b_2\ldots a_nb_n$\, and $|b_1|,|b_n| \ge 37M+250\delta$ \label{I: gamma format 1}
\item $a_1b_1a_2b_2\ldots b_{n}a_{n+1}$ where $|a_1|,|a_{n+1}|\ne 0$ \label{I: gamma format 2}
\item $a_1b_1\ldots a_nb_n$ where $|a_1|\ne 0$ and $|b_n|\ge 37M+250\delta$.\label{I: gamma format 3}
\item $a_1b_1\ldots a_nb_n$ where $|a_1|\ne 0$ and $|b_n|\le  37M + 250\delta$. \label{I: gamma format 4}
\item $b_1a_2b_2\ldots a_nb_n,\,$ where both of $|b_1|,\,|b_n|$ are less than $37M+250\delta$. \label{I: gamma format 5}
\item $b_1a_2b_2\ldots a_nb_n,\,$ where $|b_1|<37M +250\delta$ and $|b_n|\ge 37M+250\delta$. \label{I: gamma format 6}
\end{enumerate}
where for each $1\le i\le n$, $a_i\subseteq A_i\in \mathcal{A}$, for all $1\le i\le n$, $b_i\subseteq B_i\in \mathcal{B}$, and for $2\le i\le n-1$, $|b_i| \ge \Lambda$ because $\mathcal{A}$ is a $\Lambda$--separated collection. Assume also that $n$ is minimal and $\gamma$ is subdivided in a way that maximizes the sum of the lengths of the $b_i$.

Note that if $i\ne j$, then $B_i\ne B_j$ because otherwise the subsegment $b_i\ldots b_j$ of $\gamma$ could be replaced by a single geodesic segment in $B_i\subseteq T/\sim$ contradicting minimality of $n$. 
By the maximality of the lengths of the $b_i$ and the $(3M+6R+2f(R) + 21\delta)$--attractiveness of every $B\in \mathcal{B}$, $\diam(a_i \cap \mathcal{N}_{3\delta}(B_i)),\diam(a_i \cap \mathcal{N}_{3\delta}(B_{i-1}))\le 5M+39\delta$ and $\diam(a_i\cap \mathcal{N}_{6\delta}(B_{i-1})),\diam(a_i\cap \mathcal{N}_{6\delta}(B_{i}))\le 5M+57\delta$ because otherwise the interiors of the $a_i$ intersect either $B_i$ or $B_{i-1}$ so that $b_i$ or $b_{i-1}$ respectively could be made longer by convexity.

For the following arguments, recall the earlier convention that the endpoints of the $a_i,b_i$ are labeled so that $a_i = [p_i,q_i]$ and $b_i = [q_i,p_{i+1}]$. 

\textbf{Case \eqref{I: gamma format 1}: $\gamma = b_1a_2b_2\ldots a_nb_n$\, and $|b_1|,|b_n| \ge 37M+250\delta$.}

By Theorem~\ref{quasiconvexity thm}, $|\gamma'|\ge |b_1| + \left(\sum_{i=2}^n |a_i|+|b_i|\right) - (n-1)\cdot (68M+628\delta).$
Since $|b_i|\ge 136M+1256 \delta,$ for $2\le i\le n-1$ then 
\begin{eqnarray*}|\gamma'| & \ge & |b_1| + \left(\sum_{i=2}^n |a_i|+|b_i|\right) - (n-1)(68M+628\delta) \\  
& \ge & \frac12\left(\sum_{i=2}^n |a_i|\right) +  |b_1| + \left(\sum_{i=2}^{n-1} (|b_i| - (68M + 628\delta)) \right) + |b_n|-(68M+628\delta)\\ 
& \ge & \frac12\left(\sum_{i=2}^n |a_i|\right) +  2(37M+250\delta) + \left(\sum_{i=2}^{n-1} (|b_i| - (68M + 628\delta)) \right) + (68M+628\delta)\\ 
& \ge & \frac12\left(\sum_{i=2}^n |a_i|\right) + \frac12\left(\sum_{i=1}^n |b_i|\right)-128\delta\\
& \ge & \frac12|\gamma| -128\delta
 \end{eqnarray*}
hence $\gamma$ is a $(2,128\delta)$--quasigeodesic in $\tilde{X}$ in this case. 

\textbf{Case \eqref{I: gamma format 2}: $\gamma = a_1b_1a_2b_2\ldots b_{n}a_{n+1}$ where $|a_1|,|a_{n+1}|\ne 0$.}
Since $\mathcal{A}$ is a $\Lambda$--separated collection, the path $\gamma_0 = b_1a_2b_2\ldots b_n$ satisfies the hypotheses of Theorem~\ref{quasiconvexity thm}. Let $\gamma_0'$ be the geodesic connecting the endpoints of $\gamma_0$.
Then $|\gamma_0'|\ge |\gamma_0|- n(68M+628\delta)$ by Theorem~\ref{quasiconvexity thm}.
By Theorem~\ref{quasiconvexity thm}, $\gamma_0'$ has a length at least $100M + 2000\delta$--tail in $\mathcal{N}_{2\delta}(B_n)$ at $p_{n+1}$ and a $100M+2000\delta$--tail at $q_1$ in $\mathcal{N}_{2\delta}(B_1)$. 

Let $\gamma_1$ be the geodesic $[p_1,p_{n+1}]$. Let $\triangle_1$ be the geodesic triangle with sides $a_1,\gamma_0'$ and $\gamma_1$.  
The corner length of $\triangle_1$ at $q_1$ is at most $5M+57\delta$ because $\diam(a_1\cap \mathcal{N}_{5\delta}(B_1))\le 5M +57\delta$ and $\gamma_0'$ has a long tail at $q_1$ in $\mathcal{N}_{2\delta}(B_1)$.
Either $\triangle_1$ is $\delta$--thin relative to $B\ne B_1$ so that the length of the fat part of $\gamma_0'$ in $\triangle_1$ has length at most $M$ because a long tail of $\gamma_0'$ at $q_1$ is contained in $\mathcal{N}_{2\delta}(B_1)$, or $\triangle_1$ is $\delta$--thin relative to $B_1$ in which case the length of the fat part of $a_1$ in $\triangle_1$ has length at most $5M+57\delta$.
Hence there is a point $z_1$ on $\gamma_1$ so that $d(z_1,q_1)\le 10M +116\delta$ because $\triangle_1$ is $\delta$--relatively thin. Therefore by Lemma~\ref{Lem: defect counter}, $|\gamma_1| \ge |\gamma_0'| + |a_1| - (20M + 232\delta)$

Next we want to show that $\gamma_1$ has a long tail at $p_{i+1}$ in $\mc{N}_{2\delta}(B_n)$. 
If $\triangle_1$ is $\delta$--thin relative to $B_1$, the corner length at $q_1$ is at most $5M+57\delta$, and the fat part of $\gamma_0'$ in $\triangle_1$ can have an at most length $M$ intersection with the at least $100M+2000\delta$--tail of $\gamma_0'$ at $p_{i+1}$ that lies in $\mc{N}_{2\delta}(B_n)$. 
On the other hand, if $\triangle_1$ is $\delta$--thin relative to $B\ne B_1$, then the corner length of $\triangle_1$ at $q_1$ is still at most $5M+57\delta$ and the long tail of $\gamma_0'$ at $q_1$ that lies in $\mc{N}_{2\delta}(B_1)$ forces the length of the fat part of $\gamma_0'$ in $\triangle_1$ to be at most $M$. In both cases, all but $6M+57\delta$ of the $100M+2000\delta$--tail of $\gamma_0'$ at $p_{n+1}$ that lies in $\mc{N}_{2\delta}(B_n)$ must lie in the corner segment of $\gamma_0'$ at $p_{n+1}$. Hence an at least $94M+1000\delta$--tail of $\gamma_1$ at $p_{n+1}$ must lie in $\mc{N}_{3\delta}(B_n)$. 

Let $\triangle_2$ be the triangle with sides $\gamma_1,\,a_n,\,\gamma'$. By imitating the argument for $\triangle_1$, there is a point $z_2\in \gamma'$ so that $d(z_2,p_{n+1})\le 10M+116\delta$.
Hence by Lemma~\ref{Lem: defect counter}:
\[|\gamma'| \ge |\gamma_1|+|a_n| -(20M +232\delta)\]
so that:
\[|\gamma'| \ge |a_0|+|\gamma_0'|+|a_n| - (40M + 464\delta)\]
and by the computation from the previous case:
\[|\gamma'| \ge |a_0| + \frac12|\gamma_0|-128\delta + |a_n| - (40M + 464\delta) \ge \frac12|\gamma|-128\delta -(40M +464\delta)\]
so that $\gamma$ is a $(2,40M +592\delta)$--quasigeodesic in $\tilde{X}$.

\textbf{Case \eqref{I: gamma format 3}: $\gamma = a_1b_1\ldots a_nb_n$, $|a_1|\ne 0$ and $|b_n|\ge 37M+250\delta$.}

Since $\mathcal{A}$ is a $\Lambda$--separated collection, the path $\gamma_0 = b_1a_2b_2\ldots b_n$ satisfies the hypotheses of Theorem~\ref{quasiconvexity thm}. Let $\gamma_0'$ be the geodesic connecting the endpoints of $\gamma_0$.
Then $|\gamma_0'|\ge |\gamma_0|-(n-1)(68M+628\delta)$ by Theorem~\ref{quasiconvexity thm}.
By an argument similar to the one in the previous case:
\[|\gamma'|\ge |\gamma_0'| + |a_1| - (20M + 232\delta)\]
and by arguments similar to the ones above, then:
\[|\gamma'|\ge \frac12 |\gamma| - (20M+360\delta)\]
so in this case, $\gamma$ is a $(2,20M + 360\delta)$--quasigeodesic in $\tilde{X}$. 

\textbf{Case \eqref{I: gamma format 4}: $\gamma = a_1b_1\ldots a_nb_n$ where $|a_1|\ne 0$, $|b_n|\le 37M + 250\delta$}.

By a previous case, the path $a_1b_1\ldots a_n$ is a $(2,40M+592\delta)$--quasigeodesic in $\tilde{X}$.
Hence $\gamma$ is a $(2,77M+1000\delta)$--quasigeodesic in $\tilde{X}$.



\textbf{Case \eqref{I: gamma format 5}: $\gamma = b_1\ldots a_nb_n$ where $|b_1|,|b_n|< 37M + 250\delta$}.

Applying the immediately preceding case to $a_2b_1\ldots a_nb_n$ and the fact that $|b_1|\le 37M+250\delta$ implies that $\gamma$ is a $(2, 114M + 1250\delta)$--quasigeodesic in $\tilde{X}$. 

\textbf{Case \eqref{I: gamma format 6}: $\gamma = b_1a_2b_2\ldots a_nb_n,\,$ where $|b_1|<37M +250\delta$ and $|b_n|\ge 37M+250\delta$.} 
By Case~(\ref{I: gamma format 3}), $a_2b_2\ldots a_nb_n$ is a $(2,20M+360\delta)$--quasigeodesic. 
Thus $b_1a_2b_2\ldots a_nb_n$ is a $(2,57M+510\delta)$--quasigeodesic because $|b_1|< 37M+250\delta$.

Now, assume $T/\sim$ is path connected. Let $T_0$ be the image of $T/\sim$ in $\tilde{X}$. 
Let $x,y\in T/\sim$. Let $\rho_T,\rho_{T_0},\rho$ be the geodesics connecting $x$ and $y$ in $T/\sim,\,T_0$ and $\tilde{X}$ respectively.
Since $T/\sim$ is path connected, $\rho_{T}$ maps to a path in $T_0$, $|\rho_{T_0}|\le |\rho_T|$. From the preceding, $\frac12|\rho_T|- (114M+1592\delta) \le |\rho|$.  
Combining these inequalities:
\[\frac12|\rho_{T_0}| -(114M+1592\delta)\le |\rho| \le |\rho_{T_0}|\]
making $\rho_{T_0}$ a $(2,114M+1592\delta)$ quasigeodesic.
\end{proof}

\begin{proposition}\label{Prop: edge space pi1 injective}
Under Hypotheses~\ref{Hyp: strong qc}, any geodesic in $T/\sim$ is not mapped to a loop in $\tilde{X}$. 
\end{proposition}

\begin{proof}
Let $\gamma$ be a $T/\sim$--geodesic that maps to a loop in $\tilde{X}$. If $\gamma\subseteq A\in\mathcal{A}$ or $\gamma\subseteq B\in\mathcal{B}$, then $\gamma$ cannot map to a loop in $A$ or a loop in $B$. Then $\gamma$ can be written as a piecewise geodesic of the form:
\[b_1a_2b_2\ldots a_nb_n\]
where each $b_i\subseteq B_i\in\mathcal{B}$ and each $a_i\subseteq A_i\subseteq A\in\mathcal{A}$, $|b_1|,|b_n|\ge \frac12 \Lambda$ and $|b_i|\ge \Lambda$ for all $1\le i\le n$.
Since $\Lambda> 4(114M + 1592\delta)$, $|\gamma| > 2(114M+1592\delta)$.
Since $\gamma$ maps to a $(2,114M+1592\delta)$--quasigeodesic in $\tilde{X}$, the distance between the endpoints of $\gamma$ must be positive, so $\gamma$ cannot map to a loop. 
\end{proof}

\section{The geometry of special cube complexes}\label{Sec: Special}

\subsection{Non-Positively Curved Cube complexes}

A \textbf{cube complex} is a union of Euclidean cubes $[0,1]^n$ of possibly varying dimensions glued isometrically along faces.
A \textbf{non-positively curved (NPC)} cube complex is a cube complex such that the link of every vertex is a flag simplicial complex. See \cite{WiseBook} Section 2.1 for details. 

In each cube $[0,1]^n$, fixing one coordinate at $\frac12$ makes a \textbf{codimension--$1$} midcube. A \textbf{hyperplane} $H$ is a connected union of midcubes glued isometrically along faces so that the intersection of $H$ with any cube is either a codimension--$1$ midcube or empty. See Figure~\ref{Fig: NPC example} for an example of an NPC cube complex and the link of a vertex.

\begin{figure}
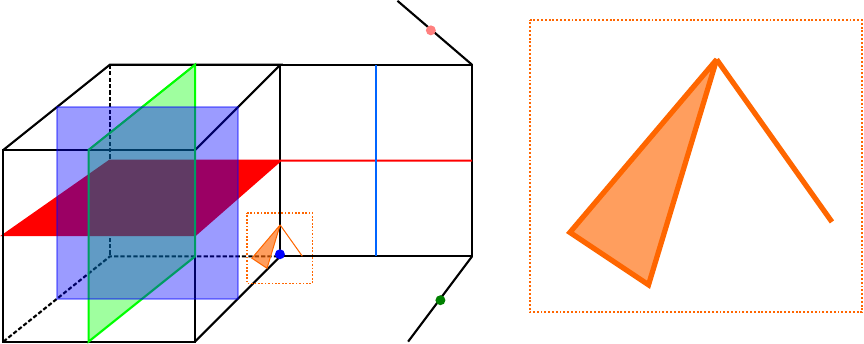
\caption{An example of a NPC cube complex (including a $3$--cube) with its hyperplanes as well as the link of the blue vertex shown in orange and enlarged on the right.}\label{Fig: NPC example}
\end{figure} 

\subsection{Special cube complexes and separability}\label{Special and Separability}
A \textbf{special cube complex} is a type of NPC cube complex developed by Wise and others whose hyperplanes are embedded, are $2$--sided and avoid two other pathologies, see \cite[Definition 4.2]{WiseBook}. The important properties of special cube complexes that will be used in the following are the embeddedness and $2$--sidedness of the hyperplanes and the fact that hyperplane subgroups of special cube complexes are separable (see Proposition~\ref{Prop: hyperplane separability}).

A group is \textbf{special} if it is the fundamental group of a special cube complex. 
By work of Haglund and Wise \cite{HW08}, compact special groups embed into right angled Artin groups and are hence residually finite. 
Recall that if $G$ is a group and $H$ is a subgroup, $H$ is \textbf{separable in $G$} if it is the intersection of the finite index subgroups containing $H$. 


Passing to finite index subgroups is compatible with separability:
\begin{lemma}\label{Lem: findex separability}
Let $G$ be a group, let $G_0$ be a finite index subgroup of $G$ and let $H\le G$. Then $H$ is separable in $G$ if and only if $H\cap G_0$ is separable in $G_0$. 
\end{lemma}


\begin{theorem}[Scott's Criterion, \cite{Scott78})]\label{Scott}
Let $X$ be a connected complex, $G=\pi_1X$ and $H\le G$. Let $p:X^H\to X$ be the cover corresponding to $H$. The subgroup $H$ is separable in $G$ if and only if for every compact subcomplex $Y\subseteq X^H$, there exists an intermediate finite cover $X^H\to\hat{X}\to X$ such that $Y\hookrightarrow \hat{X}$. 
\end{theorem}

Every finitely generated subgroup of a free group is separable. Likewise, special groups have an ample supply of separable subgroups.
For example, the hyperplane subgroups of a special cube complex are separable:
\begin{proposition}\label{Prop: hyperplane separability}
Let $X$ be a virtually special compact and non-positively curved cube complex. Let $W$ be a hyperplane of $X$. Then $\pi_1(W)$ is separable in $\pi_1(X)$.
\end{proposition}

Proposition~\ref{Prop: hyperplane separability} follows from Haglund and Wise's canonical completion and retraction (see \cite[Construction 4.12]{WiseBook} or \cite[Corollary 6.7]{HW08}). 


\subsection{Elevations and $R$--embeddings}\label{Sec: Elevations}

This subsection builds up the technical tools and terminology used to obtain finite covers whose hyperplanes elevate to sufficiently separated images in the universal cover. 

The first step is to formalize the notion of an elevation:
\begin{definition}
Let $W$ be a connected topological space and let $\phi:W\to Z$ be a continuous map. Let $p:\hat{Z}\to Z$ be a covering map. There is a minimal covering $\hat{p}:\hat{W}\to W$ such that $\phi\circ \hat{p}$ lifts to a map $\hat{\phi}: \hat{W} \to\hat{Z}$. The map $\hat\phi$ is an \textbf{elevation of $W$ to $\hat{Z}$}. 

Often, the map $\hat{W}\to \hat{Z}$ will be implied and an elevation of $\phi$ will instead refer to the image of some elevation. 

Elevations may not be unique: two elevations of the same map are \textbf{distinct} if they have different images.
\end{definition}
When $\phi:W\to Z$ is an inclusion map, then the distinct elevations of $\phi$ are precisely the components of $p\inv(W)$. 
\begin{definition}
Let $X$ be a metric space, $R\ge 0$ and let $Y\subseteq X$ be connected. Let $p:X^Y\to X$ be the covering space associated to $\pi_1(Y)$ so that the inclusion $Y\hookrightarrow X$ lifts canonically to $X^Y$. The subspace $Y$ is \textbf{$R$--embedded in $X$} if $p$ is injective on $\mathcal{N}_R(Y)\subseteq X^Y$. 
\end{definition}

The following lemma is straightforward but will be important:
\begin{lemma}\label{R-embed findex}
Let $p:\hat{X}\to X$ be a finite regular cover. If $A$ is $R$--embedded in $X$, then each component of $p\inv(A)$ is $R$--embedded in $\hat{X}$. 
\end{lemma}

The main application of hyperplane separability is to show that every compact virtually special cube complex has a finite cover where every hyperplane is $R$--embedded. 
\begin{proposition}\label{Prop: R-embedded}
Let $X$ be a compact non-positively curved cube complex, and let $V_1,V_2,\ldots,V_n$ be hyperplanes of $X$ so that $\pi_1 V_i$ is separable in $\pi_1X$. Given $R\ge 0$, then there exists a finite regular cover $C$ such that $V_1,\ldots,V_n\subseteq C$ are $R$--embedded in $C$. 

If $\tilde{W}_1,\,\tilde{W}_2$ are distinct elevations of a hyperplane $V$ of $C$ to the universal cover $\tilde{X}$, then $d_{\tilde{X}}(\tilde{W_1},\tilde{W_2}) \ge 2R.$ 
\end{proposition}

\begin{proof}
For each hyperplane $W$ of $X$, $\pi_1(W)$ is separable by Proposition~\ref{Prop: hyperplane separability}. By Theorem~\ref{Scott}, there exists a finite covering $\hat{p}:\hat{X}\to X$ such that there is an embedding $i_W:\mathcal{N}_R(W) \hookrightarrow \hat{X}$. 

Let $\tilde{p}:\tilde{X}\to X,$ $p^W: \tilde{X}^W\to X$ and $p:\tilde{X}\to X^{W}$ be canonical covering maps so that $\tilde{p} = p^W\circ p$. 
Let $\tilde{W}\to \tilde{W}_1,\tilde{W}\to\tilde{W}_2$ be distinct elevations of $W$ to $\tilde{X}$, 
 and let $\tilde{w}_1\in \tilde{W}_1$ and $\tilde{w}_2\in \tilde{W}_2$. 

Suppose toward a contradiction there exists a path $\gamma\subseteq\tilde{X}$ with $|\gamma|\le 2R$ between $\tilde{W}_1$ and $\tilde{W}_2$. Let $\tilde{x}\in \gamma$ such that $d(\tilde{x},\tilde{W}_1)<R$ and $d(\tilde{x},\tilde{W}_2)<R$. 

There exists $g\in \pi_1(X)$ such that $g\cdot \tilde{w}_1\in \tilde{W}_2$, and $g\notin\pi_1(W)$ because otherwise $g\cdot \tilde{w}_1\in W_1\cap W_2$ in which case, $\tilde{w}_1\in \tilde{W}_2$ but $\tilde{w}_1\notin \tilde{W}_2$. 
Now $d(g\cdot \tilde{x},\tilde{W}_2)\le R$. 
 Since $g\notin\pi_1(W)$, $p(\tilde{x})\ne p(g\cdot \tilde{x})$.  
By definition of an elevation, $p(\tilde{W}_2)$ is contained in the image of an inclusion of $W$ into $X^W$. 
Also $p(\tilde{x}),\,p(g\cdot \tilde{x})$ lie in an $R$--neighborhood of the image of $W$ in $X^W$. However, 
\[p^W \circ p(\tilde{x}) = \tilde{p}(\tilde{x}) =\tilde{p}(g\cdot \tilde{x}) = p^W\circ p(g\cdot \tilde{x}) \] contradicting the fact that $i_W:\mathcal{N}_R(W) \hookrightarrow \hat{X}$ is an embedding. 

Suppose $X$ has $n$ hyperplanes. 
By passing to a finite cover if necessary, assume $X^W$ is regular. The number of hyperplane orbits under deck transformations of $X^W$ is at most $n$, and every hyperplane in the orbit of an elevation of $W$ to $X^W$ is $R$--embedded. Therefore, performing this procedure at most $n$ times, will produce a finite cover $C\to X$ where every hyperplane is $R$--embedded. 
\end{proof}

Proposition~\ref{Prop: R-embedded} will be used later in Section~\ref{Hierarchy section} to make the elevations of a hyperplane a $2R$--separated family in the sense of Definition~\ref{Def: separated collection}.

\subsection{Convex Cores}\label{Sec: Convex Cores}

Specialness also plays a role in building a geometric representation of the peripheral structure. In the hyperbolic case, Wise and others (see \cite{Haglund2008},\,\cite{SW2004}), see also \cite[Proposition 7.2]{HW08} proved that quasiconvex subgroups of virtually special groups have ``convex cores'' in the \CAT$(0)$ universal cover. 
This fact and canonical completion and retraction can be used to show that hyperbolic special groups are \textbf{QCERF} or \textbf{quasiconvex extended residually finite} \cite[Theorem 1.3]{HW08} meaning that if $G$ is hyperbolic and special, then every quasiconvex subgroup of $G$ is separable.

A similar result exists in the relatively hyperbolic case. One might imagine that replacing the quasiconvex subgroup $H$ by a relatively quasiconvex subgroup might yield a generalization; however, some care is required. In particular, a subgroup may stabilize a quasiconvex subset of a CAT(0) cube complex but may fail to stabilize a convex proper subcomplex, see Example~\ref{E: why full good}. 

\begin{definition}
If $\tilde{X}$ is a CAT(0) cube complex and $\tilde{Y}\subseteq \tilde{X}$, the \textbf{cubical convex hull} of $\tilde{Y}$ is the smallest convex sub\emph{complex} of $\tilde{X}$ containing $\tilde{Y}$. 
\end{definition}

\begin{example}\label{E: why full good}
Take the standard action of $\Z^2 = \cyc{(1,0),(0,1)}$ on $\reals^2$ by translation. The diagonal $D:=\{(r,r):r\in\reals \}$ is a subspace stabilized by $L:=\cyc{(1,1)}\le \Z^2$. The subgroup $L$ is $(2,0)$--quasi-isometrically embedded in the given presentation of $\Z^2$, but the cubical convex hull of $D$ is all of $\reals^2$.
\end{example}

\textbf{Full relatively quasiconvex subgroups} eliminate these pathologies:
\begin{definition}[{\cite[Section 4]{SW2015}}]
Let $(G,\mathcal{P})$ be a relatively hyperbolic group pair and let $H$ be a relatively quasiconvex subgroup of $G$. The subgroup $H$ is a \textbf{full relatively quasiconvex subgroup} of $G$ if for each $g\in G$ and $P\in \mathcal{P}$, either $gPg\inv \cap H$ is finite or $gPg\inv\cap H$ is finite index in $gPg\inv$. 
\end{definition}

\begin{theorem}[{\cite[Theorem 1.1]{SW2015}}]\label{convexCores}
Let $X$ be a compact non-positively curved cube complex with $G=\pi_1(X)$ hyperbolic relative to subgroups $P_1,\ldots,P_n$. Let $\tilde{X}$ be the \CAT$(0)$ universal cover of $X$. If $H$ is a full relatively quasiconvex subgroup of $G$, then for any compact $U\subseteq \tilde{X}$, then there exists an $H$-cocompact convex subcomplex $\tilde{Y}\subseteq \tilde{X}$ with $U\subseteq \tilde{Y}$. 
\end{theorem}

By Proposition~\ref{peripheral qc}, if $(G,\mathcal{P})$ is a relatively hyperbolic group pair, the elements of $\mathcal{P}$ and their conjugates are relatively quasiconvex. 
By Proposition~\ref{Prop: coset separation}, the elements of $\mathcal{P}$ and their conjugates are full relatively quasiconvex. Therefore:
\begin{lemma}\label{Lem: peripheral protocomplexes}
Let $X$ be a non-positively curved cube complex with \CAT$(0)$ universal cover $\tilde{X}$ and $G:=\pi_1(X)$. Let $(G,\mathcal{P})$ be a relatively hyperbolic pair. Let $x\in\tilde{X}$ be a base point in the universal cover. For each $P\in \mathcal{P}$, there exists a $Z'(P,x)$ such that $Z'(P,x)$ is a $P$-cocompact convex subcomplex of $\tilde{X}$ containing $x$. 
\end{lemma}
It follows immediately that there exists a $Q\ge 0$ such that the cubical convex hull of $Px$ is contained in $\mathcal{N}_Q(Px)$.

\section{A Malnormal Quasiconvex Fully $\mathcal{P}$-Elliptic Hierarchy}\label{Hierarchy section}

For the following section, let $X$ be a compact non-positively curved cube complex with \CAT$(0)$ universal cover $\tilde{X}$ and $G = \pi_1(X)$ hyperbolic relative to subgroups $\mathcal{P}:= \{P_1,\ldots,P_n\}$. Fix a base point $x\in\tilde{X}$.
By Lemma~\ref{Lem: peripheral protocomplexes}, there is a convex subcomplex $\tilde{Z}'_{P,x}$ that is a $P$--cocompact convex subcomplex of $\tilde{X}$ containing $Px$. 


Let $\mathcal{B}_0 :=\{g\tilde{Z}'_{P,x}:\,g\in G,\,P\in\mathcal{P}\}$.  
By Proposition~\ref{P: situation is rel hyp pair}, there exists $f_0:\reals^{\ge 0}\to \reals^{\ge 0}$ and $\delta\ge 2$ so that $(\tilde{X},\mc{B}_0)$ is a $(\delta-2,f_0)$--relatively hyperbolic pair.

Let $\tilde{Z}_{P,x}=\mathcal{N}_{2\delta}(\tilde{Z}'_{P,x})$. Theorem~\ref{convexCores} implies that the collection $\mc{B}' = \{g\tilde{Z}_{P,x}:\,g\in G,\,P\in\mc{P}\}$ is a thickening of $\mc{B}_0$. Proposition~\ref{P: fattening peripherals} implies there exists $f':\reals^{\ge 0}\to \reals^{\ge 0}$ so that $(\tilde{X},\mc{B}')$ is a $(\delta-2,f')$--\CAT$(0)$ relatively hyperbolic pair. We also define $f:\reals^{\ge 0}\to \reals^{\ge 0}$ where $f(r) = f'(r+2)$. The function $f$ will be useful later when we carry out the  augmentation construction defined in Section~\ref{S: superconvexity augmented}.

To maintain consistency with previous notation, we will use the notation $M=f(6\delta)$ throughout Section~\ref{Hierarchy section}. Proposition~\ref{New rel hyp pair} implies:
\begin{proposition}\label{cubical attractiveness}
For every $g\in G$, $g\tilde{Z}_{P,x}$ is $(3M+6R+21\delta)$--attractive in the sense of Definition~\ref{Def: attractive}.
\end{proposition}

\subsection{Superconvexity, Peripheral Complexes and Augmented Complexes}\label{S: superconvexity augmented}

Here we will prove that bi-infinite geodesics contained in a bounded neighborhood of $\tilde{Z}_{P,x}$ actually lie in $\tilde{Z}_{P,x}$. 

\begin{definition}\label{Def: superconvexity}
Let $X$ be a non-positively curved cube complex and let $\phi:Z\to X$ be a local isometry. The map $\phi$ is \textbf{superconvex} if for any elevation $\tilde\phi:\tilde{Z}\hookrightarrow \tilde{X}$ of $Z$ to the universal cover $\tilde{X}$ of $X$ and any bi-infinite geodesic $\gamma$ in $\tilde{X}$ such $\gamma$ lies in a bounded neighborhood of (the $\tilde\phi$ image of) $\tilde{Z}$ in $\tilde{X}$, then $\gamma$ is contained (in the $\tilde\phi$ image of) $\tilde{Z}$.

If the immersion $\phi:Z\to X$ is superconvex, then $Z$ is said to be superconvex in $X$ (with respect to $\phi$).
\end{definition}

Since $\tilde{Z}_{P,x}$ is a $P$-cocompact convex subcomplex of $\tilde{X}$, the quotient $\bar{Z}_{P,x} := P\backslash \tilde{Z}_{P,x}$ is a cube complex and there is a natural local isometry $\phi_{P,x}:\bar{Z}_{P,x}\to X$ that carries $\bar{Z}_{P,x}$ to the image of $G\backslash \tilde{Z}_{P,x}$ in $X$.

\begin{proposition}\label{Prop: peripheral complexes}
$\phi_{P,x}$ is superconvex. 
\end{proposition}

\begin{proof}
Suppose $\gamma$ is a bi-infinite geodesic contained in $\mathcal{N}_R(\tilde{Z}_{P,x})$ and $p\in\gamma$.
There exist $s_1,s_2\in\gamma$ so that $p\in [s_1,s_2]$ and $d(s_i,p)> 3M+6R+21\delta$. 
Hence by Proposition~\ref{cubical attractiveness} there exist points $t_1,t_2$ so that $t_1\in [s_1,p]$ and $t_2\in [p,s_2]$ so that $t_1,t_2\in \tilde{Z}_{P,x}$. 
Therefore by convexity $p\in \tilde{Z}_{P,x}$. 
Hence $\gamma\subseteq \tilde{Z}_{P,x}$. 
\end{proof}




The complexes $\bar{Z}_{P,x}$ are called \textbf{peripheral complexes}.
There is a convenient way to upgrade the immersion to an embedding:
\begin{definition}\label{Def: augmented cube complex}
Let $X$ be a non-positively curved cube complex with \CAT$(0)$ universal cover $\tilde{X}$ and $G:=\pi_1(X)$. Let $(G,\mathcal{P})$ be a relatively hyperbolic group pair. Let $\mathcal{Z} := \bigsqcup_{P\in\mathcal{P}} \bar{Z}_{P,x}$, and let $\Phi:\mc{Z}\to X$ be the map so that $\Phi|_{\bar{Z}_{P,x}} = \phi_{P,x}$.  The \textbf{augmented cube complex for the pair $(X,\Phi)$} is the complex:
\[C(X,\Phi) := X \cup \left( \bigsqcup_{P\in\mathcal{P}} \bar{Z}_{P,x}\times [0,1]\right) / (\bar{Z}_{P,x}\times \{1\})\sim \phi_{P,x}(\bar{Z}_{P,x}),\]
consisting of the mapping cylinders of the $\phi_{P,x}$ identified along $X$.
\end{definition}

The hyperplanes $\bar{Z}_{P,x}\times \frac12$ are called \textbf{peripheral hyperplanes} while the remaining hyperplanes of $C(X,\Phi)$ are \textbf{non-peripheral}. 
Note that the non-peripheral hyperplanes of $C(X,\Phi)$ are in one-to-one correspondence with the hyperplanes of $X$. 
Since $\pi_1X \cong \pi_1(C(X,\Phi))$, a (virtual) hierarchy for $\pi_1(C(X,\Phi))$ determines a (virtual) hierarchy of $\pi_1 X$. 

\begin{proposition}\label{P: aug hyp sep}
Let $C(X,\Phi)$ be the augmented cube complex for the pair $(X,\mc{Z})$ as in Definition~\ref{Def: augmented cube complex}. If $X$ is virtually special and $W$ is a non-peripheral hyperplane of $C(X,\Phi)$, then $\pi_1W$ is separable in $\pi_1 C(X,\Phi)\cong \pi_1 X$. 
\end{proposition}

\begin{proof}[Sketch]
The natural homotopy equivalence between $C(X,\Phi)$ and $X$ that induces $\pi_1 C(X,\Phi) \cong \pi_1(X)$ brings non-peripheral hyperplanes of $C(X,\Phi)$ to hyperplanes of $X$. Therefore, $W$ is homotopy equivalent to a hyperplane $V$ of $X$ and $\pi_1 V\cong \pi_1 W$ is separable in $\pi_1X$ (recall Proposition~\ref{Prop: hyperplane separability}). 
\end{proof}

Technically, the definition of $C(X,\Phi)$ depends on the base point, but since the following results are given up to conjugacy, there is no need to keep track of base points. 
 
\begin{proposition}\label{attractiveness of images}
Let $C(X,\Phi)$ be the augmented cube complex for the pair $(X,\mathcal{Z})$ described in Definition~\ref{Def: augmented cube complex}. 
Let $\tilde{C}$ be the universal cover of $C(X,\Phi)$. 
Let $\mathcal{B}$ be the collection of (images of) elevations of (images of) $\bar{Z}_{P,x}\times [0,1]$ in $C(X,\Phi)$ to $\tilde{C}$.

The following hold:
\begin{enumerate}
\item Each $B\in\mathcal{B}$ is closed and convex,
\item $(\tilde{C},\mathcal{B})$ is a $(\delta,f)$--\CAT$(0)$ relatively hyperbolic pair, and
\item every $B\in\mathcal{B}$ is $(3M+6R+2f(R)+21\delta)$--attractive (recall Definition~\ref{Def: attractive}).
\end{enumerate}
\end{proposition}

\begin{proof}
The universal cover $\tilde{X}$ of $X$ embeds as a closed convex subset of $\tilde{C}$ so that each $B\in\mathcal{B}$ intersects $\tilde{X}$ in some $\tilde{Z}_{P,x}$. 
Since $B$ intersects $\tilde{X}$ in a closed convex subspace, $B$ is closed and convex in $\tilde{C}$. 
 
Every geodesic triangle in $\tilde{C}$ is Hausdorff distance $1$ from a geodesic triangle in $\tilde{X}$. 
Since triangles in $\tilde{X}$ are $(\delta-2)$--thin relative to translates of $\tilde{Z}_{P,x}$, triangles in $\tilde{C}$ are $\delta$ thin relative to $\mathcal{B}$.
For every $B_1,B_2$ in $\mathcal{B}$ with $B_1\ne B_2$, $\mathcal{N}_t(B_1)\cap \mathcal{N}_{t}(B_2)$ is distance at most $1$ from  the intersection of $g_1 \mc{N}_t(\tilde{Z}_{P_1,x})$ and $g_2\mc{N}_t(\tilde{Z}_{P_2,x})$ in $\tilde{X}$ for some $g_1,g_2\in G$ and $P_1,P_2\in \mathcal{P}$, so the fact that $\tilde{X}$ is a $(\delta-2,f')$--\CAT$(0)$ relatively hyperbolic pair implies that $(\tilde{C},\mathcal{B})$ is a $(\delta,f)$--\CAT$(0)$ relatively hyperbolic pair.

Let $\gamma$ be a geodesic in $\tilde{C}$ with endpoints in $\mc{N}_R(B)$ for some $B\in\mc{B}$. Since $\tilde{X}$ is \CAT$(0)$ and $B$ is convex, $\gamma\subseteq \mc{N}_R(B)$. Then $\gamma$ is either contained in $B'$ for some $B'\in\mc{B}$ in which case $|\gamma|\le f(R)$ or $\gamma$ has a subpath $\sigma$ whose endpoints in $\tilde{X}$ are at most $f(R)$ from the endpoints of $\gamma$. Therefore $|\sigma|\ge |\gamma|-2f(R)$. There is some $g\in G$ and $P\in\mc{P}$ so that $g\tilde{Z}_{P,x} = B\cap \tilde{X}$. 
If the length of $\sigma$ is at least $3M+6R+21\delta$, then $\sigma\cap g\tilde{Z}_{P,x}\ne \emptyset$ by Proposition~\ref{cubical attractiveness}. 
Therefore, if the length of $\gamma$ is at least $3M+6R+2f(R)+21\delta$, $\emptyset \ne\gamma\cap g\tilde{Z}_{P,x}\subseteq \gamma\cap B$. 
\end{proof}

\subsection{The Double Dot Hierarchy}\label{S: double dot}

The construction of a hierarchy will use a finite cover called the \textbf{double dot cover} whose construction is originally due to Wise {\cite[Construction 9.1]{WiseManuscript}}. This treatment of the double dot cover is similar to the one in \cite[Section 5]{AGM}.
\begin{definition}[{\cite[Construction 9.1]{WiseManuscript}}]
Let $X$ be a cube complex, let $W\subseteq X$ be a hyperplane of $X$. Let $\gamma$ be a based loop and let $[\gamma]\in \pi_1X$. Then $[\gamma]$ has a well defined (mod $2$) intersection number with $W$.
Let $\mathcal{W}$ be the set of embedded, $2$--sided, non-separating hyperplanes of $X$.
For each $W\in\mc{W}$ let $i_W:\pi_1 X\to\zmodnz2$ be the algebraic intersection map and define: 
\[\Psi:\pi_1X\to \bigoplus_{W\in \mathcal{W}}\zmodnz2 \qquad \Psi = \bigoplus_{W\in \mathcal{W}} i_W\]

The \textbf{double dot cover} of $X$ is the cover corresponding to the subgroup $\ker\Psi\le \pi_1X$. 
\end{definition}

The double dot cover of a cube complex is usually a high degree cover. Therefore, constructing examples can be quite difficult. Fortunately, the double dot cover of a rose with 2 petals is easy to construct:
\begin{example}
See Figure~\ref{Fig: fig8} for the double dot cover of the figure--$8$ loop. 
\begin{figure}
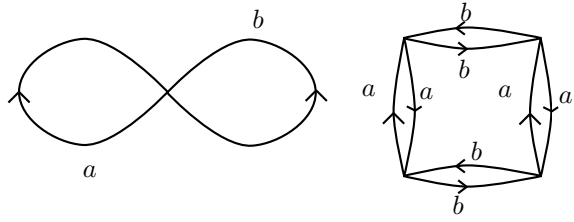
\caption{The figure--$8$ loop on the left whose two hyperplanes are the two edge midpoints and the double dot cover of the figure--$8$ loop on the right.}
\label{Fig: fig8}
\end{figure}
\end{example}

An important feature of the double dot cover is that the cover is taken over non-separating hyperplanes. 
This serves two purposes: first, making sure that double dot cover is not trivial and second, making sure that the double dot hierarchy constructed later has non-trivial splittings. There is a way to obtain a complex where every hyperplane is non-separating:
\begin{theorem}[{\cite[Proposition 2.12]{Bregman16}}]\label{BregmanTheorem} Let $X$ be a compact special NPC cube complex, then $X$ is homotopy equivalent to a compact special NPC cube complex whose hyperplanes are all non-separating.
\end{theorem}

Let $X$ be a special cube complex with finitely many hyperplanes $\mathcal{W} := \{W_1,\ldots,W_n\}$ where every hyperplane is non-separating and let $\ddot{p}_X:\ddot{X}\to X$ be the double dot cover of $X$. The hyperplanes of $\ddot{X}$ are elevations of hyperplanes of $X$, and they divide $\ddot{X}$ in a natural way.
Let $x\in\ddot{X}\setminus \bigcup \ddot{p}_X\inv(\mathcal{W})$ be a base vertex. 

Each component of $\ddot{X}\setminus \bigcup \ddot{p}_X\inv(\mathcal{W})$ contains a lift of $\ddot{p}_X(x)$ because the hyperplanes of $X$ are non-separating.
There is only one lift of $\ddot{p}_X(x)$ which lies in each component of $\ddot{X}\setminus \bigcup \ddot{p}_X\inv(\mathcal{W})$ because otherwise there is path $\nu$ between two points of $\ddot{p}_X\inv(x)$ that does not cross $\ddot{p}_X\inv(\mathcal{W})$. The path $\nu$ must project to a loop that represents a non-identity element of $\pi_1(X)\setminus \ker \Psi$ but does not cross any $W\in\mc{W}$ which is impossible. 

Since $\ker\Psi$ is normal, $\pi_1X/\ker\Psi$ acts by deck transformations on $\ddot{X}$. This action induces a free and transitive action on $\ddot{p}_X\inv(x)$. Since each component of $\ddot{X}\setminus \bigcup \ddot{p}_X\inv(\mathcal{W})$ contains exactly one element of $\ddot{p}_X\inv(x)$, we can label each of the components of $\ddot{X}\setminus \bigcup \ddot{p}_X\inv(\mathcal{W})$ by an element of $\pi_1X/\ker\Psi\cong \bigoplus_{W\in\mc{W}} \zmodnz2$. 

 With data specified below in Hypotheses~\ref{H: double dot}, we will use the labels for components of $\ddot{X}\setminus \bigcup \ddot{p}_X\inv(\mathcal{W})$ to construct a \textbf{double dot hierarchy} of spaces for the double dot cover $\ddot{C}$ of $C$.
When the data in Hypotheses~\ref{H: double dot} satisfy certain criteria discussed in Section~\ref{improved hierarchy}, the double dot hierarchy gives rise to a quasiconvex and fully $\mathcal{P}$-elliptic hierarchy of groups for $\pi_1(\ddot{C})$ which is isomorphic to a finite index subgroup of $\pi_1X$. Passing to a particular finite cover will produce an induced hierarchy that is also malnormal. The next several paragraphs outline the construction of the double dot hierarchy as it is presented in \cite[Section 5]{AGM}.    

We now establish some baseline hypotheses for the remainder of Section~\ref{S: double dot}. 
\begin{hypotheses}\label{H: double dot}
Let $X$ be a compact special NPC cube complex so that: 
\begin{itemize}
\item the hyperplanes of $X$ are non-separating,
\item there exist a disjoint union  $\mc{Z} := \bigsqcup_{i=1}^n \bar{Z}_i$ of NPC cube cube complexes together with a local isometric immersion $\Phi:\mc{Z}\to X$, 
\item let $C$ be the augmented cube complex $C(X,\Phi)$ and let $p:\ddot{C}\to C$ be its double dot cover, and
\item let $\mc{W}$ be the non-peripheral hyperplanes of $C$ and choose an ordering of the elements of $\mc{W}$ so that they are $W_1,W_2,\ldots,W_n$. 
\item Additionally, $C$ is a mapping cylinder for the map $\Phi$, so we can view $\mc{Z}$ as an embedded subspace of $C$. In the language of Definition~\ref{Def: augmented cube complex}, $\mc{Z}$ is the image of $\bigsqcup_{i=1}^n \bar{Z}_i\times \{0\}$ in $C$.   
\item Let $\ddot{\mc{Z}} = p\inv(\mc{Z})$ be the preimage of $\mc{Z}\subseteq C$ under the double dot covering map. 
\item Fix a base vertex. 
\end{itemize}
\end{hypotheses}

Each component of $\ddot{C}\setminus p\inv\left(\bigcup \mathcal{W}\right)$ is labeled (relative to the base vertex) by a vector $\hat{t}\in \bigoplus_{i=1}^n \zmodnz2$.  
For each $1\le i \le n$, let $\mathcal{W}_i$ be the first $i$ hyperplanes, let $M_i = \bigoplus_1^i \zmodnz2$. Then the complementary components of $\bigcup \mathcal{W}_i$ are labeled by elements of $M_i$. 
For each $\hat{t}\in M_i$, let $K_{\hat{t}}$ be the closure of the part labeled by $\hat{t}$. 

For each $\hat{t}\in M_i$, a \textbf{$\hat{t}$-vertex space at level $n-i+1$} is a component of $K_{\hat{t}}\cup \ddot{\mathcal{Z}}$ that intersects $K_{\hat{t}}$. 
In the construction of the double dot hierarchy, the $\hat{t}$-vertex spaces at level $n-i+1$ specify all of the vertex spaces at each level, but the actual graph of spaces structure at each level must be described.

If $A$ is the closure of a component of $p\inv(W_i)\setminus \bigcup_{j<i} p\inv(W_j)$, then $A$ is called a \textbf{partly-cut-up elevation of $W_i$}. The double dot hierarchy is constructed by cutting along an elevation of a hyperplane $W_i$ to $\ddot{C}$ and any elements of $\ddot{\mathcal{Z}}$ that intersect $W_i$, but the elevation of the hyperplane $W_i$ may have already been cut by one of the other hyperplane elevations of $W_j$ with $j<i$. 

By construction, any two $\hat{t}$-vertex spaces at level $n-i+1$ are either disjoint or intersect in a union of components of $\ddot{\mathcal{Z}}$ and disjoint partly-cut-up elevations of $W_i$.

Now it is time to construct the graph of spaces structures at each level. Let $\hat{t}\in M_i$ and let $V$ be the corresponding $\hat{t}$-vertex space at level $n-i+1$. 
 Consider the canonical projection $\pi:M_{i+1}\to M_i$, let $\hat{t}^+$ and $\hat{t}^-$ be the preimages of $\hat{t}$ under $\pi$. Let $V^+$ and $V^-$ be the collections of complementary components of $V\setminus p\inv\left(\bigcup \mc{W}_{i+1}\right)$ labeled by $\hat{t}^+$ and $\hat{t}^-$ respectively. 
Then $V = V^+\cup V^-$ and the components in $V^+,V^-$ will serve as the vertex spaces in the graph of spaces decomposition of $V$ in this hierarchy. 

The edge spaces are components of $V^+\cap V^-$.  
The attaching maps are the inclusion maps of edge spaces into vertex spaces while the realization is provided by a homotopy equivalence collapsing the mapping cylinders of the edge spaces onto the images of the edge spaces.
 
Let $\hat{t}\in M_n$. Then the components of the $\hat{t}$-vertex spaces are the vertex spaces of level $1$ of the hierarchy, so the terminal spaces of the hierarchy are precisely these spaces.

\begin{definition}\label{dbl dot hierarchy}
The hierarchy $\mathcal{H}$ constructed in the preceding paragraphs with vertex spaces is called the \textbf{double dot hierarchy for the pair $(X,\mathcal{Z})$}. 
\end{definition}
The double dot hierarchy actually depends on an ordering on the hyperplanes, but the applications that follow only need an existence of a hierarchy given some local isometric immersion $\mathcal{Z}\to X$, so this complication will be henceforth ignored. 

A version of the double dot hierarchy exists for general NPC cube complexes, see \cite[Section 5.2]{AGM}; however, the double dot hierarchy may fail to be faithful and even if it faithful, the hierarchy may fail to be quasiconvex or malnormal. Also, the terminal spaces may not be useful. However, when hyperplanes are embedded, nonseparating and two-sided, the terminal spaces are easy to understand:
\begin{lemma}[{\cite[Lemma 5.5]{AGM}}]\label{Lem: terminal space decomposition}
Assume Hypotheses~\ref{H: double dot}. 
If $Y$ is a terminal space of the double dot hierarchy for $(X,\mathcal{Z})$, then $Y$ has a graph of spaces structure $(\Gamma,\chi)$ such that:
\begin{enumerate}
\item $\Gamma$ is bipartite with vertex set $V(Y) = V(Y)^+\sqcup V(Y)^-$,
\item if $v\in V(Y)^+$, $\chi(v)$ is contractible,
\item if $v\in V(Y)^-$, $\chi(v)$ is a component of $\ddot{\mathcal{Z}}$ and
\item every edge space is contractible. 
\end{enumerate}
\end{lemma}



\begin{cor}\label{C: terminal spaces} 
Under Hypotheses~\ref{H: double dot}, the fundamental group of a terminal space of the double dot hierarchy is a free product of the form $(\bigast_{i=1}^p G_i)* F$ where $F$ is a finitely generated free group and each $G_i := \pi_1(Z_i)$ where $Z_i$ is a component of $\mathcal{Z}$. 
\end{cor}

\subsection{A fully $\mathcal{P}$-elliptic malnormal quasiconvex hierarchy}\label{improved hierarchy}

\begin{hypotheses}\label{H: cleanup}
We set some basic hypotheses and notation for Section~\ref{improved hierarchy}:
\begin{enumerate}
\item Let $X_0$ be a NPC compact special cube complex,
\item let $X$ be a NPC compact special cube complex that is homotopy equivalent to $X$ so that the hyperplanes of $X$ are all non-separating (the existence of $X$ follows from Theorem~\ref{BregmanTheorem}),
\item let $\tilde{X}$ be the universal cover of $X$ with base point $x\in \tilde{X}$ that does not lie in any hyperplane, and
\item let $G:=\pi_1X\cong\pi_1X_0$ and suppose that $(G,\mathcal{P})$ is a relatively hyperbolic group pair.
\item For each $P\in\mc{P}$, let $\phi_{P,x}:Z_P\to X$ be the superconvex local isometric immersions and let $\mc{Z} = \sqcup Z_P$ that arise as a consequence of Proposition~\ref{Prop: peripheral complexes}. Let $\Phi:\mc{Z}\to X$ be the map that restricts to $\phi_{P,x}$ on $Z_{P}$.
\item Let $C_1 = C(X,\Phi)$ be the augmented cube complex for $(X,\Phi)$ (recall Definition~\ref{Def: augmented cube complex}), and let $\tilde{C}$ be its universal cover. 
\item Viewing $C_1$ as a mapping cylinder of $\Phi$, $\Phi$ gives rise to a natural embedding $\mc{Z}\hookrightarrow C_1$. We call the components $Z_P\times\{0\}$ of the image of $\Phi$ \textbf{peripheral spaces}.  \label{I: peripheral space}
\end{enumerate}
\end{hypotheses}
By strategically passing to finite covers and building the double dot hierarchy, we  will produce a faithful, quasiconvex and fully $\mathcal{P}-$elliptic virtual  hierarchy for $\pi_1X$. 

\begin{lemma}[{See \cite[Lemma 5.18]{AGM}}]\label{L: cover augmentation preservation}
Let $C'$ be a finite regular cover of $C_1$. Then:
\begin{enumerate}
\item There exists a finite cover $X'$ of $X$ with $G':=\pi_1X'$ and a superconvex local isometric immersion $\Phi':\mathcal{Z}'\to X'$ such that $(G',\mathcal{P}')$ is the induced relatively hyperbolic group pair (see Proposition~\ref{Prop: induced peripheral structure}) and $C'$ is the augmented cube complex of the pair $(X',\mathcal{Z}')$. 
The components of $\mathcal{Z}'$ have fundamental group isomorphic to elements of $\mathcal{P}'$ and for each component $Z$ of $\mathcal{Z}'$, the image of $\pi_1Z$ is conjugate to an element of $\mathcal{P}'$ in $G'$. 
\item Every nonperipheral hyperplane of $C'$ is nonseparating. 
\end{enumerate}
\end{lemma}


%


\begin{notation}\label{N: Sec73 constants}
We now set some notation and constants:
\begin{enumerate}
\item Let $\mc{B}$ be the collection of elevations of $Z_P\times [0,1]$ (as determined by the mapping $\phi_{P,x}$) to $\tilde{C}$. Let $\tilde{\mc{Z}}$ be the union of the elements of $\mc{B}$ in $\tilde{C}$. 
\item Recall from Proposition~\ref{attractiveness of images} that there exist $(\delta,f)$ so that $(\tilde{C},\mc{B})$ is a $(\delta,f)$--\CAT$(0)$ relatively hyperbolic pair. \label{I: rel pair setup}
\item Let $M = f(6\delta)$, let $\lambda = 4$ and $\epsilon = 10000(M+\delta+1)$. 
\item Proposition~\ref{attractiveness of images} also implies that every $B\in\mc{B}$ is $(3M+6R+2f(R)+21\delta)$--attractive. \label{I: attractiveness setup} 
\item Set $L_{rftp}$ so that every pair of $(\lambda,\epsilon)$--quasigeodesics in $\tilde{C}$ $(L_{rftp},L_{rftp})$--fellow travel relative to $\mathcal{B}$ (recall Definition~\ref{D: RFTP} and Theorem~\ref{Prop: rel fellow traveling}).
\item Let $R_{rftp}= \lambda(\lambda(3f(L_{rftp})+\epsilon+2L_{rftp})+\epsilon)+2f(L_{rftp})$. 
\item Let $R_0 > \max\{4,R_{rftp},500M+10000\delta\}$. \label{I: R-embedding constant}
\end{enumerate}
\end{notation}  

\begin{observation}\label{P: C twiddle setup}
The constants established in items~(\ref{I: rel pair setup}) and (\ref{I: attractiveness setup}) of Notation~\ref{N: Sec73 constants} ensure that the pair $(\tilde{C},\mc{B})$ satisfies Hypotheses~\ref{baseline combo lemma}. 
\end{observation}

Using Propositions~\ref{Prop: R-embedded} and~\ref{P: aug hyp sep}, let $C_2$ be a finite regular cover of $C_1$ such that every non-peripheral hyperplane of $C_2$ is $R_0$--embedded and non-separating. 
Then $C_2$ is the augmented cube complex of a pair $(X_2,\mathcal{Z}'')$ where $X_2$ is a finite cover of $X$ by Lemma~\ref{L: cover augmentation preservation}. 
 Recall that $\tilde{X}$ naturally embeds in $\tilde{C}$, which is also the universal cover of $C_2$. 
Let $G_2 =\pi_1(C_2)$ and let $(G_2,\mathcal{P}'')$ be the induced peripheral structure.

Let $c:\ddot{C}_2\to C_2$ be the double dot cover of $C_2$. Let $(\ddot{G}_2,\ddot{\mathcal{P}}'')$ be the induced peripheral structure on $\ddot{G}_2: = \pi_1 \ddot{C}_2$. 
The next few statements will show that the double dot hierarchy on $\ddot{C}_2$ is faithful, quasiconvex and fully $\ddot{\mathcal{P}}''$-elliptic hierarchy for $\pi_1\ddot{C}_2$.
Passing to a finite regular cover will later yield a hierarchy which is also malnormal.

By Lemma~\ref{L: cover augmentation preservation}, $\ddot{C}_2$ is an augmented cube complex with respect to a pair $(\ddot{X}_2,\ddot{\mathcal{Z}}_2)$ where $\ddot{\mathcal{Z}}_2$ consists of components of $c\inv(\mathcal{Z}'')$. 
Let $E$ be an edge space of the double dot hierarchy on $\ddot{C}_2$. Then $E$ is a union of partly-cut-up elevations of a hyperplane of $C_2$ and elements of $\ddot{\mc{Z}}_2$. 

Recall that $(\tilde{C},\mathcal{B})$ is a $(\delta,f)$--\CAT$(0)$ relatively hyperbolic pair. 
Let $\tilde{E}$ be an elevation of $E$ to $\tilde{C}$. 
There exist $\mathcal{A}_E$ and $\mathcal{B}_E$ so that $\mathcal{A}_E$ is a collection of elevations to $\tilde{C}$ of convex partly-cut-up hyperplane elevations of $W$ and $\mathcal{B}_E$ is a collection of elevations of the peripheral spaces (recall Hypotheses~\ref{H: cleanup}\eqref{I: peripheral space}) to $\tilde{C}$ so that $\tilde{E}$ is the union of the elements of $\mc{A}_E$ and $\mathcal{B}_E$.

Each element $B_E\in \mc{B}_E$ is an elevation of a peripheral space. While $B_E$ is not an element of $\mc{B}$, there is a unique $B_E'\in\mc{B}$ containing $B_E$. In particular, $B_E'$ is the $1$--neighborhood of $B_E$ in $\tilde{C}$. 
Let $\mathcal{B}_E' = \{B\in\mc{B}:\, B_E\subseteq B\text{ for some } B_E\in\mc{B}_E\}$ be the collection of elevations of the $Z_P\times [0,1]$ to $\tilde{C}$ whose intersection with $\tilde{X}$ is some $B_E\in\mc{B}_E$. 
Let $\tilde{E}'$ be the image of $\left(\bigsqcup \mathcal{A}_E \right) \sqcup \left(\bigsqcup \mathcal{B}_E' \right)$ in $\tilde{C}$. 

\begin{figure}\label{F: elevations}
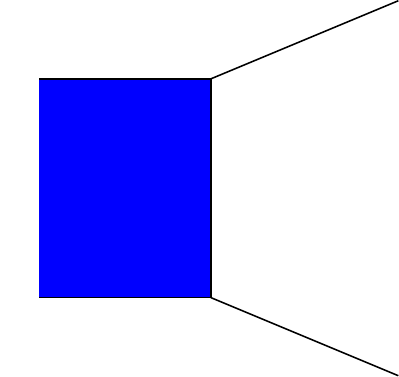
\caption{A schematic diagram showing the relationship between $B_E,\,B_E'$ and their attachment to $\tilde{X}$. The closure of the shaded region is $B_E'$.}
\end{figure}

By Observation~\ref{P: C twiddle setup}, the $R_0$--embeddedness of the hyperplane $W$ and the construction of $\tilde{E}'$ imply:
\begin{proposition}\label{P: 5.16 satisfied}
The subspace $\tilde{E}' \subseteq \tilde{C}$ is a subspace of the form specified by Hypotheses~\ref{Hyp: strong qc}. 
\end{proposition}

\begin{proof}
Observation~\ref{P: C twiddle setup} ensures $(\tilde{X},\mc{B})$ satisfies Hypotheses~\ref{baseline combo lemma}. 

Recall that $C_2$ has hyperplanes that are $R_0$--embedded (recall $R_0$ from Notation~\ref{N: Sec73 constants}\eqref{I: R-embedding constant}), and recall that $R_0$ embeddedness of hyperplanes is preserved by finite covers (Lemma~\ref{R-embed findex}). 
Therefore, for all distinct pairs of $A_1,A_2\in\mathcal{A}_E$, $d(A_1,A_2)\ge 2 R_0$, and $R_0$ is large enough to provide the separation between elements of $\mc{A}_E$ required by Hypotheses~\ref{Hyp: strong qc}. 

By construction, $\mc{B}_E'\subseteq \mc{B}$, and $\tilde{E}'$ is glued together from elements of $\mc{A}_E$ and $\mc{B}_E'$ as required. 
\end{proof}

\begin{proposition}\label{Prop: faithful}
Let $E$ be an edge space of the double dot hierarchy on $\ddot{C}_2$. Then the map $E\to \ddot{C}_2$ is $\pi_1$ injective. 
\end{proposition}

\begin{proof}
Suppose not toward a contradiction. 
Then there exists a loop $\gamma$ in $E$ such that $\gamma$ is essential in $E$ but has trivial image in $\pi_1(\ddot{C}_2)$. 
Since $\gamma$ is $\pi_1$ trivial in $\pi_1(\ddot{C}_2)$, $\gamma$ elevates to a loop $\tilde{\gamma}\subseteq \tilde{E}$ in $\tilde{C}$.
Since $\tilde{E}$ is homotopy equivalent to $\tilde{E}'$, there is a loop $\gamma'$ in $\ddot{C}_2$ that is the image of a geodesic in $\tilde{E}'$. 
Since $\tilde{E}'$ is the image of $\left(\bigsqcup \mathcal{A}_E \right) \sqcup \left(\bigsqcup \mathcal{B}_E' \right)$ in $\tilde{C}$, $\tilde{\gamma}'$ cannot be a loop by Proposition~\ref{Prop: edge space pi1 injective}.
\end{proof}

The next step is to prove that the double dot hierarchy on $\ddot{C}_2$ is quasiconvex:
\begin{proposition}\label{Prop: QC hierarchy}
Recall $\lambda,\epsilon$ from Notation~\ref{N: Sec73 constants}. If $E$ is an edge space of the double dot hierarchy on $\ddot{C}_2$ and $\tilde{E}$ is the universal cover of $E$, then any elevation $\tilde{E}\hookrightarrow \tilde{C}$ of $E$ to $\tilde{C}$ is a $(\lambda,\epsilon)$--quasi-isometric embedding. 
\end{proposition}

\begin{proof}
Let $\gamma$ be a geodesic in $\tilde{E}$ and let $\gamma'$ be a geodesic with the same endpoints in $\tilde{E'}$.  Let $\gamma''$ be a geodesic in $\tilde{C}$ with the same endpoints as $\gamma$. Proposition~\ref{P: 5.16 satisfied} implies we can use Proposition~\ref{Lem: path classes}, which implies that $\gamma'$ is a $(2,114M+1592\delta)$--quasigeodesic in $\tilde{C}$. Let $n$ be the smallest number so that $\gamma'$ can be written as $a_1b_1\ldots b_na_{n+1}$ where:
\begin{enumerate}
\item $a_i$ can be a point if $i=1$ or $i=n+1$, 
\item otherwise $a_i$ is geodesic in some $A_i\in\mc{A}_E$, and
\item $b_i$ is geodesic in some $B_i' \in\mc{B}_E'$. 
\end{enumerate}
The endpoints of each $b_i$ lie in $\tilde{E}$ because every $A_i\subseteq \tilde{E}$ and $\gamma,\gamma'$ have the same endpoints. 
Thus each $b_i$ can be replaced by a path of length $|b_i|+2$ that lies entirely in $\tilde{E}$. It is therefore possible to produce a path in $\tilde{E}$ between the endpoints of $\gamma$ whose length is at most $|\gamma'|+2n$, so $|\gamma|\le |\gamma'|+2n$. 
Further, $|b_i|\ge R_0 \ge 4$ for $1<i <n$ because the $A_i$ are $R_0$--separated, so we have that $|\gamma|\ge 4n-8$ which implies 
\begin{equation}
2n \le \frac12|\gamma|+4\label{E: weird gamma}
\end{equation} Therefore:
\begin{eqnarray*}
|\gamma''| & \ge & \frac12|\gamma'|-(114M +1592\delta)  \\
& \ge & \frac12(|\gamma|-2n) - (114M +1592\delta) \\ 
& \ge & \frac12(\frac12|\gamma| - 4) - (114M+1592\delta) \\
& \ge & \frac14|\gamma| - (114M+1592\delta +2)
\end{eqnarray*}
where the third line follows from the second by the estimate in \eqref{E: weird gamma}. 
Hence $\gamma$ is a $(4,114M+1592\delta+2)$--quasigeodesic in $\tilde{C}$. 
\end{proof}




Proposition~\ref{Prop: faithful} and Proposition~\ref{Prop: QC hierarchy} together yield the following:
\begin{cor}\label{DDH Good}
The double dot hierarchy induced on $\pi_1\ddot{C}_2$ is faithful and quasiconvex.
\end{cor}

The next step is to prove that the double dot hierarchy on $\ddot{C}_2$ is fully $\ddot{\mathcal{P}}''$-elliptic. 
Definition~\ref{Def: accidental loop} introduces geometric terminology for the situation where a subgroup of a relatively hyperbolic group pair $(G,\mathcal{P})$ contains an element $g$ conjugate into a peripheral subgroup $P$ such that no positive power of $g$ lies in $E\cap P$.
\begin{definition}\label{Def: accidental loop}
Let $Y$ be a locally convex subspace of $\ddot{C}_2$. Let $E\subseteq \ddot{C}_2$. The subspace $E$ has an \textbf{accidental $Y$ loop} if there exists a homotopically essential loop, $\gamma,$ which is both freely homotopic to a geodesic loop in $Y$ and has no positive power homotopic in $E$ to a geodesic loop in $Y$. 
\end{definition}

The next few statements will show that the edge spaces of the double dot hierarchy for $\ddot{C}_2$ have no accidental $\ddot{\mathcal{Z}}''$-loops. This will imply the hierachy is fully $\ddot{\mathcal{P}}''$-elliptic. Elevations of partly-cut-up hyperplanes do not have accidental $\ddot{\mathcal{Z}}''$-loops:
\begin{lemma}[{\cite[Lemma 5.15]{AGM}}]\label{Lem: cut up hyp no accident}
Let $(X,\mathcal{Z})$ be a superconvex pair where each component of $\mathcal{Z}$ is embedded and let $C$ be the corresponding augmented cube complex. For $n\ge 1$, let $\{W_1,\ldots,W_n\}$ be a collection of embedded, $2$--sided, nonseparating hyperplanes of $C$. Let $Q$ be a component of $W_n\setminus \cup_{i<n} W_i$. Then $Q$ has no accidental $\mathcal{Z}$-loops.
\end{lemma}

\begin{proposition}\label{Prop: full ellipticity}
Let $E$ be an edge space of the double dot hierarchy for $\ddot{C}_2$. Then $E$ has no accidental $\ddot{\mathcal{Z}}''-$loops.  
\end{proposition}

\begin{proof}
Recall that $E$ is a union of a partly-cut-up hyperplane elevations and components of $\ddot{\mathcal{Z}}''$ that intersect these elevations. Let $Q$ be one of the partly-cut-up hyperplane elevations. By Lemma~\ref{Lem: cut up hyp no accident}, $Q$ has no accidental $\ddot{\mathcal{Z}}''$-loops.

Suppose there exists a $\ddot{C}_2$-essential loop $\gamma$ in $E$ such that $\gamma$ is freely homotopic in $\ddot{C}_2$ into $\ddot{\mathcal{Z}}''$. 
Then a representative of the homotopy class of $\gamma$ lifts to a bi-infinite $\tilde{E}$-geodesic $\hat\gamma$ where $\tilde{E}$ is an elevation of $E$ to $\tilde{C}$, and a representative of the homotopy class of $\gamma$ lifts to a bi-infinite $\tilde{C}$-geodesic $\rho\subseteq \tilde{Z},$ an elevation of a component of $\ddot{\mathcal{Z}}''$ and there exists $R\ge 0$ so that $\hat\gamma\subseteq \mathcal{N}_R(\rho)$.

Since $\hat\gamma$ is a $\tilde{E}$-geodesic, $\hat\gamma$ is a $(\lambda,\epsilon)$--quasigeodesic in $\tilde{C}$ by Proposition~\ref{Prop: QC hierarchy}. 

Let $\hat\gamma_0$ be a subsegment of $\hat\gamma$ with $|\hat\gamma_0| = |\gamma|$ (e.g. take $\hat\gamma_0$ to be the subsegment between two consecutive lifts of a point of $\gamma$ to $\hat\gamma$). 
If $\hat\gamma_0\subseteq \tilde{Z}'$ where $\tilde{Z}'$ is an elevation of a component of $\ddot{\mathcal{Z}}''$, then $\hat\gamma\subseteq \tilde{Z}'$ and $\hat\gamma$ is geodesic in $\tilde{C}$. Then $\tilde{Z} = \tilde{Z}'$ because $\diam (\mathcal{N}_R(\tilde{Z})\cap \mathcal{N}_R(\tilde{Z}')) = \infty$ in which case $\gamma$ was not an accidental $\ddot{\mathcal{Z}}''$ loop.

On the other hand, if $\hat{\gamma_0} \subseteq \tilde{Q}$ where $\tilde{Q}$ is some elevation of $Q$ to $\tilde{C}$, then $Q$ has an accidental $\ddot{\mc{Z}}''$-loop, contradicting the fact that there are no such accidental $\mathcal{Z}$ loops. 

Therefore, there exist subsegments of $\hat\gamma$ of the form $\gamma_m = a_{m,1}b_{m,1}a_{m,2}b_{m,2}\ldots a_{m,k_m}b_{m,k_m}$ such that $\cup_1^\infty \gamma_m = \hat\gamma$, $|\gamma_m|\to\infty$ and $k_m\to\infty$ as $m\to\infty$, $a_{m,i}$ lies in an elevation $\tilde{Q}_i$ of $Q$ to $\tilde{C}$, $b_{m,i}\subseteq \tilde{Z}_{m,i}$ where $\tilde{Z}_{m,i}$ is an elevation of a component of $\ddot{\mathcal{Z}}''$ to $\tilde{C}$, and if $i\ne j$, $b_{m,i}\subseteq \tilde{Z}_i$ and $b_{m,j}\subseteq \tilde{Z}_j \ne \tilde{Z}_i$ (otherwise, by convexity of $Z_i$, $\gamma_m$ could be written as a concatenation of fewer geodesic segments). 
Recall that $Q$ is $R_0$--embedded, so for all $m,i$, $|b_{m,i}|\ge R_0$.

By construction there is a unique $B\in\mathcal{B}$ so that $\tilde{Z}\subseteq B$. 
Let $\tau_m$ be the $\tilde{C}$-geodesic connecting the endpoints of $\gamma_m$. 
Since $\tau_m\subseteq \mathcal{N}_R(B)$ and $B$ is $(3M+6R+2f(R)+9\delta)$--attractive, all but $(3M+6R+2f(R)+9\delta)$-tails of the endpoints of $\tau_m$ lie in $B$.
Therefore, there exists a subsegment $\tau_m^B\subseteq \tau_m\cap B$ so that $|\tau_m^B|\ge |\tau_m|-2(3M+6R+2f(R)+9\delta)$. 

Recall that all $(\lambda,\epsilon)$--quasigeodesics with the same endpoints $(L_{rftp},L_{rftp})$--fellow travel relative to $\mc{B}$. 
There exists a unique $B_{m,i}\in\mc{B}$ containing $\tilde{Z}_{m,i}$, so $b_{m,i}\subseteq B_{m,i}$. 
Then for $B\in\mc{B}$ with $B\ne B_{m,i}$, $\diam(b_{m,i}\cap \mc{N}_{L_{rftp}}(B))\le f(L_{rftp})$. Since $\tau_m$ and $\gamma_m$ relatively fellow travel, either:
\begin{itemize}
\item there exist points $p_{m,i}^-$ and $p_{m,i}^+$ on $b_{m,i}\subseteq \gamma_m$ that are at most $f(L_{rftp})$ from the endpoints of $b_{m,i}$ and are distance $L_{rftp}$ from $\tau_m$ or
\item there exist $p_{m,i}^-,p_{m,i}^+$ on $\gamma_m$ so that $p_{m,i}^-,p_{m,i}^+$ are distance at most $L_{rftp}$ from points in $\tau_m$ that lie in $\mc{N}_{L_{rftp}}(B_{m,i})$ and the interval of $\gamma_m$ between $p_{m,i}^-$ and $p_{m,i}^+$ contains all of $b_{m,i}$ except for a length at most $2(f(L_{rftp}))$ subsegment of $b_{m,i}$.
\end{itemize}
Indeed, any subsegment $b_{m,i}$ that lies in $\mc{N}_{L_{rftp}}(B)$ for any $B\in\mc{B}$ with $B\ne B_{m,i}$ has length at most $f(L_{rftp})$. 
In either case since $\gamma_m$ is $(\lambda,\epsilon)$--quasigeodesic, there exists a $\frac1\lambda(\frac1\lambda (R_0- 2(f(L_{rftp})))-\epsilon) - \epsilon-2L_{rftp}$ subsegment of $\tau_m$ that lies in $\mc{N}_{L_{rftp}}(B_{m,i})$. As $m\to \infty$, $|\tau_m|\to \infty$ while $|\tau_m|-|\tau_m^B| \le 2(3M+6R+2f(R)+9\delta)$, which does not depend on $m$. Therefore, for $m>>0$, there are at least two $i$ such that $\tau_m^B$ has a length 
\[\frac1\lambda\left(\frac1\lambda \left(R_0- 2(f(L_{rftp}))\right)-\epsilon\right) - \epsilon-2L_{rftp}> 3f(L_{rftp})\]
 subsegment lying in $\mc{N}_{L_{rftp}}(B_{m,i})\cap\mc{N}_{L_{rftp}}(B)$ (Recall $R_0$ was chosen in Notation~\ref{N: Sec73 constants}). 
Since the $B_{m,i}$ are pairwise distinct, we obtain a contradiction. 
Therefore, $\gamma$ cannot be an accidental $\ddot{Z}''$-loop.
\end{proof}

\begin{cor}
The double dot hierarchy on $\ddot{C}_2$ is fully $\ddot{\mathcal{P}}''$-elliptic. 
\end{cor}

Faithfulness, quasiconvexity and full $\mathcal{P}$-ellipticity are preserved by taking the induced hierarchy of a finite regular cover of $\ddot{C}_2$. 
The final step is to show that there exists a finite cover of $\ddot{C}_2$ whose induced hierarchy is also a malnormal hierarchy. 

The following lemma is straightforward:

\begin{lemma}
Suppose $H\le G$ and $G_0$ is a finite index subgroup of $G$ and let $H_0=H\cap G_0$. If $H$ is malnormal in $G$, then $H_0$ is malnormal in $H$. 
\end{lemma}

The following is a special case of \cite[Corollary 6.4]{SW2015}:
\begin{proposition}
Let $G$ be the fundamental group of a relatively hyperbolic special compact NPC cube complex, and let $H\le G$ be full relatively quasiconvex. Then $H$ is separable in $G$.
\end{proposition}

\begin{proposition}[Hruska-Wise {\cite[Theorem 9.3]{HW09}}]\label{P: HW relative malnormal}
If $G$ is relatively hyperbolic and $H\le G$ is relatively quasiconvex and separable, then there exists a finite index subgroup $K_0\le G$ containing $H$ such that for every $g\in K_0\setminus H$ either $gHg\inv\cap H$ is finite or $gHg\inv\cap H$ is parabolic in $K$.
\end{proposition}

\begin{proposition}\label{P: frqc almost malnormal}
If $G$ is relatively hyperbolic and $H\le G$ is full relatively quasiconvex, there is a  finite index subgroup $K\le G$ containing $H$ such that $H$ is almost malnormal in $G$. 
\end{proposition}

\begin{proof}
\textbf{Claim: if $H\le G$ is full relatively quasiconvex, then there are only finitely many double cosets of the form $HgH$ so that $H\cap H^g$ is infinite and parabolic.}

Let $\mc{D}$ be the induced peripheral structure on $H$. If $H\cap H^g$ is infinite parabolic, then fullness implies there are $Q_1,Q_2\le H$ that are maximal parabolic in $H$ so that $H\cap H^g$ is finite index in $Q_1\cap Q_2^g$.  
Then there exist $D_1,D_2\in \mc{D}$ and $h_1,h_2\in H$ so that $Q_1 = D_1^{h_1}$ and $Q_2 = D_2^{h_2}$. 
It is easy to verify that if $g_0 =h_1\inv gh_2$:
\begin{enumerate}
\item $g_0\in HgH$,
\item $HgH = Hg_0H$, and
\item $H\cap H^{g_0}\le D_1\cap D_2^{g_0}$. 
\end{enumerate}
In other words, given a double coset $HgH$ so that $H\cap H^g$ is infinite parabolic, we may assume that $g$ is chosen so that there are maximal parabolic $D_1,D_2\le H$ so that $H\cap H^g\le D_1\cap D_2^{g}$. 

Since $\mc{D}$ is finite, it suffices to show that for any $D_1,D_2\in \mc{D}$ ($D_1,D_2$ need not be distinct) there are finitely many double cosets of the form $HgH$ so that $H\cap H^g$ is infinite and $H\cap H^g\subseteq D_1\cap D_2^g$. 

Now suppose $Hg_1H$ is another double coset so that $H\cap H^{g_1}$ is an infinite subgroup of $D_1\cap D_2^g$. We see that $D_1^{g\inv}$ and $D_1^{g_1\inv}$ have infinite intersection with $D_2$ and are therefore finite index in $D_2$ by fullness, so $D_1^{g\inv}\cap D_1^{g_1\inv}$ is infinite and hence $D_1\cap D_1^{gg_1\inv}$ is infinite.
Let $P$ be the maximal parabolic subgroup of $G$ containing $D_1$. 
The fullness of $H$ implies that $D_1$ is finite index in $P$. 
Therefore, $P\cap P^{gg_1\inv}$ is infinite, so $g g_1\inv \in P$. 
There are finitely many left cosets $t_1 D_1, t_2D_1,\ldots, t_\ell D_1$ of $D_1$ in $P$. 
Hence $g g_1\inv = t_i d$ for some $d\in D_1\le H$ and $1\le i\le \ell$ which means $g_1\inv = g\inv t_id$, so $Hg_1\inv H = Hg\inv t_iH$. There are only finitely many choices for $t_i$, proving the claim. 

Proposition~\ref{P: HW relative malnormal} implies that if we first pass to a finite index $K_0\le G$ containing $H$, we can ensure that if $g\in K_0\setminus H$ and $H\cap H^g$ is not finite, it is infinite parabolic. By the preceding, there is a finite collection of double cosets $Hk_1H,\ldots,Hk_mH$ so that $g\in Hk_iH$ for some $1\le i\le m$. Note all $k_i\notin H$.  
The separability of $H$ implies that we can choose a finite index $K\le K_0$ containing $H$ so that $k_1,\ldots,k_m\notin K$. 
Then $Hk_iH \cap K=\emptyset$ because $H\le K$. By the preceding, there exists no $k\in K$ such that $H\cap H^k $ is infinite parabolic, so $H\cap H^k$ is finite for all $k\in K$. 
\end{proof}

Corollary~\ref{malnormal promotion} is based on \cite[Corollary 3.29]{AGM}. Corollary~\ref{malnormal promotion} follows immediately from the two preceding statements and the fact that when $G$ is virtually special, $G$ is linear and hence virtually torsion free. 
\begin{cor}\label{malnormal promotion}
If $G$ is hyperbolic relative to $\mathcal{P}$ and special, and $H\le G$ is full relatively quasiconvex, then $H$ is virtually malnormal. 
\end{cor}

\begin{theorem}\label{malnormal hierarchy promotion}
Let $G$ be special, virtually torsion-free and let $(G,\mathcal{P})$ be a relatively  hyperbolic group pair. 
Let $\mathcal{H}$ be a fully $\mathcal{P}$-elliptic quasiconvex hierarchy for $G$, then there exists a finite index normal subgroup $G_0 \le G$ with induced fully $\mathcal{P}-$elliptic quasiconvex hierarchy $\mathcal{H}_0$ of $G_0$ which is malnormal and fully $\mathcal{P}-$elliptic. 
\end{theorem}

The proof here is nearly the same as in \cite[Theorem 3.30]{AGM}.

\begin{proof} 
Because $\mathcal{H}$ is fully $\mathcal{P}$-elliptic, the edge subgroups are full.  Since there are finitely many edge groups, by Corollary~\ref{malnormal promotion}, there exists some $G_0$ such that for every edge group $E$ of $\mathcal{H}$, $E\cap G_0$ is malnormal in $G_0$. By passing to a deeper finite index subgroup, we may insist that $G_0$ is normal. Since $G_0$ is normal, conjugation by $g\in G$ is an automorphism of $G_0$, so in particular, these edge groups $E\cap G_0$ are malnormal in $G$.
\end{proof}

At last, it is time to prove Theorem~\ref{Thm: main thm}.
\mainthm*

\begin{proof}[Proof of Theorem~\ref{Thm: main thm}]
Let $X$ be a NPC compact special cube complex so that $\pi_1(X)$ is finite index in $G$. 

First, pass to a finite index regular cover of $X$, $X_1$ that is special.
By applying a homotopy equivalence, $X_1$ is homotopy equivalent to a cube complex where every hyperplane gives a nontrivial splitting of $\pi_1X_1$ (see \cite[Lemma 5.17]{AGM}).  

By Corollary \ref{DDH Good}, there exists a special cube complex $X_1'$ homotopy equivalent to $X_1$ with a finite regular cover $X_2$ such that $G_2:=\pi_1X_2$ with induced peripheral structure $(G_2,\mathcal{P}_2)$ has a faithful, quasiconvex, fully $\mathcal{P}_2-$elliptic hierarchy terminating in $\mathcal{P}_2*F_k$ where $F_k$ is a free group. 

By Theorem~\ref{malnormal hierarchy promotion}, there exists a finite regular cover $X_0$ with $G_0:=\pi_1X_0$ and induced peripheral structure $(G_0,\mathcal{P}_0)$ such that the induced hierarchy on $G_0$ is malnormal as well and terminates in free products of free groups and elements of $\mathcal{P}_0$ (recall Corollary~\ref{C: terminal spaces}). The hierarchy can then be continued to a malnormal, quasiconvex, fully-$\mathcal{P}_0$-elliptic one that terminates in $\mathcal{P}_0$.
\end{proof}

\section{A Relatively Hyperbolic Version of the Malnormal Special Quotient Theorem}\label{MSQT section}

Recall Wise's Malnormal Special Quotient Theorem (MSQT), see Theorem~\ref{MSQT} above or \cite[Theorem 12.2]{WiseManuscript} mentioned in the introduction.
The purpose of this section is to apply Theorem~\ref{Thm: main thm} to obtain a relatively hyperbolic version of Wise's MSQT using techniques from \cite[Sections 6-9]{AGM}. 

Wise's Quasiconvex Hierarchy Theorem \cite[Theorem 13.3]{WiseManuscript} has the following useful consequence:
\begin{cor}\label{Wise main thm cor}
Let $G$ be a hyperbolic group with a quasiconvex hierarchy terminating in finite groups. Then $G$ is virtually special. 
\end{cor}
The technique for proving a relatively hyperbolic analog of Theorem~\ref{MSQT} will be to start with the hierarchy provided by Theorem~\ref{Thm: main thm} and strategically take quotients using group theoretic Dehn fillings (see Definition~\ref{Def: gp Dehn filling}). These quotients can be constructed to be hyperbolic, and with some care, the hierarchy structure can be passed down to the quotient so that Corollary~\ref{Wise main thm cor} can be used.
In \cite{AGM}, the authors avoided using Corollary~\ref{Wise main thm cor} because their account aimed to give a new proof of auxiliary results used to prove Corollary~\ref{Wise main thm cor}. Consequently, they needed to ensure that the hierarchy structure on the quotient is also a malnormal hierarchy. By using Corollary~\ref{Wise main thm cor}, we only need a quasiconvex hierarchy for such a quotient.

\subsection{Group Theoretic Dehn Filling}
For this section, let $(G,\mathcal{P})$ be a relatively hyperbolic group pair where $\mathcal{P} = \{P_1,\ldots,P_m\}$ unless stated otherwise.
When $M$ is a finite volume hyperbolic $3$--manifold with torus cusps, a \textbf{Dehn filling} of $M$ is a gluing of solid tori $T_i \cong D\times S^1$ by a diffeomorphism to the boundary components. The result of the gluing depends only on the isotopy class of the curve $\gamma_i\subseteq \partial M$ that each copy of $\partial D \times \{p\}\subseteq T_i$ is glued to (see e.g. \cite[Section 10.1]{Martelli}).
In this situation $\pi_1M$ is hyperbolic relative to a collection of copies of $\Z^2$, one for each boundary component of $M$. 

The next definition is a group theoretic analog of Dehn filling
\begin{definition}\label{Def: gp Dehn filling}
Let $\{N_i\lhd P_i:\, 1\le i \le m\}$. Then there exists a \textbf{group theoretic Dehn filling} of $G$ with \textbf{filling map} $\pi$ defined by the quotient:
\[ \pi:\,G\to G(N_1,\ldots,N_m) := G/\cyc{\cyc{\bigcup N_i}}.\]
The subgroups $N_i$ are called \textbf{filling kernels}.

A filling is called \textbf{peripherally finite} if each filling kernel $N_i$ is finite index in $P_i$. 
\end{definition}

For a classical filling, if every $T_i$ is filled by gluing along the curves $\gamma_i$ that are sufficiently long, Thurston's Dehn filling theorem says that the resulting manifold is hyperbolic. The group theoretic analog of a sufficiently long classical Dehn filling is a group theoretic Dehn filling where the filling kernels avoid a finite set of elements:
\begin{definition}
A statement $\mathfrak{P}$ holds for all sufficiently long fillings if there exists a finite $B\subseteq G\setminus \{1\}$ such that whenever $B\cap N_i = \emptyset$ for all $1\le i\le m$, the filling $G(N_1,\ldots N_m)$ has $\mathfrak{P}$. 
\end{definition}

Osin showed that sufficiently long Dehn fillings of relatively hyperbolic groups are relatively hyperbolic, have kernels which intersect each peripheral subgroup $P_i$ precisely in $N_i$ and can be manipulated so that any finite set of elements are not killed by the filling map. 
\begin{theorem}[{\cite[Theorem 1.1]{Osin2007}}]\label{Thm: Osin Dehn filling}
Let $F\subseteq G$ be any finite subset of $G$. Then for all sufficiently long Dehn fillings:
\begin{enumerate}
\item $\ker(\phi|_{P_i}) = N_i$ for $i=1,2,\ldots,m$, \label{I: specify kernel}
\item the pair $(G(N_1,\ldots,N_m),\{\phi(P_1),\ldots,\phi(P_m)\})$ is a relatively hyperbolic group pair, and \label{I: descent to hyp}
\item $\phi|_F$ is injective.
\end{enumerate}
\end{theorem}

The edge subgroups of the hierarchy from Theorem~\ref{Thm: main thm} will need to be full relatively quasiconvex subgroups of $G$. The quasiconvexity of the hierarchy will ensure that these subgroups are relatively quasiconvex. 
\begin{theorem}[{\cite[Theorem 1.5]{Hruska2010}}]\label{undistorted rel qc}
Let $H\le G$ be a quasi-isometrically embedded subgroup. Then $H$ is relatively quasiconvex in $G$. 
\end{theorem}

\begin{theorem}[{\cite[Theorem 1.2]{Hruska2010}}]\label{Thm: induced qc rel hyp}
Let $H\le G$ be  relatively quasiconvex. Then there exists a relatively hyperbolic structure $(H,\mathcal{D})$ where $\mathcal{D}$ is finite and every element of $\mathcal{D}$ is conjugate into an element of $\mathcal{P}$. 
\end{theorem}

\begin{cor}\label{Prop: modified structure}
The collection $\mathcal{D}$ can be chosen so that:
\begin{enumerate}
\item every element of $\mathcal{D}$ is infinite, and
\item whenever $H\cap P^g$ is infinite, for some $g\in G$, there exists $h\in H$ so that $(H\cap P^g)^h$ is an element of $\mc{D}$. 
\end{enumerate}
\end{cor}

\begin{proof}
For the first statement, simply remove all finite elements of $\mathcal{D}$. 
The second statement follows from \cite[Theorem 9.1]{Hruska2010}.
\end{proof}

When a filling of $G$ interacts nicely with a subgroup $H$, it is possible to induce a filling on the subgroup $H$.

\begin{definition}[{\cite[Definition B.1]{MMP}}]
Let $H\le G$. A filling $G\to G(N_1,\ldots,N_m)$ is an \textbf{$H$-filling} if whenever $gP_ig\inv \cap H$ is infinite for some $P_i\in\mathcal{P}$, then $gN_ig\inv \subseteq H$. 
\end{definition}

\begin{definition}
Suppose $H\le G$ is a relatively quasiconvex subgroup and let $(H,\mathcal{D})$ be the relatively hyperbolic structure from Theorem~\ref{Thm: induced qc rel hyp} and Corollary~\ref{Prop: modified structure}. 
Let $\pi:G\to G(N_1,\ldots, N_m)$ be an $H$-filling. Let $D_j\in \mathcal{D}$.
Then there exists some $P_i\in\mathcal{P}$ and $g\in G$ with $g\inv D_j g\subseteq P_i$. 
Let $K_j := g N_i g\inv$. 
Since $\pi$ is an $H$-filling, $K_j\lhd D_j$, so the groups $K_j$ determine a filling:
\[\pi_H:H\to H(K_1,\ldots,K_N)\]
called the \textbf{induced filling of $H$} with respect to $G(N_1,\ldots,N_m)$. 
\end{definition}
Since $N_i$ is normal in $P_i$, then groups $K_j$ (and hence the filling) do not depend on the choice of $g\in G$.
The following theorem appears as stated in \cite{AGM} as Theorem 7.11 and collects results about induced Dehn fillings from \cite{AGM09}: 
\begin{theorem}\label{Thm: qc hierarchy omnibus}
Let $H\le G$ be a full relatively quasiconvex subgroup and let $F\subseteq G$ be a finite subset. For all sufficiently long $H$-fillings, $\phi: G\to G(N_1,\ldots, N_m)$ of $G$:
\begin{enumerate}
\item $\phi(H)$ is a full relatively quasiconvex subgroup of $G(N_1,\ldots,N_m)$,
\item $\phi(H)$ is isomorphic to the induced filling in that if $\phi_H:H\to H(K_1,\ldots,K_m)$ is the induced filling map, then $\ker\phi_H = \ker\phi \cap H$, and
\item $\phi(F)\cap \phi(H) = \phi(F\cap H)$. 
\end{enumerate}
\end{theorem}

\subsection{The filled hierarchy}

Let $\mathcal{H}$ be a quasiconvex fully $\mathcal{P}-$elliptic hierarchy. 
By Lemma~\ref{vertex are qi}, Theorem~\ref{undistorted rel qc} and the full $\mathcal{P}$-ellipticity of the hierarchy, the edge and vertex groups of the hierarchy are full relatively quasiconvex. 
Let $\pi:G\to \bar{G}$ be a filling and let $(\bar{G},\bar{\mathcal{P}})$ be the relatively hyperbolic structure induced on the filling by Theorem~\ref{Thm: Osin Dehn filling}. 
The goal of this subsection is to build an induced hierarchy $\bar{\mathcal{H}}$ (which may not be faithful) for $\bar{G}$ based on $\mathcal{H}$ where the vertex and edge groups of $\bar{\mathcal{H}}$ are induced fillings of vertex and edge groups of $\mathcal{H}$. 
The hierarchy $\bar{\mathcal{H}}$ will be called a \textbf{filled hierarchy} for $(\bar{G},\bar{\mathcal{P}})$. 

The filled hierarchy is built by starting at the top level and building the hierarchy inductively downward. 

At the top level, let $\bar{\mathcal{H}}$ have the degenerate graph of groups decomposition for $\bar{G}$ consisting of a single vertex labeled $\bar{G}$.
Let $n$ be the length of $\mathcal{H}$.  
Suppose the filled hierarchy has been filled down to the $(n-i)$th level and let $\bar{A}$ be a vertex group at level $n-i$ so that $\bar{A}$ is the induced filling of a vertex group $A$ at level $n-i$ of $\mathcal{H}$. 
Let $(\Gamma,\chi)$ be the graph of groups structure for $A$ provided by $\mathcal{H}$. Recall that $\chi$ is the assignment map for the graph of groups structure. 

If $x$ is a vertex or edge of $\Gamma$, let $A_x := \chi(x)$ be the corresponding vertex or edge group.
Let $\bar\chi(x):= \bar{A}_x$ where $\bar{A}_x$ is the induced filling $\pi_x:A_x\to \bar{A}_x$.
The problem is that the pair $(\Gamma,\bar{\chi})$ still needs attachment homomorphisms to be a graph of groups. 
 
Let $\phi_e:A_e\to A_v$ be an attachment homomorphism of an edge group $A_e$ to a vertex group $A_v$. Two details need to be checked: first there need to be attachment maps $\bar\phi_e:\bar{A}_e\to \bar{A}_v$ such that $\bar\phi_e \circ\pi_e = \pi_v \circ \phi_e$. 
Let $T$ be the maximal tree that determines $\pi_1(\Gamma,\chi,T)$. 
There will also need to be an isomorphism $\bar\alpha: \pi_1(\Gamma,\bar\chi,T)\to \bar{A}$ so that $(\Gamma,\bar\chi,T)$ is a graph of groups structure for $\bar{A}$ where $\bar\alpha \circ \pi_\Gamma = \pi_A\circ \alpha$. 

Completing the square:
\[\begin{CD}
A_e @>> \pi_e > \bar{A}_e \\
@VV\phi_e V @VV\bar{\phi}_e V \\
A_v @>> \pi_v > \bar{A}_v \\
\end{CD}\]
with a map $\bar{\phi}_e: \bar{A}_e\to \bar{A}_v$ is straightforward because $\pi_e$ is surjective and $\ker \pi_e  \subseteq\ker\pi_v\circ \phi_e$. 

Constructing the desired isomorphism $\bar\alpha:\pi_1(\Gamma,\bar\chi,T)\to \bar{G}$ amounts to completing the square:
\[\begin{CD} 
\pi_1(\Gamma,\chi,T) @>> \pi_\Gamma > \pi_1(\Gamma,\bar\chi,T) \\ 
@VV \alpha V @VVV \\
A @>> \pi_A > \bar{A}
\end{CD}
\]

\begin{lemma}\label{L: diagram completion}
There exists an isomorphism $\bar\alpha:\pi_1(\Gamma,\bar\chi,T)\to \bar{G}$ that completes the diagram.
\end{lemma}

The proof of Lemma~\ref{L: diagram completion} is essentially identical to \cite[Lemma 8.1]{AGM}.

For the following, let $(G,\mathcal{P})$ be a relatively hyperbolic group pair and let $\mathcal{H}$ be a quasiconvex fully $\mathcal{P}-$elliptic hierarchy for $G$.
The next lemma ties together some definitions:
\begin{lemma}
If $A\le G$ is an edge or vertex group of $\mathcal{H}$, then $A$ is a full relatively quasiconvex subgroup of $(G,\mathcal{P})$ and every filling is an $A$-filling. 
\end{lemma}

\begin{proof}
That $A$ is full relatively quasiconvex follows immediately from the definition of full $\mathcal{P}$-ellipticity and Theorem~\ref{undistorted rel qc}.

Whenever $gP_ig\inv\cap A$ is infinite, then $gP_ig\inv \subseteq A$, so if $N_i\lhd P_i$, then $gN_ig\inv \lhd A$.  
\end{proof}

\begin{lemma}\label{Lem: induced omnibus}
Let $A$ be an edge or vertex group of $\mathcal{H}$. Then for all sufficiently long fillings:
\[\pi:(G,\mathcal{P})\to (\bar{G},\bar{\mathcal{P}})\]
the following hold:
\begin{enumerate}
\item The subgroup $\bar{A} :=\phi(A)$ is full relatively quasiconvex in $(\bar{G},\bar{P})$,
\item If $\bar{G}$ is hyperbolic, then $\bar{A}$ is quasiconvex in $\bar{G}$,
\item The subgroup $\bar{A}$ is isomorphic to the induced filling of $A$. 
\end{enumerate}  
\end{lemma}

\begin{proof}
There are only finitely many edge and vertex groups, so the first and third statements follow from Theorem~\ref{Thm: qc hierarchy omnibus}.

If $\bar{A}$ is full relatively quasiconvex in $(\bar{G},\bar{P})$, then $\bar{A}$ is undistorted in $\bar{G}$ by \cite[Theorem 10.5]{Hruska2010} and by \cite[Corollary III.$\Gamma$.3.6]{BridsonHaefliger}, $\bar{A}$ is quasiconvex in $\bar{G}$ whenever $\bar{G}$ is hyperbolic. 
\end{proof}

The third point also makes the filled hierarchy $\bar{\mc{H}}$ faithful:
\begin{cor}\label{Cor: faithful induced hierarchy}
For all sufficiently long fillings $\pi:(G,\mathcal{P})\to (\bar{G},\bar{\mathcal{P}})$, the filled hierarchy $\bar{\mathcal{H}}$ for $\bar{G}$ is faithful.
\end{cor}

\begin{proof}
Let $\phi_e:A_e\to A_v$ be an attachment homomorphism mapping an edge group $A_e$ to a vertex group $A_v$. Since $\pi(A_e)$ and $\pi(A_v)$ are isomorphic to the induced fillings, we can regard the induced filling maps as maps $\pi_v:A_v\to \bar{A}_v$ and $\pi_e:A_v\to \bar{A}_e$. 
Let $\bar\phi_e:\bar{A}_e\to \bar{A}_v$ be the induced edge homomorphism.

We now need to check that given $g_e \in A_e$, $\phi_e\circ \pi_e(g_e) =1$ implies that $\pi_e(g_e) = 1$. 
If $\phi_e\circ \pi_e(g_e) = 1$, then $\pi_v\circ \phi_e(g_e) = 1$, so $\phi_e(g_e)\in \ker \pi_v = \ker \pi \cap A_v \subseteq \ker\pi$.
Faithfulness of the original hierarchy now implies $g_e\in \ker \pi\cap A_e = \ker \pi_e$, so $\pi_e(g_e)=1$.   
\end{proof}

The preceding results combine to produce a quasiconvex hierarchy:

\begin{theorem}[see {\cite[Theorem 2.12]{AGM}}]\label{Thm: induced hierarchy theorem}
Let $(G,\mathcal{P})$ be a relatively hyperbolic group pair and let $\mathcal{H}$ be a quasiconvex fully $\mathcal{P}$-elliptic hierarchy terminating in $\mathcal{P}$. For all sufficiently long peripherally finite fillings $\pi:(G,\mathcal{P})\to (\bar{G},\bar{\mathcal{P}})$ so that every $\bar{P}\in\bar{\mathcal{P}}$ is hyperbolic, the group $\bar{G}$ is hyperbolic and has a quasiconvex hierarchy terminating in $\bar{\mathcal{P}}$. 
\end{theorem}

\begin{proof}
Theorem~\ref{Thm: Osin Dehn filling} implies that all sufficiently long peripherally finite fillings are hyperbolic. 

By Corollary~\ref{Cor: faithful induced hierarchy}, the quotient $\bar{G}$ has a faithful hierarchy $\bar{\mc{H}}$ where the underlying graphs and every vertex or edge group of $\bar{\mc{H}}$ is the image of a vertex or edge group (respectively) of $\bar{\mc{H}}$ under $\pi$. 

By Lemma~\ref{Lem: induced omnibus} (2), every edge and vertex group of $\bar{\mc{H}}$ is quasiconvex in $\bar{G}$ and is hence also quasi-isometrically embedded in $\bar{G}$, so the hierarchy $\bar{\mc{H}}$ is quasiconvex. 

By construction, the terminal groups are fillings of the terminal groups of $\mathcal{H}$, so the terminal groups of $\bar{\mathcal{H}}$ are in $\bar{\mathcal{P}}$. 
\end{proof}

Theorem~\ref{Thm: induced hierarchy theorem} works for a group with a quasiconvex hierarchy, but Theorem~\ref{Thm: main thm} only gives a hierarchy for a finite index subgroup. When the filling kernels are chosen carefully, a filling of a finite index subgroup $G'\lhd G$ can be promoted to a filling of $G$. 

\begin{definition}
Let $(G,\mathcal{P})$ be a relatively hyperbolic group pair and let $G'\lhd G$ be a finite index normal subgroup with induced peripheral structure $(G',\mathcal{P}')$. 
Let $\{N_j'\lhd P_j'\,|\,P_j'\in\mathcal{P}_j'\}$ be a collection of filling kernels.
The collection $\{N_j'\}$ is \textbf{equivariantly chosen} if 
\begin{enumerate}
\item whenever $gP'_jg\inv$ and $hP'_kh\inv$ both lie in $P_i$, then $gN'_jg\inv = hN_k'h\inv$ and 
\item every such $gN_j'g\inv$ is normal in $P_i$. 
\end{enumerate}
An \textbf{equivariant filling} of $(G',\mathcal{P}')$ is a filling with equivariantly chosen filling kernels. 
\end{definition}
An equivariant filling of $(G',\mathcal{P}')$ will induce a nice filling of $(G,\mathcal{P})$:
\begin{proposition}\label{Prop: induce up filling}
An equivariant filling $(G',\mathcal{P}')\to (\bar{G}',\bar{\mathcal{P}}')$ determines a filling $(G,\mathcal{P})\to (\bar{G},\bar{\mathcal{P}})$ so that $\bar{G}'$ is finite index normal in $\bar{G}$ and $(\bar{G}',\mathcal{\bar{P}}')$ is the peripheral structure induced by $(\bar{G},\bar{\mathcal{P}})$.
\end{proposition}


For the reader's convenience, here is a restatement of Theorem~\ref{MSQT rel hyp}.
\relhypmsqt*

\begin{proof}
By Theorem~\ref{Thm: main thm}, there exists a finite index $G'\lhd G$ with induced peripheral structure $(G',\mathcal{P}')$ and a quasiconvex, fully $\mathcal{P}'-$elliptic hierarchy terminating in $\mathcal{P}'$. 
Let $\mc{P}' = \{P_1',\ldots,P_M'\}$.  
Since $G$ is virtually special and hence residually finite, there exist arbitrarily long peripherally finite fillings of $(G',\mathcal{P}')$. 
In particular, our fillings of $(G',\mc{P}')$ will be sufficiently long for Theorem~\ref{Thm: induced hierarchy theorem} to hold.

Let $G'(K_1,\ldots,K_M)$ be such a peripherally finite filling. 
Now pass to subgroups of the filling kernels to obtain an equivariant filling; choose $K_j'$ so that if $K_j^g\le P_i$:
\[(K_j')^g = \bigcap \{K_\ell^h\,|\,h\in G,\,\#(K_\ell^h\cap P_i) =\infty \}:\qquad 1\le j\le M.\]
We set $\dot{P}_i\lhd P_i$ equal to $(K_j')^g$ for some (any) choice of $g\in G$ where $K_j'$ so that $(K_j')^g\le P_i$. 
The new filling $\bar{G}' = G'(K_1',\ldots,K_M')$ is longer than $G'(K_1,\ldots,K_M)$ and remains peripherally finite. 
By Proposition~\ref{Prop: induce up filling}, the filling $G'(K_1',\ldots,K_M')$ determines a filling of $G$.

Consider any filling $G(N_1,\ldots,N_m)$ so that for each $i$:
\begin{enumerate}
\item $N_i\lhd P_i$
\item $N_i\le \dot{P}_i$ and \label{I: deep enough}
\item $P_i/N_i$ is virtually special and hyperbolic. 
\end{enumerate}
with an induced equivariant filling:
\[G' \to G'(N_1',\ldots,N_M')\]
so that $N_j'\le K_j'$ and $N_j'\lhd P_j'$ for each $j$. 
Condition~\eqref{I: deep enough} ensures the filling is sufficiently long so that Theorem~\ref{Thm: induced hierarchy theorem} implies:
\begin{enumerate}
\item $\bar{G}'$ is hyperbolic, and
\item $\bar{G}'$ has a quasiconvex hierarchy terminating in $\bar{\mathcal{P}}' = \{P_j'/N_j'\}$.
\end{enumerate}

Then $\bar{G}'$ is a hyperbolic group with a quasiconvex hierarchy that terminates in finite groups (which are hence hyperbolic and virtually special). So by Corollary~\ref{Wise main thm cor} (see \cite[Theorem 13.3]{WiseManuscript}), $G'(N_1',\ldots,N_M')$ is virtually special. 
By Proposition~\ref{Prop: induce up filling}, $\bar{G}' = G'(N_1',\ldots, N_M')$ is finite index normal in $G(N_1,\ldots,N_m)$, so the filling $G(N_1,\ldots N_m)$ is also virtually special.
\end{proof}

\bibliography{cubes}{}
\bibliographystyle{plain}
\end{document}

%% file: relThinTriangle2.pdf_tex
\begingroup%
  \makeatletter%
  \providecommand\color[2][]{%
    \errmessage{(Inkscape) Color is used for the text in Inkscape, but the package 'color.sty' is not loaded}%
    \renewcommand\color[2][]{}%
  }%
  \providecommand\transparent[1]{%
    \errmessage{(Inkscape) Transparency is used (non-zero) for the text in Inkscape, but the package 'transparent.sty' is not loaded}%
    \renewcommand\transparent[1]{}%
  }%
  \providecommand\rotatebox[2]{#2}%
  \ifx\svgwidth\undefined%
    \setlength{\unitlength}{489.60000346bp}%
    \ifx\svgscale\undefined%
      \relax%
    \else%
      \setlength{\unitlength}{\unitlength * \real{\svgscale}}%
    \fi%
  \else%
    \setlength{\unitlength}{\svgwidth}%
  \fi%
  \global\let\svgwidth\undefined%
  \global\let\svgscale\undefined%
  \makeatother%
  \begin{picture}(1,0.35213657)%
    \put(0,0){\includegraphics[width=\unitlength,page=1]{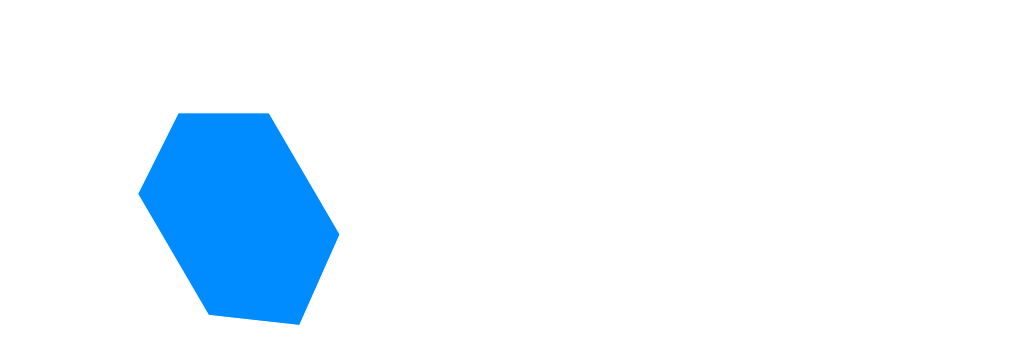}}%
    \put(0.31519389,0.16555019){\color[rgb]{0,0.54901961,1}\makebox(0,0)[lb]{\smash{$\subseteq \mathcal{N}_\delta(F)$}}}%
    \put(0,0){\includegraphics[width=\unitlength,page=2]{RelThinTriangle2.pdf}}%
    \put(-0.00120165,0.00279866){\color[rgb]{0,0,0}\makebox(0,0)[lb]{\smash{$a$}}}%
    \put(0.2111897,0.33991859){\color[rgb]{0,0,0}\makebox(0,0)[lb]{\smash{$b$}}}%
    \put(0.41280171,0.00856249){\color[rgb]{0,0,0}\makebox(0,0)[lb]{\smash{$c$}}}%
    \put(0,0){\includegraphics[width=\unitlength,page=3]{RelThinTriangle2.pdf}}%
    \put(0.09671284,0.01785389){\color[rgb]{0,0,0}\makebox(0,0)[lb]{\smash{$x$}}}%
    \put(0.0698496,0.09684056){\color[rgb]{0,0,0}\makebox(0,0)[lb]{\smash{$y$}}}%
    \put(0.10397181,0.07369029){\color[rgb]{1,0,0}\makebox(0,0)[lb]{\smash{$d(x,y)<\delta$}}}%
    \put(0,0){\includegraphics[width=\unitlength,page=4]{RelThinTriangle2.pdf}}%
    \put(0.50924163,0.00749318){\color[rgb]{0,0,0}\makebox(0,0)[lb]{\smash{$\bar{a}$}}}%
    \put(0.73173694,0.32498288){\color[rgb]{0,0,0}\makebox(0,0)[lb]{\smash{$\bar{b}$}}}%
    \put(0.91900165,0.02342405){\color[rgb]{0,0,0}\makebox(0,0)[lb]{\smash{$\bar{c}$}}}%
    \put(0,0){\includegraphics[width=\unitlength,page=5]{RelThinTriangle2.pdf}}%
    \put(0.59616576,0.04459367){\color[rgb]{0,0,0}\makebox(0,0)[lb]{\smash{$p$}}}%
  \end{picture}%
\endgroup%

%% file: GraphOfSpacesExample.pdf_tex
\begingroup%
  \makeatletter%
  \providecommand\color[2][]{%
    \errmessage{(Inkscape) Color is used for the text in Inkscape, but the package 'color.sty' is not loaded}%
    \renewcommand\color[2][]{}%
  }%
  \providecommand\transparent[1]{%
    \errmessage{(Inkscape) Transparency is used (non-zero) for the text in Inkscape, but the package 'transparent.sty' is not loaded}%
    \renewcommand\transparent[1]{}%
  }%
  \providecommand\rotatebox[2]{#2}%
  \ifx\svgwidth\undefined%
    \setlength{\unitlength}{364.25785095bp}%
    \ifx\svgscale\undefined%
      \relax%
    \else%
      \setlength{\unitlength}{\unitlength * \real{\svgscale}}%
    \fi%
  \else%
    \setlength{\unitlength}{\svgwidth}%
  \fi%
  \global\let\svgwidth\undefined%
  \global\let\svgscale\undefined%
  \makeatother%
  \begin{picture}(1,0.49773037)%
    \put(0,0){\includegraphics[width=\unitlength,page=1]{GraphOfSpacesExample.pdf}}%
    \put(0.16742856,0.20005293){\color[rgb]{1,0,0}\makebox(0,0)[lb]{\smash{$\Sigma_{1,1}$}}}%
    \put(0.66632576,0.2016229){\color[rgb]{1,0,0}\makebox(0,0)[lb]{\smash{$\Sigma_{1,1}$}}}%
    \put(0.42228598,0.13284798){\color[rgb]{0,0,1}\makebox(0,0)[lb]{\smash{$S^1$}}}%
    \put(0,0){\includegraphics[width=\unitlength,page=2]{GraphOfSpacesExample.pdf}}%
    \put(0.16825564,0.07351058){\color[rgb]{1,0,0}\makebox(0,0)[lb]{\smash{$\langle a,b \rangle$}}}%
    \put(0.66715284,0.07508056){\color[rgb]{1,0,0}\makebox(0,0)[lb]{\smash{$\langle c,d\rangle$}}}%
    \put(0.42311305,0.00630563){\color[rgb]{0,0,1}\makebox(0,0)[lb]{\smash{$\mathbb{Z}$}}}%
  \end{picture}%
\endgroup%

%% file: nontrivialhierarchy.pdf_tex
\begingroup%
  \makeatletter%
  \providecommand\color[2][]{%
    \errmessage{(Inkscape) Color is used for the text in Inkscape, but the package 'color.sty' is not loaded}%
    \renewcommand\color[2][]{}%
  }%
  \providecommand\transparent[1]{%
    \errmessage{(Inkscape) Transparency is used (non-zero) for the text in Inkscape, but the package 'transparent.sty' is not loaded}%
    \renewcommand\transparent[1]{}%
  }%
  \providecommand\rotatebox[2]{#2}%
  \ifx\svgwidth\undefined%
    \setlength{\unitlength}{245.78641396bp}%
    \ifx\svgscale\undefined%
      \relax%
    \else%
      \setlength{\unitlength}{\unitlength * \real{\svgscale}}%
    \fi%
  \else%
    \setlength{\unitlength}{\svgwidth}%
  \fi%
  \global\let\svgwidth\undefined%
  \global\let\svgscale\undefined%
  \makeatother%
  \begin{picture}(1,0.9076864)%
    \put(0,0){\includegraphics[width=\unitlength,page=1]{nontrivialhierarchy.pdf}}%
    \put(0.14102133,0.35954979){\color[rgb]{1,0,0}\makebox(0,0)[lb]{\smash{$F_2$}}}%
    \put(0.62161208,0.36106207){\color[rgb]{1,0,0}\makebox(0,0)[lb]{\smash{$F_2$}}}%
    \put(0.38652696,0.2948108){\color[rgb]{0,0,1}\makebox(0,0)[lb]{\smash{$\mathbb{Z}$}}}%
    \put(0.19432943,0.00607425){\color[rgb]{0,0,1}\makebox(0,0)[lb]{\smash{$\mathbb{Z}$}}}%
    \put(0.20475698,0.094824){\color[rgb]{0,0,0}\makebox(0,0)[lb]{\smash{}}}%
    \put(0.2495821,0.07832763){\color[rgb]{0,0.49803922,0}\makebox(0,0)[lb]{\smash{$\langle 1 \rangle$}}}%
    \put(0,0){\includegraphics[width=\unitlength,page=2]{nontrivialhierarchy.pdf}}%
    \put(0.64060101,0.00607425){\color[rgb]{0,0,1}\makebox(0,0)[lb]{\smash{$\mathbb{Z}$}}}%
    \put(0.6958535,0.07832763){\color[rgb]{0,0.49803922,0}\makebox(0,0)[lb]{\smash{$\langle 1 \rangle$}}}%
    \put(0,0){\includegraphics[width=\unitlength,page=3]{nontrivialhierarchy.pdf}}%
  \end{picture}%
\endgroup%

%% file: ch4new.tex
\subsection{CAT$(0)$ Relatively Hyperbolic Pairs}

The main result of the section is Theorem~\ref{Prop: rel fellow traveling}. In \cite{Hruska04NPC2}, Hruska proved that piecewise Euclidean $2$--complexes satisfy a relative form of quasigeodesic stability called the \textbf{relative fellow traveling property}. 
In \cite[Proposition 4.1.6]{HruskaKleiner}, Hruska and Kleiner showed that \CAT$(0)$ spaces with isolated flats have the relative fellow traveling property relative to the isolated flats. 
Earlier, Epstein proved a version of relative fellow traveling for truncated hyperbolic spaces associated to finite volume cusped hyperbolic manifolds \cite[Theorem 11.3.1]{WPGroups}. 
Theorem~\ref{Prop: rel fellow traveling} is a version of relative fellow traveling for \CAT$(0)$ spaces with a proper cocompact action by a relatively hyperbolic group.  Theorem~\ref{Prop: rel fellow traveling} is presumed to be known to experts based on the works of \cite{DrutuSapir,Hruska04NPC2,Hruska2004,HruskaKleiner} and others, but the exact formulation used here proved difficult to find in the literature. Therefore, a proof is provided here. 

\begin{definition}\label{Def: rel hyp pair CAT(0)}
Let $\tilde{X}$ be a \CAT$(0)$ space, let $\delta\ge 0$, let $f:\reals^{\ge 0}\to\reals^{\ge 0}$ be a function and let $\mathcal{B}$ be a collection of subsets of $\tilde{X}$. The pair $(\tilde{X},\mathcal{B})$ is a \textbf{$(\delta,f)$--CAT$(0)$ Relatively Hyperbolic Pair} if 
\begin{enumerate}
\item every geodesic triangle in $\tilde{X}$ is $\delta$--thin relative to some $F\in \mathcal{B}$, \label{I: rel thin triangles}
\item for all $r\ge 0$ and $F_1,F_2\in \mathcal{B}$ with $F_1\ne F_2$, $\diam(\mathcal{N}_{r}(F_1)\cap \mathcal{N}_{r}(F_2) )\le f(r)$. \label{I: separation property}
\end{enumerate}
We say that a $(\delta,f)$--\CAT$(0)$ relatively hyperbolic pair has \textbf{the $L$--quasiconvexity property} if there exists $L\ge 0$ so that each $F\in\mathcal{B}$ is \textbf{$L$--quasiconvex} in the sense that any $\tilde{X}$--geodesic with endpoints in $F$ lies in $\mc{N}_L(F)$. 
The subspaces $\mathcal{B}$ are called \textbf{peripheral spaces}. 
\end{definition}

An immediate consequence of CAT(0) geometry is the following useful fact that we will use repeatedly: 
\begin{observation}\label{O: nbhd unif qc}
If $\tilde{Y}$ is an $L$--quasiconvex subspace of a \CAT$(0)$ space $\tilde{X}$, then for any $R\ge 0$, $\mc{N}_R(\tilde{Y})$ is also $L$--quasiconvex. In other words, if $x,y\in \mc{N}_R(\tilde{Y})$, then any geodesic between $x,y$ lies in $\mc{N}_{R+L}(\tilde{Y})$. 
\end{observation}



\begin{definition}
Let $(\tilde{X},\mc{B}_0)$ be a $(\delta,f_0)$--\CAT$(0)$ relatively hyperbolic pair, and let $R\ge 0$. An \textbf{$R$--thickening of $\mc{B}_0$} is a collection, $\mc{B}$, of  subspaces of $\tilde{X}$ so that there exists a bijection $B_0\in \mc{B}_0\leftrightarrow B\in \mc{B}$ where $B_0\subseteq B$, and $B\subseteq \mc{N}_R(B_0)$. 
\end{definition}

\begin{proposition}\label{P: fattening peripherals}
Let $(\tilde{X},\mathcal{B}_0)$ be a $(\delta,f_0)$--\CAT$(0)$ relatively hyperbolic pair, and let $\mc{B}$ be an $R$--thickening of $\mc{B}_0$. 
Let $f(r) = f_0(r+R)$. 
Then $(\tilde{X},\mc{B})$ is a $(\delta,f)$--\CAT$(0)$ relatively hyperbolic pair. 
\end{proposition}

\begin{proof}
 Let $F_1,F_2\in \mathcal{B}$ with $F_1\ne F_2$. Then there exist $F_{1,0},F_{2,0}\in \mc{B}_0$ so that $F_1\subseteq\mc{N}_R(F_{1,0})$ and $F_2\subseteq \mc{N}_R(F_{2,0})$. Then:
\[\diam(\mathcal{N}_r(F_1) \cap \mathcal{N}_r(F_2)) \le f(r).\]
A geodesic triangle $\triangle$ in $\tilde{X}$ is $\delta$--relatively thin relative to some $F_0$ in $\mathcal{B}_0$. Since $F_0$ is contained in some $F\in \mc{B}$ element, $\triangle$ is $\delta$--relatively thin relative to $F$. 
\end{proof}


\begin{definition}[Similar to {\cite[Definition 4.1.4]{HruskaKleiner}}]\label{D: RFTP}
Let $(\tilde{X},\mc{B})$ be a $(\delta,f)$--CAT$(0)$ relatively hyperbolic pair. 
The pair $(\tilde{X},\mc{B})$ has the \textbf{relative fellow traveling property} if for all $\lambda\ge 1$ and $\epsilon \ge 0$, there exist $U,V\ge 0$ depending on $\lambda,\epsilon$ such that for any $(\lambda,\epsilon)$--quasigeodesics $\sigma:[0,t_\sigma]\to \tilde{X}$ and $\gamma:[0,s_\gamma]\to \tilde{X}$ with the same endpoints, there exist partitions:
\[0 = s_0 \le s_1\le \ldots \le s_{2n+1} = s_\gamma \text{ and } 0 = t_0 \le t_1\le t_2\le \ldots \le t_{2n+1} = t_\sigma \]
such that:
\begin{enumerate}
\item for all $i$, $d(\gamma(s_i),\sigma(t_i))\le U$
\item if $i$ is even, then $d_{\text{Haus}}(\gamma([s_i,s_{i+1}]),\sigma([t_i,t_{i+1}]))\le U$ or
\item if $i$ is odd, $\gamma([s_i,s_{i+1}]),\sigma([t_i,t_{i+1}])\subseteq \mathcal{N}_V(F_i)$ for some $F_i\in\mc{B}$.  
\end{enumerate}
For a fixed $(\lambda,\epsilon)$, we say that $(\lambda,\epsilon)$--quasigeodesics \textbf{$(U,V)$--fellow travel relative to $\mc{B}$}. 
\end{definition}

All the \CAT$(0)$ relatively hyperbolic pairs we consider in later sections are of the form considered in the next proposition:
\begin{proposition}\label{P: situation is rel hyp pair}
Let $(G,\mc{P})$ be a relatively hyperbolic group pair so that $G$ acts geometrically on a \CAT$(0)$ cube complex $\tilde{X}$. Let $x\in\tilde{X}$ be a basepoint. Let $\mc{B}_{\mc{P}} = \{gPx:\, g\in G,\,P\in\mc{P}\}$, and let $\mc{B}$ be any  $R$--thickening of $\mc{B}_\mc{P}$. 
There exist $\delta,\,L(R)\ge 0$ and $f:\reals^{\ge 0}\to\reals^{\ge 0}$ so that $(\tilde{X},\mc{B})$ is a $(\delta,f)$--\CAT$(0)$ relatively hyperbolic pair that has the $L(R)$--quasiconvexity property.  
\end{proposition}

\begin{proof} By \cite[Theorem 1.1]{SW2015}, for each $P\in \mc{P}$, the convex hull of $Px$ lies in a bounded neighborhood of $Px$. Since $\mc{P}$ is finite, there is an $L\ge 0$ so that the convex hull of $gPx$ lies in $\mc{N}_L (gPx)$. Thus any geodesic between points in $gPx$ lies in $\mc{N}_L(gPx)$. By Observation~\ref{O: nbhd unif qc}, any $R$--thickening will have the $(L+R)$--quasiconvexity property because the $R$--neighborhood of each $B\in\mc{B}$ is $L$--quasiconvex. Let $B_{gP}$ be the convex hull of $gPx\in \mc{B}_{\mc{P}}$. Since $\mc{P}$ is finite, there is an $R$ (independent of $g,P$) so that each $B_{gP}\subseteq \mc{N}_R(gPx)$. Hence $\mc{B} = \{B_{gP}:\,g\in G,\,P\in\mc{P}\}$ is an $R$--thickening of $\mc{B}_{\mc{P}}$. 
By Proposition~\ref{P: fattening peripherals}, it suffices to show that there exist $\delta\ge 0$ and $f_\mc{P}:\reals^{\ge 0}\to\reals^{\ge 0}$ so that $(\tilde{X},\mc{B}_\mc{P})$ is a $(\delta, f_\mc{P})$--\CAT$(0)$ relatively hyperbolic pair. 
Proposition~\ref{Prop: thin triangles} implies Definition~\ref{Def: rel hyp pair CAT(0)}\eqref{I: rel thin triangles} holds. Corollary~\ref{peripheral separation} ensures that Definition~\ref{Def: rel hyp pair CAT(0)}\eqref{I: separation property} holds. 
\end{proof}


\begin{restatable}{theorem}{rftp} \label{Prop: rel fellow traveling}
Let $(G,\mathcal{P})$ be a relatively hyperbolic group pair where $G$ acts geometrically on a \CAT$(0)$ space $\tilde{X}$ with basepoint $x\in\tilde{X}$. If $\mc{B}$ is any $R$--thickening of $\{gPx\,|\,g\in G,\,P\in \mathcal{P}\}$ then $(\tilde{X},\mc{B})$ has the relative fellow traveling property. 
\end{restatable}

The remainder of this section is devoted to the proof of Theorem~\ref{Prop: rel fellow traveling}. The proof of Theorem~\ref{Prop: rel fellow traveling} is completely self-contained, so a reader who is not interested in the technical details may wish to skip to the next section. 
We now set the following standing hypotheses for the remainder of Section~\ref{RelFelTrv}:
\begin{hypotheses}\label{H: Sec4 baseline}
Let $(G,\mc{P})$ be a relatively hyperbolic group pair where $G$ acts geometrically on a \CAT$(0)$ cube complex $\tilde{X}$. Fix a basepoint $x$ and let $\mc{B}$ be an $R$--thickening of $\{gPx\,|\,g\in G,\,P\in \mathcal{P}\}$.
Fix $\delta\ge 0$, $L\ge 0$ and $f:\reals^{\ge 0} \to \reals^{\ge 0}$ so that $(\tilde{X},\mc{B})$ is a $(\delta,f)$--\CAT$(0)$ relatively hyperbolic pair with the $L$--quasiconvexity property. 
\end{hypotheses}

\subsection{Some geometric features of $(\tilde{X},\mc{B})$ under Hypotheses~\ref{H: Sec4 baseline}.}

In this section, we establish some geometric facts about the $(\delta,f)-$CAT$(0)$ relatively hyperbolic pair $(\tilde{X},\mc{B})$. 
\begin{definition}
Let $(\tilde{X},\mathcal{B})$ be a $(\delta,f)$--\CAT$(0)$ relatively hyperbolic pair. Let $\gamma\subseteq \tilde{X}$ and let $\mu\ge 0$. The $\mu$--\textbf{saturation of $\gamma$} (with respect to $\cal{B}$) is:
\[\text{Sat}_\mu (\gamma) = \bigcup \{B\in\cal{B}:\, \gamma\cap \cal{N}_\mu(B)\ne \emptyset\}\]
\end{definition}
In the following, $\gamma$ will usually be a quasigeodesic.




The following is a consequence of \cite[Lemma 8.10]{DrutuSapir} and the Milnor-\u{S}varc Lemma:

\begin{proposition}\label{P: DrutuSapir saturation}
Under Hypotheses~\ref{H: Sec4 baseline}, for every $\lambda\ge 1$ and $\epsilon\ge 0$:
there exists $u_{\lambda,\epsilon}$ so that if $\gamma,\sigma$ are $(\lambda,\epsilon)$--quasigeodesics with the same endpoints, then:
\[\sigma\subseteq \mc{N}_{u_{\lambda,\epsilon}}(\gamma)\cup \left(\bigcup_{F\in\text{Sat}_{u_{\lambda,\epsilon}}(\gamma)}\mc{N}_{u_{\lambda,\epsilon}}(F)\right)\]
\end{proposition}

\begin{definition}\label{D: coarsely relatively thin}
Let $\tilde{X}$ be a geodesic metric space and let $\mc{B}$ be a collection of subspaces of $\tilde{X}$. Let $B\in\mc{B}$, $\lambda\ge 1$ and $\epsilon>0$. Let $\triangle$ be a $(\lambda,\epsilon)$--quasigeodesic triangle. Let $\gamma_1,\gamma_2,\gamma_3$ be the sides of $\triangle$.  We say that $\triangle$ is \textbf{coarsely $\xi$--thin relative to $F\in\mc{B}$} if:
\begin{enumerate}
\item there exists a point $p\in\tilde{X}$ so that $d(p,\gamma_1),d(p,\gamma_2),d(p,\gamma_3)<\frac\xi2$ or 
\item there exist subpaths $c_i\subseteq \gamma_i$ so that $c_i\subseteq \mc{N}_\xi(F)$ and the distance between terminal point of $c_i$ and the initial point of $c_{i+1}$ (where indices are taken mod $3$) is less than $\xi$.
\end{enumerate}
\end{definition}

Proposition~\ref{qgd triangle} and the Milnor-\u{S}varc Lemma imply: 
\begin{proposition}\label{P: coarsely rel thin}
With Hypotheses~\ref{H: Sec4 baseline}, for all $\lambda\ge 1$ and $\epsilon\ge 0$, there exist $\delta_{\lambda,\epsilon}$ so that if $\triangle$ is a $(\lambda,\epsilon)$--quasigeodesic triangle, then there is an $F_\triangle \in \mc{B}$ so that $\triangle$ is coarsely $\delta_{\lambda,\epsilon}$--thin relative to $F_\triangle$. 
\end{proposition}

To simplify the proof of relative fellow traveling, we can make the following reduction:

\begin{proposition}\label{P: geodesic reduction}
Assume Hypotheses~\ref{H: Sec4 baseline}. To show that $(\tilde{X},\mc{B})$ has the relative fellow traveling property, it suffices to prove Definition~\ref{D: RFTP} holds in the special case that $\gamma$ is geodesic.
\end{proposition}

The proof of Proposition~\ref{P: geodesic reduction} is essentially identical to the reduction step in \cite[Proof of Theorem 13.1]{Hruska04NPC2}. 

Proposition~\ref{P: DrutuSapir saturation} suggests it might be possible for a quasigeodesic to remain far from a geodesic with the same endpoints by passing from one peripheral space to another. However, Lemma~\ref{L: clashing flats} shows that such a quasigeodesic must always come close to the geodesic with the same endpoints when transitioning from one peripheral space to another: 

\begin{lemma}\label{L: clashing flats}
Given $\mu\ge 0$, $\lambda\ge 1$ and $\epsilon>0$, there exists $D_\cap(\mu,\lambda,\epsilon)\ge\mu$ so that if $\sigma$ is a $(\lambda,\epsilon)$--quasigeodesic, $\gamma$ is a geodesic with the same endpoints as $\sigma$, and $\sigma(t) \in\mc{N}_{\mu}(F_1)\cap \mc{N}_{\mu}(F_2)$ for some distinct $F_1,F_2\in\text{Sat}_\mu(\gamma)$, then $\sigma(t)\in\mc{N}_{D_\cap(\mu,\lambda,\epsilon)}(\gamma)$. 
\end{lemma}

\begin{proof}
There exist $p_1,p_2\in\gamma$ so that $p_i\in\mc{N}_\mu(F_i)$. 
Let $\tau_1,\tau_2$ be geodesics so that $\tau_i$ joins $\sigma(t)$ to $p_i$. 
By Observation~\ref{O: nbhd unif qc} and the $L$--quasiconvexity of $F_i$, $\tau_i\subseteq \mc{N}_{\mu+L}(F_i)$. 
Let $\triangle$ be the geodesic triangle with sides $\tau_1,\tau_2$ and the subpath of $\gamma$ joining $p_1$ to $p_2$.
Then $\triangle$ is $\delta$--thin relative to some $F\in \mc{B}$. 

Recall corner segments and fat parts of relatively thin triangles from Definition~\ref{Def: corner segments}. 
Let $\tau_1'$ and $\tau_2'$ be the corner segments of $\triangle$ at $\sigma(t)$. Observe that $\tau_1'\subseteq \mc{N}_{\mu+L}(F_1)\cap \mc{N}_{\mu+L+\delta}(F_2)$, so $|\tau_1'|=|\tau_2'|\le f(\mu+L+\delta)$. 

Up to exchanging the indices of $F_1,F_2$, we may assume that $F\ne F_1$. 

The fat part of $\tau_1$ in $\triangle$ lies in $\mc{N}_{\delta}(F)\cap \mc{N}_{\mu+L}(F_1)$, so it has length at most $f(\mu+L+\delta)$. The fat part of $\tau_1$ also intersects $\mc{N}_\delta(\gamma)$. 
Therefore, $d(\sigma(t),\gamma)\le 2f(\mu+L+\delta)+\delta$. 

If necessary, we may enlarge $D_\cap(\mu,\lambda,\epsilon)$ to ensure $D_\cap(\mu,\lambda,\epsilon)\ge \mu$.
\end{proof}

\subsection{Relative Fellow Traveling}

\begin{hypotheses}\label{H: rftp hyp}
For the following subsection, we adopt the following baseline hypotheses in addition to Hypotheses~\ref{H: Sec4 baseline}:
\begin{enumerate}
\item Fix $\lambda\ge 1$ and $\epsilon\ge 0$. 
\item Let $\sigma:[0,t_\sigma]\to \tilde{X}$ be a $(\lambda,\epsilon)$--quasigeodesic triangle and let $\gamma:[0,s_\gamma]\to\tilde{X}$ be a geodesic that has the same endpoints as $\sigma$. 
\item Enlarge $\delta$ from Hypotheses~\ref{H: Sec4 baseline} so that all $(\lambda,\epsilon)$--quasigeodesic triangles are coarsely $\delta$--relatively thin relative to some $F\in\mc{B}$ (recall Definition~\ref{D: coarsely relatively thin} and Proposition~\ref{P: coarsely rel thin}) and all geodesic triangles are $\delta$--relatively thin relative to some $F\in \mc{B}$. 
\item Let $u = u_{\lambda,\epsilon}$ as in Proposition~\ref{P: DrutuSapir saturation}.
\item We abuse notation slightly and use $D_{\cap} = D_\cap(u+\epsilon+1,\lambda,\epsilon)$, (see Lemma~\ref{L: clashing flats}). Note $D_\cap\ge u+\epsilon+1\ge u$. \label{I: clashing flats hyp} 
\item Let $\epsilon' = \epsilon + 2D_\cap.$ 
\item Choose $D>> \delta_{\lambda,\epsilon'}+\epsilon'$ where $\delta_{\lambda,\epsilon'}$ is a constant such that all $(\lambda,\epsilon')$--quasigeodesic triangles are coarsely $\delta_{\lambda,\epsilon'}$--thin relative to some $F\in\mc{B}$ (recall Proposition~\ref{P: coarsely rel thin}). 
\item Let $\ell \ge f(D)$.  
\end{enumerate}
\end{hypotheses}

We first obtain a stability result for $(\lambda,\epsilon)$--quasigeodesics with endpoints in $\mc{N}_q(F)$ for some $F\in\mc{B}$:

\begin{proposition}\label{P: quasiquasiconvex}
Let $q\ge 0$. There exists $K(q)\ge 0$ so that if $\alpha:[a_1,a_2]\to \tilde{X}$ is a $(\lambda,\epsilon)$--quasigeodesic with $\alpha(a_1),\alpha(a_2)\in\mc{N}_{q}(F)$ for some $F\in\mc{B}$, then $\alpha([a_1,a_2])\subseteq \mc{N}_{K(q)}(F)$. 
\end{proposition}

\begin{proof}
Let $\beta:[b_1,b_2]\to \tilde{X}$ be a geodesic with $\beta(b_1) = \alpha(a_1)$ and $\beta(b_2)=\alpha(a_2)$. 
Since $\mc{N}_q(F)$ is $L$--quasiconvex by Observation~\ref{O: nbhd unif qc}, $\beta\subseteq \mc{N}_{L+q}(F)$. 
Let $y =\alpha(x)$ for some $a_1\le x\le a_2$. Let $\alpha_l =\alpha([a_1,x])$ and let $\alpha_r = \alpha([x,a_2])$. 
The sides $\alpha_l,\alpha_r,\beta$ define a $(\lambda,\epsilon)$--quasigeodesic triangle that is coarsely thin relative to some $F'\in\mc{B}$. 

If there exist $p,\,\alpha(a_l),\,\alpha(a_r),$ and $\beta(x_b)\in \beta$ so that $d(p,\alpha(a_l)),d(p,\alpha(a_r)),d(p,\beta(x_{b}))\le \frac\delta2$, then $|x-a_l|\le |a_l-a_r|\le \lambda( \delta+\epsilon)$. 
Then 
\[d(\beta(b),y) \le d(\alpha(a_l), y) + d(\alpha(a_l), \beta(x_b))\le \lambda(|x-a_l|)+ \epsilon +\delta \le\lambda^2\delta +\lambda\epsilon + \epsilon +\delta.\] 

If $F=F'$, then there exist $a_l\le x\le a_r$ so that $\alpha(a_l),\alpha(a_r)\in\mc{N}_\delta(F)$ and $d(\alpha(a_l),\alpha(a_r))\le \delta$. 
Hence $|a_l-x|\le |a_l-a_r|\le \lambda(\delta+\epsilon)$. 
Then $d(\alpha(a_l),y)\le \lambda^2\delta+\lambda^2 \epsilon+\epsilon$, so $y\in\mc{N}_{\delta+\lambda^2\delta+\lambda^2 \epsilon+\epsilon}(F)$. 

Finally, if $F\ne F'$, then there exist $a_l,a_r,b_l,b_r$ with $a_l\le x\le a_r$ so that $d(\alpha(a_l),\beta(b_l))\le \delta$, $(\alpha(a_r),\beta(b_r))\le \delta$ and $\beta([b_l,b_r]) \subseteq \mc{N}_{q+L}(F)\cap \mc{N}_\delta(F')$. 
Therefore $d(\alpha(a_l),\alpha(a_r))\le d(\beta(b_l),\beta(b_r))+2\delta\le f(q+L+\delta)+2\delta$.
Following computations similar to those in the previous cases:
\[|a_l-x| \le |a_l-a_r| \le \lambda (f(q+L+\delta)+2\delta)+\epsilon\]
\[d(\alpha(a_l),y)\le \lambda^2(f(q+L+\delta)+2\delta) + \lambda^2 \epsilon +\epsilon\]
\[d(\beta(b_l),y)\le \lambda^2(f(q+L+\delta)+2\delta) + \lambda^2 \epsilon +\epsilon+\delta\]
Therefore, $y\in \mc{N}_{q+L+\lambda^2(f(q+L+\delta)+2\delta) + \lambda^2 \epsilon +\epsilon+\delta}(F)$. 
Taking $K(q)$ to be the maximum of the constants generated in the three cases yields an appropriate constant. 
\end{proof}

Here is a brief overview of our strategy for the rest of this section:
\begin{enumerate}
\item We will partition $[0,t_{\sigma}]$ into subintervals so that on each subinterval either $\sigma$ is near an element of $\mc{B}$ or $\sigma$ does not stay close to any element of $\mc{B}$ for long (Proposition~\ref{P: first partition}).
\item In Lemma~\ref{L: partition touchup 1}, we alter our partition of $[0,t_\sigma]$ by widening the intervals where $\sigma$ remains near some element of $F$ so that $\sigma$ is near $\gamma$ at the endpoints of these intervals. In exchange, we need to calculate looser upper bounds (Proposition~\ref{P: flat depth}) on how close $\sigma$ is to an element of $\mc{B}$ on these intervals.   
\item On what remains of the subintervals where $\sigma$ is not near an element of $\mc{B}$, we prove that $\sigma$ lies within bounded Hausdorff distance of a part of $\gamma$ (Proposition~\ref{P: hausdorff interval}). 
\item We use this information to find subintervals of $[0,s_{\gamma}]$ that cover $[0,s_{\gamma}]$ where $\gamma$ is either close to an element of $\mc{B}$ or within bounded Hausdorff distance of $\sigma$. However, these subintervals may overlap. In Propositions~\ref{P: twist bound} and~\ref{P: augmented partition facts}, we show that overlapping can be controlled. 
\item In Propositions~\ref{P: untwist} and~\ref{P: final partition}, we rearrange the interval endpoints and delete some subintervals of $[0,t_{\sigma}]$ and $[0,s_{\gamma}]$ to eliminate any overlap and use the bounds found in Propositions~\ref{P: twist bound}, \ref{P: augmented partition facts} and \ref{P: untwist} to ultimately construct a partition that witnesses relative fellow traveling. 
\end{enumerate}

 In the following, we will use superscripts to help track the stages of partitioning and repartitioning $[0,t_\sigma]$ and covering $[0,s_\gamma]$ by subintervals. 

\begin{proposition}\label{P: first partition}
There exists a partition $0 = t_0^0\le  t_1^0\le t_2^0\le\ldots \le t_{2n+1}^0=t_{\sigma}$ and $F_0,F_1,\ldots,F_{n-1}\in\mc{B}$ with the following properties:
\begin{enumerate}
\item $\diam \{t\in [t_{2i}^0,t_{2i+1}^0]:\,\sigma(t)\in\mc{N}_D(F)\}\le \ell$ for all $F\in\mc{B}$,
\item $\sigma(t_{2i+1}^0),\sigma(t_{2i+2}^0)\in \mc{N}_{D+\epsilon}(F_i)$,
\item for all $F\in \mc{B}$, there do not exist $t_F^- < t_{2i+1}^0 \le t_{2i+2}^0 < t_F^+$ so that $\sigma(t_F^-),\sigma(t_F^+)\in \mc{N}_{u+\epsilon}(F)$.\label{I: widest flat part} 
\item $F_j\ne F_k$ for $j\ne k$. 
\end{enumerate}
\end{proposition}

It turns out the choice of $\ell$ is somewhat arbitrary, but it does affect how much the partition produced by Proposition~\ref{P: first partition} will need to be altered to give partitions of $[0,t_\sigma]$ and $[0,s_{\gamma}]$ that witness relative fellow traveling.

\begin{proof}
Let $m\in \naturals$ so that $(m-1)\ell \le t_{\sigma} < m\ell$. We proceed by induction on $m$. 

If $|t_\sigma|<\ell$, then setting $t_0^0 = 0$ and $t_1^0 = t_\sigma$ suffices. 

Assume that Proposition~\ref{P: first partition} holds for quasigeodesics parameterized over intervals of length less than $(m-1)\ell$. 
Find $0\le t_- \le t_+ \le t_\sigma$ so that $|t_+-t_-|$ realize 
$\sup_{F\in\mc{B}}\{|a-b|:\,\sigma(a),\sigma(b)\in\mc{N}_D(F)\}$. 
If $|t_+-t_-|<\ell$, then $t_0^0 = 0$ and $t_1^0=t_\sigma$ suffices. 

Otherwise, by the inductive hypothesis, we obtain partitions:
\[0= t_0^0 \le t_1^0 \le \ldots \le t_{2j+1}^0 = t_-\]
\[t_+ = t_{2j+2}^0 \le t_{2j+3}^0 \le\ldots \le t_{2n+1}^0 = t_\sigma\]
so that $\diam_{t\in [t_{2i}^0,t_{2i+1}^0]} \{\sigma(t)\in\mc{N}_D(F)\}\le \ell$ for all $F\in\mc{B}$, $|t_{2i+2} - t_{2i+1}|\ge \ell$ and $\sigma(t_{2i+1}^0),\sigma(t_{2i+2}^0)\in \mc{N}_{D+\epsilon}(F_i)$ for some $F_i\in \mc{B}$.    
Combining these partitions into a partition of $[0,t_\sigma]$ immediately satisfies the first two requirements. We obtain Item~\ref{I: widest flat part} because $D\ge \epsilon' \ge D_{\cap}\ge u+\epsilon$ (recall Hypotheses~\ref{H: rftp hyp}), the inductive hypothesis and $|t_+-t_-|$ is determined by a supremum. However, we need to check that if $k_1\le j$ and $k_2\ge j$ (with $k_1\ne k_2$), then $F_{k_1}\ne F_{k_2}$. 
If $F_{k_1}=F_{k_2}$, then there exist $t_l< t_- <t_+ < t_r$ so that $\sigma(t_l),\sigma(t_r)\in \mc{N}_D(F_{k_1})$ with $|t_{l} - t_{r}|> |t_-- t_+|\ge \ell$, contradicting hypothesis \eqref{I: widest flat part}. 
\end{proof}

In Proposition~\ref{P: first partition}, it is not guaranteed that the $\sigma(t_j^0)$ are near $\gamma$. To remedy this, we widen the intervals $[t_{2i+1}^0,t_{2i+2}^0]$ as necessary while shrinking $[t_{2i}^0,t_{2i+1}^0]$:

\begin{lemma}\label{L: partition touchup 1}
For $0\le j \le 2n+1$, there exist $t_j^1$ so that:
\begin{enumerate}
\item for all $0\le i\le n$, $\diam \{t\in [t_{2i}^1,t_{2i+1}^1]:\,\sigma(t)\in\mc{N}_D(F)\}\le \ell$ for all $F\in\mc{B}$,
\item $0 = t_0^1 \le t_1^1 \le \ldots \le t_{2n}^1 \le t_{2n+1}^1 = t_{\sigma}$,
\item $t_{2i}^0\le t_{2i}^1\le t_{2i+1}^1 \le t_{2i+1}^0$, and \label{I: rel subdiv order}
\item either $t_{2i}^1 = t_{2i+1}^1$ or $d(\sigma(t_{2i}^1),\gamma),d(\sigma(t_{2i+1}^1),\gamma)\le D_\cap$. 
\item $|t_{2i+1}^1  - t_{2i+1}^0|, \, |t_{2i+2}^1 - t_{2i+2}^0|\le \ell$. 
\end{enumerate}
\end{lemma}

\begin{proof}
For each $i$, we perform the following procedure to set $t_{2i}^1$. 
Consider $p_i = \sigma(t_{2i}^0)$. By Proposition~\ref{P: DrutuSapir saturation}, either $p_i\in\mc{N}_u(\gamma)$ or $p_i\in\mc{N}_u(F)$ for some $F\in\text{Sat}_u(\gamma)$ (where $u$ is as defined in Hypotheses~\ref{H: rftp hyp}). In the first case, we set $t_{2i}^1 = t_{2i}^0$ noting that $u\le D_\cap$. 

Suppose we are in the second case: let $t_{ext}^+ = \sup \{t\in [t_{2i}^0,t_{2i+1}^0]:\,\sigma(t)\in\mc{N}_u(F)\}$. Then $|t_{ext}^+-t_{2i}^0|\le \ell$ by Proposition~\ref{P: first partition}. One of the following holds: 
\begin{itemize}
\item  $t_{ext}^+ = t_{2i+1}^0$ and $\sigma(t_{ext}^+) \in \mc{N}_{u+\epsilon}(F)$ because $t_{ext}^+$ is a supremum,
\item $\sigma(t_{ext}^+) \in \mc{N}_{u+\epsilon+1}(\gamma)$, or 
\item $\sigma(t_{ext}^+) \in \mc{N}_{u+\epsilon+1}(F')$ for some $F'\in \mc{F}$ with $F'\ne F$. 
\end{itemize}
Indeed, if $t_{ext}^+ \ne t_{2i+1}^0$, then Proposition~\ref{P: DrutuSapir saturation} and the fact that $t_{ext}^+$ is a supremum ensure either the second or third possibility must hold.
In the case that $t_{ext}^+ = t_{2i+1}^0$, set $t_{2i+1}^1 = t_{2i}^1 = t_{2i+1}^0$.
Otherwise, set $t_{2i}^1 = t_{ext}^+$. In this case, either $\sigma(t_{2i}^1)$ lies in $\mc{N}_{D_\cap}(\gamma)$ directly or Lemma~\ref{L: clashing flats} with $\mu = u+\epsilon+1$ (recall Hypotheses~\ref{H: rftp hyp}\eqref{I: clashing flats hyp}) implies that $\sigma(t_{2i}^1)\in \mc{N}_{D_\cap}(\gamma)$. 

Proceeding similarly, if $\sigma(t_{2i+1}^0)\in \mc{N}_{u+\epsilon+1}(\gamma)$, we set $t_{2i+1}^1 = t_{2i+1}^0$. Otherwise, $\sigma(t_{2i+1}^1)\in \mc{N}_{u}(G)$ for some $G\in \text{Sat}_u(\gamma)$. We then set $t_{2i+1}^1 = \inf \{t\in [t_{2i}^1,t_{2i+1}^1]:\,\gamma(t) \in \mc{N}_u(G)\}$ where $G\in \text{Sat}_u(\gamma)$. 
As in the preceding argument, $|t_{2i+1}^1-t_{2i+1}^0|\le \ell$ and one of the following holds: $t_{2i+1}^1 = t_{2i}^1$ so that $\sigma(t_{2i+1}^1)\in \mc{N}_{D_\cap}(\gamma)$, $\sigma(t_{2i+1}^1)$ immediately lies in $\mc{N}_{D_\cap}(\gamma)$ or there exists $G'\in \text{Sat}_u(\gamma)$ so that $\sigma(t_{2i+1}^1) \in  \mc{N}_{u+\epsilon+1}(G')\cap \mc{N}_{u+\epsilon}(G) \subseteq \mc{N}_{D_\cap}(\gamma)$. In the third case, the final containment follows from Lemma~\ref{L: clashing flats} and Hypotheses~\ref{H: rftp hyp}\eqref{I: clashing flats hyp}.

Since $[t_{2i}^1,t_{2i+1}^1]\subseteq [t_{2i}^0,t_{2i+1}^0]$, we automatically retain the property that 
\[\diam \{t\in [t_{2i}^1,t_{2i+1}^1]:\,\sigma(t)\in\mc{N}_D(F)\}\le \ell\] 
for all $F\in\mc{B}$.
\end{proof} 

We now show that $\sigma([t_{2i+1}^1,t_{2i+2}^1])$ remains boundedly close to $F_i$. 

\begin{proposition}\label{P: flat depth} 
There exists $D_{depth}\ge 0$ so that for all $0\le i\le n$, $\sigma([t_{2i+1}^1,t_{2i+2}^1])\subseteq \mc{N}_{D_{depth}}(F_i)$, and if $t_{2i}^1 = t_{2i+1}^1$, $d(\sigma(t_{2i}^1),\gamma) \le f(D_{depth})+D_\cap$.
\end{proposition} 

\begin{proof}
Since $\sigma(t_{2i+1}^0)\in\mc{N}_{D+\epsilon}(F_i)$ and $|t_{2i+1}^0- t_{2i+1}^1|\le \ell$, $\sigma(t_{2i+1}^1)\in \mc{N}_{D+\epsilon+\lambda \ell+\epsilon}(F_i)$. 
Similarly, $\sigma(t_{2i+2}^1)\in\mc{N}_{D+\epsilon+\lambda \ell+\epsilon}(F_i)$. 
Set $D_{depth}= K(D+\lambda \ell+2\epsilon)$ where $K(D+\lambda \ell+2\epsilon)$ is determined (as a function of $\lambda,\epsilon,\ell$) as in Proposition~\ref{P: quasiquasiconvex}. 

Now suppose $t_{2i}^1 = t_{2i+1}^1$. By Proposition~\ref{P: DrutuSapir saturation}, if $t_{2i}^1\notin \mc{N}_u(\gamma)$, there exists $F\in\mc{B}$ so that $\sigma(t_{2i}^1)\in\mc{N}_u(F)$. 

Suppose first that $F\ne F_i$. 
Let $t_F = \sup\{t\in [0,t_\sigma]:\,\sigma([t_{2i}^1,t])\subseteq \mc{N}_u(F)\}$. 
By Lemma~\ref{L: partition touchup 1}\eqref{I: rel subdiv order}, $t_{2i}^1\le t_{2i+1}^0 \le t_{2i+2}^0\le t_{2i+2}^1$. Then 
Proposition~\ref{P: first partition}\eqref{I: widest flat part} implies $t_F\le t_{2i+2}^1$. 
Moreover, $F\ne F_i$ implies that $d(\sigma(t_{2i+1}^1),\sigma(t_F))\le f(D_{depth})$. 
Since $t_F$ is a supremum, there exists a $t>t_F$ with $d(\sigma(t),\sigma(t_F))\le \epsilon+1$ so that $\sigma(t)\in \mc{N}_{u}(\gamma)$ or $\sigma(t)\in \mc{N}_u(F')$ for some $F'\ne F$. Hence by Lemma~\ref{L: clashing flats}, $d(\sigma(t_F),\gamma)\le D_\cap$.  
Therefore, $d(\sigma(t_{2i}^1),\gamma)\le f(D_{depth})+D_\cap$. 

For the case $F\ne F_{i-1}$ set $t_F = \inf\{t\in [0,t_\sigma]:\, \sigma([t, t_{2i}^1])\subseteq \mc{N}_u(F)\}$ and then proceed using a similar argument to the case $F\ne F_i$. 
\end{proof}

We apply the bounds from Lemma~\ref{L: partition touchup 1} and Proposition~\ref{P: flat depth} to obtain the following.
\begin{cor}
Let $D_{endpoints} = f(D_{depth})+D_\cap \ge 0$. Then $d(\sigma(t_{j}^1),\gamma) \le D_{endpoints}$. 
\end{cor}

We now find $s_i^1$ in $[0,s_{\gamma}]$ so that $\gamma(s_i^1)$ is close to $\sigma(t_i^1)$. Let $0\le s_j^1 \le s_\gamma$ be such that $d(\gamma(s_j^1),\sigma(t_j^1))$ is at most $D_{endpoints}$ if $t_j^1 = t_{j\pm 1}^1$ or $D_{\cap}$ otherwise. 
If $t_{2i}^1 = t_{2i+1}^1$, ensure that $s_{2i}^1 = s_{2i+1}^1$. We may further assume that $s_0^1 = t_0^1 =0$, $t_{2n+1}^1 = t_\sigma$ and $s_{2n+1}^1=s_\gamma$. 

\begin{proposition}\label{P: hausdorff interval}
There exists $D_{hausdorff}$ so that $d_{haus}(\sigma([s_{2i}^1,s_{2i+1}^1]),\gamma([t_{2i}^1,s_{2i+1}^1])\le D_{hausdorff}$ for all $0\le i \le n$. 
\end{proposition}

\begin{proof}
If $t_{2i+1}^1=t_{2i}^1$, then $D_{hausdorff} = D_{endpoints}$ suffices. Otherwise, Lemma~\ref{L: partition touchup 1} implies $d(\sigma(t_{2i}^1),\gamma(s_{2i}^1)),d(\sigma(t_{2i+1}^1),\gamma(s_{2i+1}^1))\le D_\cap$. 

Recall from Hypotheses~\ref{H: rftp hyp} that $\epsilon' =\epsilon+2D_\cap$. Construct $\sigma_i$, a $(\lambda,\epsilon')$--quasigeodesic from $\sigma([t_{2i}^1,t_{2i+1}^1])$ by adding geodesics of length at most $D_\cap$ connecting $\sigma(t_{2i}^1)$ and $\sigma(t_{2i+1}^1)$ to $\gamma(s_{2i}^1)$ and $\gamma(s_{2i+1}^1)$ respectively.

Let $y\in\sigma_i$. Partition $\sigma_i$ into $\sigma_l$ and $\sigma_r$ so that $\sigma_l$ is from $\gamma(s_{2i}^1)$ to $y$ and $\sigma_r$ is from $y$ to $\sigma(s_{2i+1}^1)$. 
The triangle bounded by $\gamma([s_{2i}^1,s_{2i+1}^1])$, $\sigma_l$ and $\sigma_r$ is $\delta_{\lambda,\epsilon'}$--coarsely thin relative to some $F\in \mc{B}$. 

There are two possibilities:

\textbf{Case: there exist points $p_l\in\sigma_l$, $p_r\in\sigma_r$ and $p_\gamma$ in $\gamma$ so that $d(p_l,p_r),d(p_r,p_\gamma),d(p_l,p_\gamma)\le \delta_{\lambda,\epsilon'}$.} 
Since $\sigma_i$ is quasigeodesic, $d(y,p_l) \le \lambda (\lambda(\delta_{\lambda,\epsilon'}+\epsilon'))+\epsilon'$ (a similar computation was carried out in more detail in the proof of Proposition~\ref{P: quasiquasiconvex}).  
Then $d(y,\gamma) \le d(y,p_\gamma)\le \delta_{\lambda,\epsilon'} + \lambda (\lambda(\delta_{\lambda,\epsilon'}+\epsilon'))+\epsilon'$. 

\textbf{Case: there exist $p_l,p_{l,\gamma}\in \sigma_l$, $p_r\in\sigma_r$ and $F\in\mc{B}$ so that the interval of $\sigma_l$ between $p_l$ and $p_{l,\gamma}$ lies in $\mc{N}_{\delta_{\lambda,\epsilon'}}(F)$, $d(p_l,\gamma([s_{2i}^1,s_{2i+1}^1])\le \delta_{\lambda,\epsilon'}$ and $d(p_r,p_l)\le \delta_{\lambda,\epsilon'}$.} Recall that \[\diam \{t\in [t_{2i}^1,t_{2i+1}^1]:\,\sigma(t)\in\mc{N}_D(F)\}\le \ell\] 
so $d(p_l,p_{l,\gamma})\le \lambda\ell+3\epsilon'$ where the additional $2\epsilon'$ is accounting for the length of the segment linking $\gamma(s_{2i}^1)$ to $\sigma(t_{2i}^1)$ and the segment linking $\gamma(s_{2i+1}^1)$ to $\sigma(t_{2i+1}^1)$. We have that $d(y,p_l) \le \lambda (\lambda(\delta_{\lambda,\epsilon'}+\epsilon'))+\epsilon'$ following the computation from the previous case. 
Hence
\[d(y,\gamma)\le \delta_{\lambda,\epsilon'} + \ell + \epsilon' + \lambda (\lambda(\delta_{\lambda,\epsilon'}+\epsilon'))+\epsilon'.\]

From the two previous cases, we determine that $d(y,\gamma)$ is bounded as a function of $\lambda,\epsilon'$. 

Now consider $x\in\gamma$. We will bound $d(x,\sigma)$. 
Similar to the previous case, divide $\gamma|_{[s_{2i}^1,s_{2i+1}^1]}$ into two segments $\gamma_l$ from $\gamma(s_{2i}^1)$ to $x$ and $\gamma_r$ from $x$ to $\gamma(s_{2i+1}^1)$ and consider the quasigeodesic triangle with sides $\gamma_l,\gamma_r,\sigma_i$ that is $\delta_{\lambda,\epsilon'}$--coarsely thin relative to some $F\in\mc{B}$. 
There are two possibilities:

\textbf{Case: there exist $x_l,x_r,x_\sigma$ so that $x_l\in\gamma_l$, $x_r\in\gamma_r$, $x_\sigma\in \sigma_i$ with $d(x_l,\sigma),d(x_r,x_l)\le\delta_{\lambda,\epsilon'}$.}
Then $d(x_l,x)\le d(x_l,x_r)\le\delta_{\lambda,\epsilon'}$ because $\gamma$ is geodesic. 
Hence we have $d(x,\sigma_i)\le d(x,x_\sigma)\le d(x,x_l)+d(x_l,x_\sigma)\le 2\delta_{\lambda,\epsilon'}$. Thus $d(x,\sigma)\le 2\delta_{\lambda,\epsilon'}+D_\cap$. 

\textbf{Case: there exist $x_l,x_r,x_\sigma$ and $F\in\mc{B}$ so that $x_l\in\gamma_l$, $x_r\in\gamma_r$, $p_{\sigma_l},p_{\sigma_r}\in \sigma_i$ so that $p_{\sigma_l},p_{\sigma_r}\in \mc{N}_{\delta_{\lambda,\epsilon'}}(F)$ and $d(x_l,p_{\sigma_l}),d(x_r,p_{\sigma_r})\le\delta_{\lambda,\epsilon'}$.}
Since $d(p_{\sigma_l},\sigma),d(p_{\sigma_r},\sigma)\le \epsilon'$, there exist $t_l,t_r$ so that $p_{\sigma_l} = \sigma(t_l),\,p_{\sigma_r} = \sigma(t_r)\in\mc{N}_{\delta_{\lambda,\epsilon'}+\epsilon'}(F)$. 
Then by Proposition~\ref{P: first partition} and Lemma~\ref{L: partition touchup 1} and the fact that $ D>>\delta_{\lambda,\epsilon'}+\epsilon'$, $|t_l-t_r|\le \ell$.
It follows that $d(\sigma(t_l),\sigma(t_r))\le \lambda \ell +\epsilon'$. 
Hence:
\[d(x,\sigma(t_l))\le d(x_l,x_r)+d(x_l,p_{\sigma_l}) + d(p_{\sigma_l},\sigma(t_l)) \le  d(x_l,x_r)+ \delta_{\lambda,\epsilon'}+ \epsilon'\]
\[\le d(\sigma(t_l),\sigma(t_r))+2\epsilon' + 3\delta_{\lambda,\epsilon'} \le \lambda \ell +3\epsilon' + 3\delta_{\lambda,\epsilon'}\]

Taking the largest constant from the four cases above yields an acceptable value for $D_{Hausdorff}$. 
\end{proof}

Unfortunately, it is possible that $j<k$ and $s_j^1> s_k^1$, but this behavior can be controlled:

\begin{proposition}\label{P: twist bound}
There exists $D_{outorder}$ so that if $j<k$ and $s_j^1>s_k^1$, then $|t_j-t_k|\le D_{outorder}$. 
\end{proposition}

\begin{proof}
It suffices to consider the case where $k$ is the largest index such that $j<k$ and $s_j^1 > s_k^1$. 

By construction, $d(\sigma(t_j^1),\gamma(s_j^1))\le D_{endpoints}$. Since $k$ is largest, $s_j^1\in [s_{k}^1,s_{k + 1}^1]$ where $s_{k}^1\le s_{k + 1}^1$.
Since $s_0^1 = 0$ and $s_{2n}^1=s_{\gamma}$, there exists $h_- <  j$ so that $s_{k}^1$ lies in $[s_{h_-}^1,s_{h_- + 1}^1]$ where $s_{h_-}^1\le s_{h_- + 1}^1$.  


\textbf{Case: $k$ is even.}

Then $d(\gamma(s_j^1),\sigma(t_j^1))\le D_{endpoints}$ and there exists $t_+\in [t_{k}^1,t_{k + 1}^1]$ such that $d(\gamma(s_{j}^1),\sigma(t_+))\le D_{hausdorff}$. 
Hence $d(\sigma(t_{j}^1),\sigma(t_+))\le D_{hausdorff}+D_{endpoints}$. 
We then obtain:
\[|t_k^1-t_j^1|\le |t_{j}^1 - t_+| \le \lambda(D_{hausdorff}+D_{endpoints})+\epsilon .\]

\textbf{Case: $h_-$ is even.}

Then $d(\gamma(s_{k}^1),\sigma(t_{k}^1))\le D_{endpoints}$ and there exists $t_-\in [t_{h_-}^1,t_{h_-+1}^1]$ so that $d(\sigma(t_-),\gamma(s_{k}^1))\le D_{hausdorff}$. Similar to the previous case, we conclude:
\[|t_{h_-}^1-t_{h_-+1}^1|\le |t_k^1-t_j^1|\le \lambda(D_{hausdorff}+D_{endpoints} )+\epsilon.\]

\textbf{Case: $h_-$ and $k$ are both odd.}

Set $h_- = 2i_-+1$ and $k = 2i_+ + 1$. 
Observe that $\gamma([s_{h_-}^1,s_{h_- + 1}^1])\subseteq \mc{N}_{D_{endpoints} + D_{depth}}(F_{i_-})$ and similarly $\gamma([s_{k}^1,s_{k + 1}^1])\subseteq \mc{N}_{D_{endpoints} + D_{depth}}(F_{i_+})$.

We have $s_{h_-}^1 \le s_k^1<s_j^1 \le s_{k+1}^1$. 
If $s_{h_-}^1\le s_k^1 \le s_j^1\le s_{h_- + 1}^1,s_{k+1}^1$, then:
\[\gamma([s_k^1,s_j^1]) \subseteq \mc{N}_{D_{endpoints} + D_{depth}}(F_{i_-})\cap \mc{N}_{D_{endpoints} + D_{depth}}(F_{i_+}).\]
Therefore, $d(\gamma(s_k^1),\gamma(s_j^1))\le f(D_{endpoints}+D_{depth})$, and $d(\sigma(t_j^1),\sigma(t_k^1)) \le 2D_{endpoints} + f(D_{endpoints}+D_{depth})$. Then 
\[|t_j^1-t_k^1| \le \lambda( 2D_{endpoints} + f(D_{endpoints}+D_{depth}))+\epsilon.\]
 
Otherwise $s_{h_-}^1\le s_k^1 \le s_{h_-+1}^1\le s_j^1 \le s_{k+1}^1$ so that:
\[\gamma([s_k^1,s_{h_-+1}^1]) \subseteq \mc{N}_{D_{endpoints} + D_{depth}}(F_{i_-})\cap \mc{N}_{D_{endpoints} + D_{depth}}(F_{i_+}).\]
We see $d(\sigma(t_k^1),\sigma(t_{h_-+1}^1))\le 2D_{endpoints} + f(D_{endpoints}+D_{depth})$. Recalling $h_-<j$, then:
\[|t_j^1-t_k^1| \le |t_{h_-+1}^1 - t_k^1|  \le \lambda( 2D_{endpoints} + f(D_{endpoints}+D_{depth}))+\epsilon.\]

Taking $D_{outorder}$ to be the maximum of the bounds found in each of the three cases therefore suffices. 
\end{proof}

\begin{definition}
An \textbf{augmented partition} of $[0,t_\sigma]$ is a partition:
\[ 0 \le t_1\le t_2\le\ldots \le t_m =t_\sigma\]
together with choices $0=s_0, s_1,s_2,\ldots,s_m = s_\gamma$ where $s_i \in [0,s_\gamma]$. 
We denote such an augmented partition by:
\begin{equation}(t_0,s_0) \le (t_1,s_1)\le \ldots \le (t_{m-1},s_{m-1}) \le (t_m,s_m).\label{E: augpartition def}\end{equation}
We call  $t_j\le t_{j+1} \le \ldots \le t_k$ a \textbf{maximal crossover subinterval} of the augmented partition \eqref{E: augpartition def} if $s_h < s_j$ for all $h\le j$ and $k$ is the largest index so that $s_k<s_j$. 
\end{definition}

In Propositions~\ref{P: augmented partition facts} and~\ref{P: untwist}, we explain how to take an augmented partition like $(t_0^1,s_0^1)\le \ldots \le (t_{2n+1}^1,s_{2n+1}^1)$ and obtain an augmented partition with similar properties that has one fewer maximal crossover interval from an augmented partition. Then, in Proposition~\ref{P: final partition}, we work on $(t_0^1,s_0^1)\le \ldots \le (t_{2n+1}^1,s_{2n+1}^1)$ from left to right using Proposition~\ref{P: untwist} to obtain a new augmented partition with similar properties but no maximal crossover intervals. 

\begin{proposition}\label{P: augmented partition facts}
Let $t^1_j\le t^1_{j+1}\le \ldots \le t^1_k$ be a maximal crossover subinterval of an augmented partition 
\[(t_0,s_0)\le (t_1,s_1)\le \ldots \le (t_{i_j-1},s_{i_j-1})\le (t_j^1,s_j^1)\le (t^1_{j+1},s_{j+1}^1)\le \ldots \le (t_k^1,s_k^1)\le \ldots \le (t_{2n+1}^1,s_{2n+1}^1)\] of $[0,t_\sigma]$. Then:
\begin{itemize}
\item $d(\sigma(t^1_k),\gamma(s_j^1))\le \lambda D_{outorder} + \epsilon + D_{endpoints}$,
\item $d(\sigma(t^1_j),\gamma(s^1_k)\le \lambda D_{outorder} + \epsilon + D_{endpoints}$, and 
\item $d_{haus}(\sigma([t^1_j,t^1_k]),\gamma([s_k^1,s_j^1])) \le \lambda D_{outorder} + \epsilon + 3D_{endpoints}$. 
\end{itemize}
\end{proposition}

\begin{proof}
Recall $d(\sigma(t_j^1),\gamma(s_j^1)),d(\sigma(t_{j+1}^1),\gamma(s_{j+1}^1)),\ldots,d(\sigma(t_k^1),\gamma(s_k^1))\le D_{endpoints}$. 
By Proposition~\ref{P: twist bound}, $|t^1_j-t^1_k|\le D_{outorder}$. 
Then $d(\sigma(t^1_j),\sigma(t^1_k))\le \lambda D_{outorder} +\epsilon$. 
We can conclude then that $d(\gamma(s_j^1),\gamma(s_k^1))\le \lambda D_{outorder} +\epsilon + 2D_{endpoints}$. 
Therefore, for all $s_k^1\le s\le s_j^1$:
\[d(\gamma(s),\sigma(t_k^1))\le d(\gamma(s_j^1),\gamma(s_k^1)) + d(\gamma(s_k^1),\sigma(t_k^1)) \le \lambda D_{outorder} +\epsilon + 2D_{endpoints} + D_{endpoints}\]

Similarly for all $t_j^1\le t\le t_k^1$, $|t-t_k^1|\le D_{outorder}$ so:
\[d(\sigma(t),\sigma(t_k^1))\le \lambda D_{outorder} +\epsilon\]
Therefore,
\[d(\sigma(t),\gamma(s_k^1)) \le \lambda D_{outorder} + \epsilon + D_{endpoints}.\]
A similar argument will also show that $d(\sigma(t_k^1),\gamma(s_j^1))\le \lambda D_{outorder} + \epsilon + D_{endpoints}$. 
\end{proof}

\begin{proposition}\label{P: untwist}
Let $t_j^1\le t_{j+1}^1\le \ldots \le t_k^1$ be a maximal crossover subinterval of an augmented partition 
\begin{multline}\label{E: oldpartition}(t_0,s_0)\le (t_1,s_1)\le (t_2,s_2)\le \ldots \le (t_{i_j-1},s_{i_j-1})\\ \le (t_j^1,s_j^1)\le (t^1_{j+1},s_{j+1}^1)\le \ldots \le (t^1_k,s_1^k)\le \ldots \le (t_{2n+1}^1,s_{2n+1}^1)\end{multline} of $[0,t^1_\sigma]$ so that $t_0,t_1,t_2, \ldots, t_{i_j-1}$ are not contained in any maximal crossover subintervals of \eqref{E: oldpartition}. 
There is a new augmented partition 
\begin{multline}\label{newaug}
0=(t_0,s_0)\le \ldots \le (t_{i_j-1},s_{i_j-1})\le  (t^1_j,s_k^1) \\ \le (t^1_k,s_j^1) \le (t^1_{k+1},s_{k+1}^1)\le \ldots\le (t^1_{2n+1},s_{2n+1}^1)
\end{multline}
that has the following properties:
\begin{itemize}
\item $t_0,t_1,\ldots,t_{i_{j}-1},t_j^1,t_k^1$ are not contained in any maximal crossover subinterval of \eqref{newaug},
\item  $d(\sigma(t^1_k),\gamma(s_j^1)),d(\sigma(t_j^1),\gamma(s_k^1))\le \lambda D_{outorder} + \epsilon + D_{endpoints}$, and
\item $d_{haus}(\sigma([t_j^1,t_k^1]),\gamma([s_k^1,s_j^1]))<\lambda D_{outorder} + \epsilon + 3D_{endpoints}$. 
\end{itemize}
\end{proposition}

\begin{proof}
Since $t_0,t_1,t_2, \ldots, t_{i_j-1}$ are not contained in any maximal crossover subinterval, $s_0\le s_1 \le s_2\le\ldots\le s_{i_j-1} \le  s_k$ and $s_k\le s_j$ by hypothesis. 
Moreover, for all $k'> k$, $s_{k'}^1\ge s_j^1\ge s_k^1$ because $t_j^1\le\ldots \le t_k^1$ is a \emph{maximal} crossover subinterval. 
Therefore, $t_j^1$ and $t_k^1$ cannot be contained in a maximal crossover subinterval of the augmented partition \eqref{newaug}.  

From Proposition~\ref{P: augmented partition facts}, we immediately obtain $d(\sigma(t_k^1),\gamma(s_j^1))\le \lambda D_{outorder} + \epsilon + D_{endpoints}$ and $d_{haus}(\sigma([t_j^1,t_k^1]),\gamma([s_k^1,s_j^1])) \le \lambda D_{outorder} + \epsilon + 3D_{endpoints}$. 
\end{proof}

\begin{proposition}\label{P: final partition}
There exist partitions:
\[ 0 = t_0^2 \le t_1^2 \le t_2^2 \le t_3^2 \le\ldots \le t_{n'}^2 = t_\sigma\]
\[ 0 = s_0^2 \le s_1^2 \le s_2^2 \le s_3^2\le\ldots \le s_{n'}^2 = s_\gamma\]
so that for $0\le j\le n'$:
\begin{enumerate}
\item $d(\sigma(t_j^2),\gamma(s_j^2)) \le \lambda D_{outorder}+ \epsilon + D_{endpoints}$.
\item For each $j$, one of the following holds:
\[d_{haus}(\sigma([t_j^2,t_{j+1}^2]),\gamma([s_j^2,s_{j+1}^2]))\le \lambda D_{outorder} + \epsilon + 3D_{endpoints} \text{ OR }\]
\[\sigma([t_j^2,t_{j+1}^2]),\gamma([s_j^2,s_{j+1}^2])\subseteq \mc{N}_{K(D_{depth} + D_{endpoints})}(F_j^2)\text{ for some }F_j^2\in \mc{B}.\] 
\item If $j\ne j'$, then $F_j^2\ne F_{j'}^2$. 
\end{enumerate}
\end{proposition}

\begin{proof}[Proof sketch]
We can obtain the desired partition by starting with the partition from Lemma~\ref{L: partition touchup 1} and then working left to right using Proposition~\ref{P: untwist} to eliminate any maximal crossover subintervals. Immediately, $s_0^1=0$, so $t_0^1$ is not contained in any maximal crossover subintervals. The bound on $d(\sigma(t_j^2),\gamma(s_j^2))$ is implied by Proposition~\ref{P: untwist}. 
One of the following holds: 
\begin{itemize}
\item $t_j^2 = t_{2i}^1,\,t_{j+1}^2 = t_{2i+1}^1$, $s_j^2 = s_{2i}^1$ and $s_{j+1}^2 = s_{2i+1}^1$ for some $i$, 
\item  $t_j^2 = t_{2i+1}^1,\,t_{j+1}^2 = t_{2i+2}^1$, $s_j^2 = s_{2i+1}^1$ and $s_{j+1}^2 = s_{2i+2}^1$ for some $i$,  or
\item Proposition~\ref{P: untwist} implies that  $d_{haus}(\sigma([t_j^2,t_{j+1}^2]),\gamma([s_j^2,s_{j+1}^2]))\le \lambda D_{outorder} + \epsilon + 3D_{endpoints}$. 
\end{itemize}
In the first case, Proposition~\ref{P: hausdorff interval} implies that $d_{haus}(\sigma([t_j^2,t_{j+1}^2]),\gamma([s_j^2,s_{j+1}^2]))$ is bounded appropriately. 
In the second case, Proposition~\ref{P: flat depth} implies that $\sigma([t_j^2,t_{j+1}^2])\subseteq \mc{N}_{D_{depth}}(F_i)$, so set $F_j^2=F_i$. 
Since the endpoints of $\gamma([s_j^2,s_{j+1}^2])$ are within $D_{endpoints}$ of the endpoints of $\sigma([t_j^2,t_{j+1}^2])$ and $\gamma$ is geodesic:
\[\gamma([s_j^2,s_{j+1}^2])\subseteq \mc{N}_{K(D_{depth} + D_{endpoints})}(F_j^2)\]
Since the $F_i$ are distinct, if $j\ne j'$, then $F_j^2\ne F_{j'}^2$. 
\end{proof}

In the partition from Proposition~\ref{P: final partition}, we call interval $[t_j^2,t_{j+1}^2]$ a \textbf{Hausdorff interval} if $d_{haus}(\sigma([t_j^2,t_{j+1}^2]),\gamma([s_j^2,s_{j+1}^2]))\le \lambda D_{outorder} + \epsilon + 3D_{endpoints}$. Otherwise, if $\sigma([t_j^2,t_{j+1}^2]),\gamma([s_j^2,s_{j+1}^2])\subseteq \mc{N}_{K(D_{depth} + D_{endpoints})}(F_j^2)$, we call $[t_j^2,t_{j+1}^2]$ a \textbf{peripheral interval}. 

\rftp*

\begin{proof}
By Proposition~\ref{P: situation is rel hyp pair}, $(\tilde{X},\mc{B})$ is a $(\delta,f)$--\CAT$(0)$ relatively hyperbolic pair and there exists $L(R)$ so that Hypotheses~\ref{H: Sec4 baseline} hold. 

Given $(\lambda,\epsilon)$--quasigeodesics $\gamma,\sigma$ with the same endpoints, we can reduce to the case where $\gamma$ is geodesic by Proposition~\ref{P: geodesic reduction}. 
Proposition~\ref{P: final partition} nearly provides the partition for relative fellow traveling except that the intervals $[t_j^2,t_{j+1}^2]$ as constructed in Proposition~\ref{P: final partition} do not alternate between Hausdorff intervals and peripheral intervals. 
This can be easily remedied by turning any two adjacent Hausdorff intervals into a single Hausdorff interval. In other words, if $[t_j^2,t_{j+1}^2]$ and $[t_{j+1}^2,t_{j+2}^2]$ are both Hausdorff intervals, we remove these two intervals from the partition and replace them with the single interval $[t_j^2,t_{j+2}^2]$. Likewise, replace $[s_j^2,s_{j+1}^2]$ and $[s_{j+1}^2,s_{j+2}^2]$ with $[s_{j}^2,s_{j+2}^2]$. 
It is easy to check that $d_{haus}(\sigma([t_j^2,t_{j+2}^2]),\gamma([s_{j}^2,s_{j+2}^2]))\le \lambda D_{outorder} + \epsilon + 3D_{endpoints}$ in this case. 
Repeat this process until no adjacent Hausdorff intervals remain.
\end{proof}

%% file: foursquare.pdf_tex
\begingroup%
  \makeatletter%
  \providecommand\color[2][]{%
    \errmessage{(Inkscape) Color is used for the text in Inkscape, but the package 'color.sty' is not loaded}%
    \renewcommand\color[2][]{}%
  }%
  \providecommand\transparent[1]{%
    \errmessage{(Inkscape) Transparency is used (non-zero) for the text in Inkscape, but the package 'transparent.sty' is not loaded}%
    \renewcommand\transparent[1]{}%
  }%
  \providecommand\rotatebox[2]{#2}%
  \ifx\svgwidth\undefined%
    \setlength{\unitlength}{340.15748031bp}%
    \ifx\svgscale\undefined%
      \relax%
    \else%
      \setlength{\unitlength}{\unitlength * \real{\svgscale}}%
    \fi%
  \else%
    \setlength{\unitlength}{\svgwidth}%
  \fi%
  \global\let\svgwidth\undefined%
  \global\let\svgscale\undefined%
  \makeatother%
  \begin{picture}(1,0.5344511)%
    \put(0,0){\includegraphics[width=\unitlength,page=1]{foursquare.pdf}}%
    \put(0.36003074,0.48851614){\color[rgb]{0,0,0}\makebox(0,0)[lb]{\smash{$\gamma$}}}%
    \put(0.42985155,0.01104555){\color[rgb]{0,0,0}\makebox(0,0)[lb]{\smash{$\alpha$}}}%
    \put(0,0){\includegraphics[width=\unitlength,page=2]{foursquare.pdf}}%
    \put(0.69686498,0.13646497){\color[rgb]{0,0,1}\makebox(0,0)[lb]{\smash{$\rho_3$}}}%
    \put(0.48794249,0.25584929){\color[rgb]{0,1,0}\makebox(0,0)[lb]{\smash{$\rho_2$}}}%
    \put(0.19546425,0.39820254){\color[rgb]{1,0,0}\makebox(0,0)[lb]{\smash{$\rho_1$}}}%
    \put(0.86101839,0.25584929){\color[rgb]{0,0,0}\makebox(0,0)[lb]{\smash{$\sigma_2$}}}%
    \put(-0.00451772,0.25584929){\color[rgb]{0,0,0}\makebox(0,0)[lb]{\smash{$\sigma_1$}}}%
    \put(0.01611327,0.02336988){\color[rgb]{0,0,0}\makebox(0,0)[lb]{\smash{p}}}%
    \put(0.83344347,0.02229026){\color[rgb]{0,0,0}\makebox(0,0)[lb]{\smash{q}}}%
    \put(0.84609536,0.46477174){\color[rgb]{0,0,0}\makebox(0,0)[lb]{\smash{r}}}%
    \put(0.01040532,0.46477174){\color[rgb]{0,0,0}\makebox(0,0)[lb]{\smash{s}}}%
    \put(0,0){\includegraphics[width=\unitlength,page=3]{foursquare.pdf}}%
    \put(0.14471264,0.47969478){\color[rgb]{1,0.50196078,0}\makebox(0,0)[lb]{\smash{$\gamma_1$}}}%
    \put(0.38514219,0.42296679){\color[rgb]{0,0.50196078,1}\makebox(0,0)[lb]{\smash{$\gamma_2$}}}%
    \put(0.68940349,0.47969478){\color[rgb]{1,0,1}\makebox(0,0)[lb]{\smash{$\gamma_3$}}}%
    \put(0.66483538,0.34185024){\color[rgb]{0,0,0}\makebox(0,0)[lb]{\smash{$\triangle_2$}}}%
    \put(0.16619317,0.15936396){\color[rgb]{0,0,0}\makebox(0,0)[lb]{\smash{$\triangle_1$}}}%
  \end{picture}%
\endgroup%

%% file: dbltrianglediagramCase1.pdf_tex
\begingroup%
  \makeatletter%
  \providecommand\color[2][]{%
    \errmessage{(Inkscape) Color is used for the text in Inkscape, but the package 'color.sty' is not loaded}%
    \renewcommand\color[2][]{}%
  }%
  \providecommand\transparent[1]{%
    \errmessage{(Inkscape) Transparency is used (non-zero) for the text in Inkscape, but the package 'transparent.sty' is not loaded}%
    \renewcommand\transparent[1]{}%
  }%
  \providecommand\rotatebox[2]{#2}%
  \ifx\svgwidth\undefined%
    \setlength{\unitlength}{255.44594832bp}%
    \ifx\svgscale\undefined%
      \relax%
    \else%
      \setlength{\unitlength}{\unitlength * \real{\svgscale}}%
    \fi%
  \else%
    \setlength{\unitlength}{\svgwidth}%
  \fi%
  \global\let\svgwidth\undefined%
  \global\let\svgscale\undefined%
  \makeatother%
  \begin{picture}(1,0.67073388)%
    \put(0,0){\includegraphics[width=\unitlength,page=1]{dbltrianglediagramCase1.pdf}}%
    \put(0.22952975,0.22669338){\color[rgb]{0,0,1}\makebox(0,0)[lb]{\smash{fat in $\triangle_i^2$}}}%
    \put(0.40739513,0.02758687){\color[rgb]{0,0,1}\makebox(0,0)[lb]{\smash{fat in $\triangle_i^2$}}}%
    \put(0.74401904,0.2678127){\color[rgb]{0,0,1}\makebox(0,0)[lb]{\smash{fat in $\triangle_i^2$}}}%
    \put(0,0){\includegraphics[width=\unitlength,page=2]{dbltrianglediagramCase1.pdf}}%
    \put(0.37126974,0.3864928){\color[rgb]{0,0,0}\makebox(0,0)[lb]{\smash{fat in $\triangle_i^1$}}}%
    \put(0.11491695,0.43126059){\color[rgb]{0,0,0}\makebox(0,0)[lb]{\smash{fat in $\triangle_i^1$}}}%
    \put(0.32385275,0.64016882){\color[rgb]{0,0,0}\makebox(0,0)[lb]{\smash{fat in $\triangle_i^1$}}}%
    \put(0.11392838,0.65136794){\color[rgb]{0,0,0}\makebox(0,0)[lb]{\smash{$p_i$}}}%
    \put(0.02320776,0.00506269){\color[rgb]{0,0,0}\makebox(0,0)[lb]{\smash{$q_1$}}}%
    \put(0.54956612,0.61361913){\color[rgb]{0,0,0}\makebox(0,0)[lb]{\smash{$q_i$}}}%
    \put(0.84672565,0.08726906){\color[rgb]{0,0,0}\makebox(0,0)[lb]{\smash{$p_{i+1}$}}}%
    \put(0.16728291,0.00963749){\color[rgb]{0,0,0}\makebox(0,0)[lb]{\smash{$\gamma_i$}}}%
    \put(0.11368484,0.09543721){\color[rgb]{0,0,0}\makebox(0,0)[lb]{\smash{$\omega_i$}}}%
    \put(-0.05974289,0.14091224){\color[rgb]{0,0,0}\makebox(0,0)[lb]{\smash{$\gamma_{i-1}$}}}%
    \put(0.66525785,0.44374926){\color[rgb]{0,0,0}\makebox(0,0)[lb]{\smash{$b_i$}}}%
    \put(0.2116232,0.64750474){\color[rgb]{0,0,0}\makebox(0,0)[lb]{\smash{$a_i$}}}%
    \put(0.41118907,0.45453471){\color[rgb]{0,1,0}\makebox(0,0)[lb]{\smash{$c_i$}}}%
    \put(0.04151877,0.33724371){\color[rgb]{0,0,0}\makebox(0,0)[lb]{\smash{$y$}}}%
    \put(0.28551184,0.30482705){\color[rgb]{0,0,0}\makebox(0,0)[lb]{\smash{$y'$}}}%
    \put(0,0){\includegraphics[width=\unitlength,page=3]{dbltrianglediagramCase1.pdf}}%
    \put(0.16270121,0.25691168){\color[rgb]{0,0,0}\makebox(0,0)[lb]{\smash{$\sigma$}}}%
  \end{picture}%
\endgroup%

%% file: complexWithHyperplanes.pdf_tex
\begingroup%
  \makeatletter%
  \providecommand\color[2][]{%
    \errmessage{(Inkscape) Color is used for the text in Inkscape, but the package 'color.sty' is not loaded}%
    \renewcommand\color[2][]{}%
  }%
  \providecommand\transparent[1]{%
    \errmessage{(Inkscape) Transparency is used (non-zero) for the text in Inkscape, but the package 'transparent.sty' is not loaded}%
    \renewcommand\transparent[1]{}%
  }%
  \providecommand\rotatebox[2]{#2}%
  \ifx\svgwidth\undefined%
    \setlength{\unitlength}{414.30096484bp}%
    \ifx\svgscale\undefined%
      \relax%
    \else%
      \setlength{\unitlength}{\unitlength * \real{\svgscale}}%
    \fi%
  \else%
    \setlength{\unitlength}{\svgwidth}%
  \fi%
  \global\let\svgwidth\undefined%
  \global\let\svgscale\undefined%
  \makeatother%
  \begin{picture}(1,0.39837614)%
    \put(0,0){\includegraphics[width=\unitlength,page=1]{complexWithHyperplanes.pdf}}%
  \end{picture}%
\endgroup%

%% file: doubledotfig8.pdf_tex
\begingroup%
  \makeatletter%
  \providecommand\color[2][]{%
    \errmessage{(Inkscape) Color is used for the text in Inkscape, but the package 'color.sty' is not loaded}%
    \renewcommand\color[2][]{}%
  }%
  \providecommand\transparent[1]{%
    \errmessage{(Inkscape) Transparency is used (non-zero) for the text in Inkscape, but the package 'transparent.sty' is not loaded}%
    \renewcommand\transparent[1]{}%
  }%
  \providecommand\rotatebox[2]{#2}%
  \ifx\svgwidth\undefined%
    \setlength{\unitlength}{229.36764935bp}%
    \ifx\svgscale\undefined%
      \relax%
    \else%
      \setlength{\unitlength}{\unitlength * \real{\svgscale}}%
    \fi%
  \else%
    \setlength{\unitlength}{\svgwidth}%
  \fi%
  \global\let\svgwidth\undefined%
  \global\let\svgscale\undefined%
  \makeatother%
  \begin{picture}(1,0.37712765)%
    \put(0,0){\includegraphics[width=\unitlength,page=1]{doubledotfig8.pdf}}%
    \put(0.1233079,0.07777009){\color[rgb]{0,0,0}\makebox(0,0)[lb]{\smash{$a$}}}%
    \put(0.40233591,0.3219196){\color[rgb]{0,0,0}\makebox(0,0)[lb]{\smash{$b$}}}%
    \put(0,0){\includegraphics[width=\unitlength,page=2]{doubledotfig8.pdf}}%
    \put(0.58171102,0.20981015){\color[rgb]{0,0,0}\makebox(0,0)[lb]{\smash{$a$}}}%
    \put(0,0){\includegraphics[width=\unitlength,page=3]{doubledotfig8.pdf}}%
    \put(0.76108623,0.10517459){\color[rgb]{0,0,0}\makebox(0,0)[lb]{\smash{$b$}}}%
    \put(0.67638133,0.20482749){\color[rgb]{0,0,0}\makebox(0,0)[lb]{\smash{$a$}}}%
    \put(0,0){\includegraphics[width=\unitlength,page=4]{doubledotfig8.pdf}}%
    \put(0.73061182,0.01669599){\color[rgb]{0,0,0}\makebox(0,0)[lb]{\smash{$b$}}}%
    \put(0,0){\includegraphics[width=\unitlength,page=5]{doubledotfig8.pdf}}%
    \put(0.80592993,0.20856447){\color[rgb]{0,0,0}\makebox(0,0)[lb]{\smash{$a$}}}%
    \put(0,0){\includegraphics[width=\unitlength,page=6]{doubledotfig8.pdf}}%
    \put(0.90558284,0.19859917){\color[rgb]{0,0,0}\makebox(0,0)[lb]{\smash{$a$}}}%
    \put(0,0){\includegraphics[width=\unitlength,page=7]{doubledotfig8.pdf}}%
    \put(0.74364698,0.3343762){\color[rgb]{0,0,0}\makebox(0,0)[lb]{\smash{$b$}}}%
    \put(0.74057713,0.24091495){\color[rgb]{0,0,0}\makebox(0,0)[lb]{\smash{$b$}}}%
    \put(0,0){\includegraphics[width=\unitlength,page=8]{doubledotfig8.pdf}}%
  \end{picture}%
\endgroup%

%% file: ElevationFigure.pdf_tex
\begingroup%
  \makeatletter%
  \providecommand\color[2][]{%
    \errmessage{(Inkscape) Color is used for the text in Inkscape, but the package 'color.sty' is not loaded}%
    \renewcommand\color[2][]{}%
  }%
  \providecommand\transparent[1]{%
    \errmessage{(Inkscape) Transparency is used (non-zero) for the text in Inkscape, but the package 'transparent.sty' is not loaded}%
    \renewcommand\transparent[1]{}%
  }%
  \providecommand\rotatebox[2]{#2}%
  \ifx\svgwidth\undefined%
    \setlength{\unitlength}{191.3561435bp}%
    \ifx\svgscale\undefined%
      \relax%
    \else%
      \setlength{\unitlength}{\unitlength * \real{\svgscale}}%
    \fi%
  \else%
    \setlength{\unitlength}{\svgwidth}%
  \fi%
  \global\let\svgwidth\undefined%
  \global\let\svgscale\undefined%
  \makeatother%
  \begin{picture}(1,0.94427233)%
    \put(0,0){\includegraphics[width=\unitlength,page=1]{ElevationFigure.pdf}}%
    \put(0.63082334,0.47605568){\color[rgb]{0,0,0}\makebox(0,0)[lb]{\smash{$\tilde{X}\ldots$}}}%
    \put(0,0){\includegraphics[width=\unitlength,page=2]{ElevationFigure.pdf}}%
    \put(0.25456167,0.12918934){\color[rgb]{0,0,1}\makebox(0,0)[lb]{\smash{$B_E'$}}}%
    \put(0,0){\includegraphics[width=\unitlength,page=3]{ElevationFigure.pdf}}%
    \put(-0.00395512,0.42502197){\color[rgb]{1,0,0}\makebox(0,0)[lb]{\smash{$B_E$}}}%
  \end{picture}%
\endgroup%